\documentclass[10pt, letterpaper]{amsart}
\date{\today}
\usepackage{amsthm}
\usepackage{amsmath}
\usepackage{amssymb}

\usepackage{mathtools}
\mathtoolsset{showonlyrefs}

\usepackage{listings}

\usepackage{hyperref}
\usepackage{graphicx}
\usepackage{caption}
\usepackage{subcaption}
\usepackage{morefloats}
\usepackage{float}
\usepackage{xcolor}
\usepackage{perpage}
\MakeSorted{figure}
\MakeSorted{table}
\usepackage{color}
\usepackage{multirow}
\usepackage{url}
\usepackage{bbm}
\usepackage{rsfso}
\usepackage{enumerate}
\usepackage{enumitem}
\usepackage{comment}
\usepackage{mathrsfs}
\usepackage{dsfont}
\usepackage{algorithm}
\usepackage{algpseudocode}
\usepackage{placeins}

\usepackage{tikz}
\usepackage{pgfplots}
\pgfplotsset{compat=1.18}

\newif\ifwidedisplays
\widedisplaysfalse

\DeclareMathOperator{\Span}{span}

\DeclareMathOperator{\Tr}{Tr}

\def\bR{\mathbb{R}}

\def\R{{\mathbb{R}}}

\newtheorem{theorem}{Theorem}
\newtheorem{proposition}[theorem]{Proposition}

\newtheorem{lem}[theorem]{Lemma}

\newtheorem{remark}[theorem]{Remark}

\newtheorem{defn}[theorem]{Definition}

\newtheorem{corollary}[theorem]{Corollary}

\newtheorem{assumption}[theorem]{Assumption}

\def\min{{\rm min}}

\usetikzlibrary{positioning}
\definecolor{uclablue}{RGB}{39,116,174}
\definecolor{uclamuted}{RGB}{90,90,90}
\definecolor{signalgreen}{RGB}{50,130,90}
\definecolor{darkgreen}{RGB}{0,100,0}
\definecolor{warnorange}{RGB}{210,120,40}
\definecolor{softred}{RGB}{170,60,60}
\definecolor{darkpurple}{RGB}{125,30,150}
\definecolor{darkorange}{RGB}{204,102,0}

\makeatletter
\newcommand{\assumlabel}[2]{\protected@edef\@currentlabel{#2}\label{#1}}
\makeatother

\begin{document}

\title[A Sharp SNR Threshold for QML Breakpoint Estimation]{A Sharp Signal-to-Noise Threshold for Quasi-Maximum Likelihood Breakpoint Estimation}

\author{Hubeyb Gurdogan}
\author{Georg Menz}

\subjclass[2020]{Primary 62M10; Secondary 62H12, 62H25, 62F12, 62R07}

\keywords{Structural breaks, change-point detection, breakpoint estimation, quasi-maximum likelihood, signal-to-noise ratio, covariance change, high-dimensional factor models, regularized log-determinant objective}

\begin{abstract}
We establish a sharp pathwise signal-to-noise criterion for quasi-maximum-likelihood (QML) estimation of a dominant breakpoint in the second-moment structure of a multivariate time series: the QML estimator is consistent whenever the between-regime contrast exceeds the within-regime fluctuation by an explicit factor, and below this threshold global recovery can fail. Two innovations drive the result. First, the framework is pathwise: no stochastic model is imposed on the data. Second, the log-determinant objective carries an additive ridge regularization: it removes the endpoint boundary layers, so no trimming of the candidate set is required, and tuning the ridge weakens the consistency condition. As an application of the main theorem, we establish consistency of QML breakpoint estimation in pervasive factor models whose dimension diverges with the sample size, for completely general error terms --- not necessarily independent, idiosyncratic, or even random.
\end{abstract}

\maketitle

\setcounter{tocdepth}{1}
\tableofcontents

\section{Introduction}\label{sec:intro}

Many real-world systems exhibit abrupt structural shifts. Financial crises reshape cross-sectional correlations of asset returns; central-bank policy regime changes alter macroeconomic dynamics; climate regimes reorganize correlated environmental variables; experimental conditions in neuroscience modify neural population dynamics. Knowing the moment of the regime change is essential for forecasting, risk management, model adaptation, and interpretation of post-shift behavior. \\

Mathematically, a sudden structural shift means that the underlying data time series $X=(x_t)_{1\le t\le n}\in\R^{p\times n}$ is generated by two different models: one before the breakpoint $\tau=\lfloor \theta n\rfloor$, and one after~$\tau$ (see Figure~\ref{fig:intro_qml_demo}). The central inferential task is then to estimate~$\tau$, or at least the relative break time $\theta\in(0,1)$, and to prove consistency of the estimator as the sample size~$n$ grows.  The quasi-maximum-likelihood (QML) breakpoint estimator $\hat\theta_n$ minimizes the regularized log-likelihood objective
\begin{align}\label{eq:intro_J_n}
J_n(\nu) \;=\; \nu\log\det\bigl(\widehat\Sigma_{0:\nu}+\varepsilon I_q\bigr) + (1-\nu)\log\det\bigl(\widehat\Sigma_{\nu:1}+\varepsilon I_q\bigr)
\end{align}
over candidate fractions $\nu\in(0,1)$. Here, $\widehat\Sigma_{a:b}$ is the empirical second moment of a $q$-dimensional signal series $y_t$ on the macroscopic block $\{\lfloor an\rfloor+1,\dots,\lfloor bn\rfloor\}$. The constant $\varepsilon\ge 0$ is a regularization parameter. Figure~\ref{fig:intro_qml_demo} illustrates the idea on a simple series with a single breakpoint: $J_n(\nu)$ attains its global minimum at the grid point corresponding to the break in this example.\\

\begin{figure}[h]
    \centering
    \includegraphics[width=0.8\linewidth]{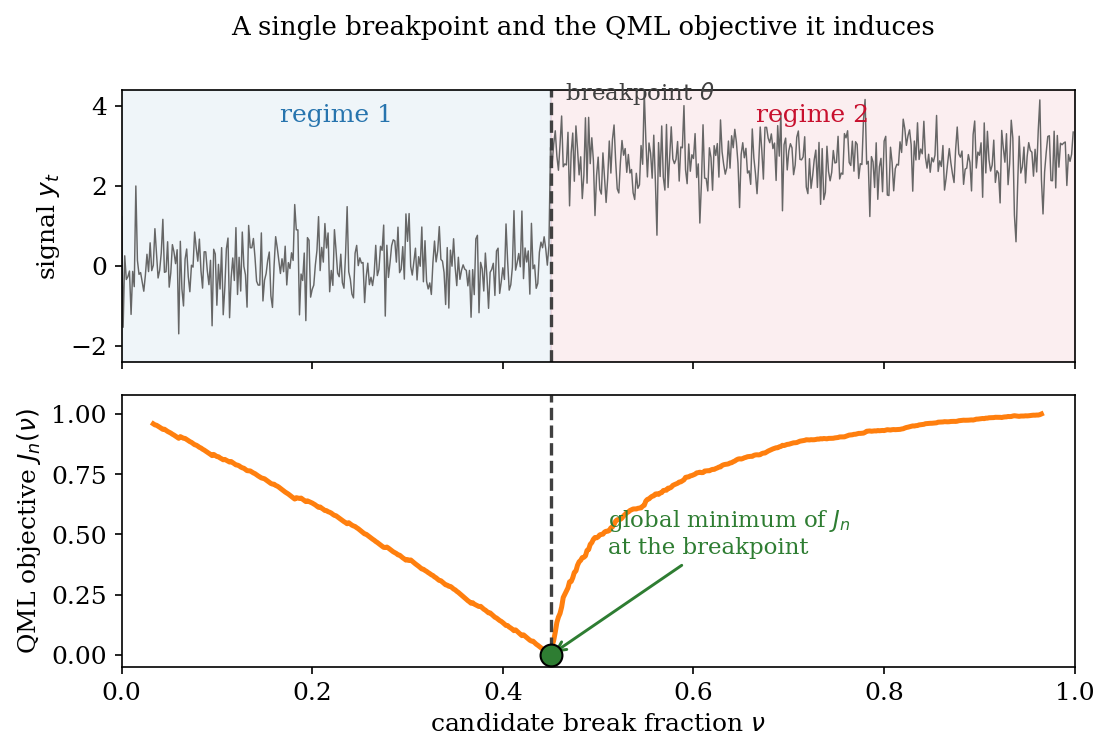}
    \caption{Example of a time series with a single breakpoint (top): regime~1 (blue) gives way to regime~2 (red) at the breakpoint  $\tau=\lfloor \theta n\rfloor$ for $\theta=0.45$. The regularized QML objective $J_n(\nu)$ it induces (bottom) attains its global minimum at the grid point corresponding to the breakpoint in this example.}
    \label{fig:intro_qml_demo}
\end{figure}

Classical change-point theory imposes distributional or independence structure on the data --- from CUSUM/MOSUM~\cite{Page1954,BassevilleNikiforov1993,chu1995mosum} to nonparametric feature-map / kernel-discrepancy methods~\cite{GrettonEtAl2012,MattesonJames2014}: consistency is typically proved under structural assumptions --- independence or weak dependence, moment conditions, often homogeneity --- strong enough to force a (quantitative) law of large numbers on the relevant empirical statistics. In many real-world settings these assumptions are unrealistic. For example, considering financial time series: returns are not independent and are heavy-tailed, and insider trading makes the idiosyncratic component correlated with the history and non-Markov. In general physical systems, the measurement itself can introduce systemic and persistent artifacts that are hard to detect and to remove. Overall, demanding that statistics satisfy a SLLN is too strong. For breakpoint detection it should only matter whether there is enough evidence of a regime change, regardless of the microstructure of the signal, such as whether the noise is independent, idiosyncratic, or even random. Estimating the breakpoint should be possible as soon as the \emph{signal} overcomes the \emph{noise}.

This intuition has been formalized in a substantial line of work on model-free and nonparametric change-point detection. The monographs of Brodsky and Darkhovsky~\cite{BrodskyDarkhovsky1993} and Cs\H{o}rg\H{o} and Horv\'ath~\cite{CsorgoHorvath1997} treat consistency under jump-size or moment conditions rather than full distributional models; in the econometric literature, Bai and Perron~\cite{BaiPerron1998} and Lavielle and Moulines~\cite{LavielleMoulines2000} prove consistency of least-squares multi-break estimators under jump-size hypotheses and weak-dependence assumptions on the noise.
More recently, consistency and optimal detection have been established under explicit lower bounds on the size and spacing of the changes, in works such as Frick--Munk--Sieling~\cite{FrickMunkSieling2014} (multiscale, SMUCE), the kernel-based methods of Garreau--Arlot~\cite{GarreauArlot2018} and Arlot--Celisse--Harchaoui~\cite{ArlotCelisseHarchaoui2019}, the sparse-projection approach of Wang and Samworth~\cite{WangSamworth2018}, and the minimax-optimality theory of Verzelen et al.~\cite{VerzelenFromontLerasleReynaudBouret2023}; see also the survey~\cite{TruongOudreVayatis2020}. Figure~\ref{fig:intro_inhomogeneous_lln} shows a data set of the kind we have in mind: block averages converge locally, but no  regimewise homogeneous law of large numbers holds.\\

\begin{figure}[htp]
    \centering
    \includegraphics[width=0.92\linewidth]{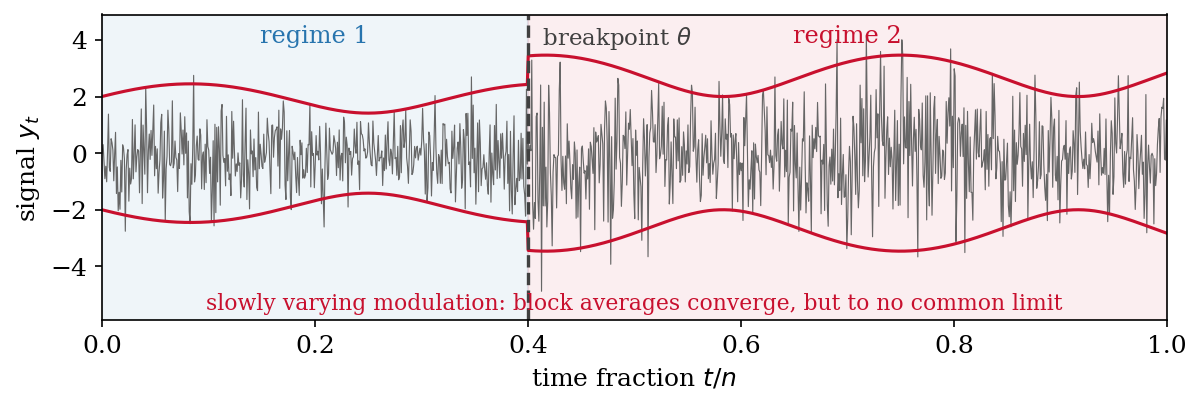}
    \caption{A time series with a single breakpoint at $\theta=0.4$ (base variance $1\to2$; red: $\pm2\sigma$ envelope) under a slowly varying modulation: every macroscopic block average converges, but to limits that vary along the sample --- no regimewise homogeneous law of large numbers holds. The pathwise framework of Section~\ref{sec:abstract_QML} covers such data.}
    \label{fig:intro_inhomogeneous_lln}
\end{figure}

We bring this viewpoint to the QML criterion --- in a model-free setting --- identifying the sharp signal-to-noise threshold separating consistency from inconsistency. We measure the strength of the signal via
\begin{align}
\Delta_{\text{between}} & = \inf_{I_1\subset[0,\theta],\,I_2\subset[\theta,1]}\liminf_{n\to\infty}\bigl\|\widehat\Sigma_{I_1}-\widehat\Sigma_{I_2}\bigr\|_F.
\end{align}
It measures the minimal between-regime contrast. The strength of the noise is measured via
\begin{align}\label{equ:SNR_introduction}
  \Delta_{\text{within}} & = \sup_{\text{same regime}}\limsup_{n\to\infty}\bigl\|\widehat\Sigma_{I_1}-\widehat\Sigma_{I_2}\bigr\|_F.
\end{align}
It measures the maximal within-regime fluctuation. In the main result of this paper, see Theorem~\ref{thm:QML-robust} in Section~\ref{sec:abstract_QML} below, we show that the QML estimator $\hat\theta_n$ is consistent ($\hat\theta_n\to\theta$ as $n\to\infty$) whenever the signal-to-noise ratio (SNR) exceeds
\begin{align}\label{eq:intro_SNR_threshold}
\frac{\Delta_{\text{between}}}{\Delta_{\text{within}}} \;>\; \frac{1}{2\sqrt{\theta(1-\theta)}}\cdot\frac{M+\varepsilon}{m+\varepsilon},
\end{align}
where $M$ and $m$ are asymptotic upper and lower spectral bounds for the empirical block second moments involved; see~\eqref{equ:def_M_m} below for the precise definition. No SLLN, no independence, no idiosyncratic structure, no stochastic model whatsoever is required of the data.
The classical SLLN-based consistency result (see Theorem~\ref{thm:classical_QML_consistency} below) is recovered as the SLLN forces $\Delta_{\text{within}}=0$ corresponding to an infinite SNR. We also show that weakening the SNR by removing the first factor on the right-hand side of~\eqref{eq:intro_SNR_threshold} still forces a local minimum of the QML objective function at the breakpoint. The proof of Theorem~\ref{thm:classical_QML_consistency} rests on a quantitative version of the strict concavity of $\log\det$ --- a Jensen-gap inequality with explicit constants (see Lemma~\ref{lem:jensen_gap_logdet} below).
\\

  An additional structural novelty is the additive ridge regularization $\varepsilon I_q$ inside each $\log\det$. It entails three important consequences.  First, the ridge keeps the objective well defined when a regime second moment is singular, e.g.\ when a factor disappears or emerges at the break (see Remark~\ref{rem:epsilon_regularization} below); without regularization, positive definiteness of both regime second moments is needed, see Assumption~(PD).

Second, the ridge controls the threshold~\eqref{eq:intro_SNR_threshold} itself. The threshold is the product of two factors: the geometric factor $\{2\sqrt{\theta(1-\theta)}\}^{-1}$, which only reflects the position~$\theta$ of the break, and the conditioning factor $(M+\varepsilon)/(m+\varepsilon)$, which increasing~$\varepsilon$ drives to~$1$ (see Remark~\ref{rem:ridge_regularization_and_SNR} below for the explicit choice). The weakest condition obtainable by tuning the ridge is therefore that the signal-to-noise ratio exceeds the geometric factor alone, and Proposition~\ref{prop:sharp-A4-fixed-theta} shows that this residual threshold is sharp: below it, global recovery can fail. Third, and specific to the breakpoint problem, the ridge removes the endpoint boundary layers (see Figure~\ref{fig:boundary_ridge}). For $\varepsilon>0$ the QML criterion can be minimized over the full grid $\mathcal D_n$ rather than over a fixed trimmed subset (see Remark~\ref{rem:untrimmed_minimization} and Theorems~\ref{thm:classical_QML_consistency} and~\ref{thm:QML-robust} below).\\

\begin{figure}[htp]
    \centering
    \includegraphics[width=0.92\linewidth]{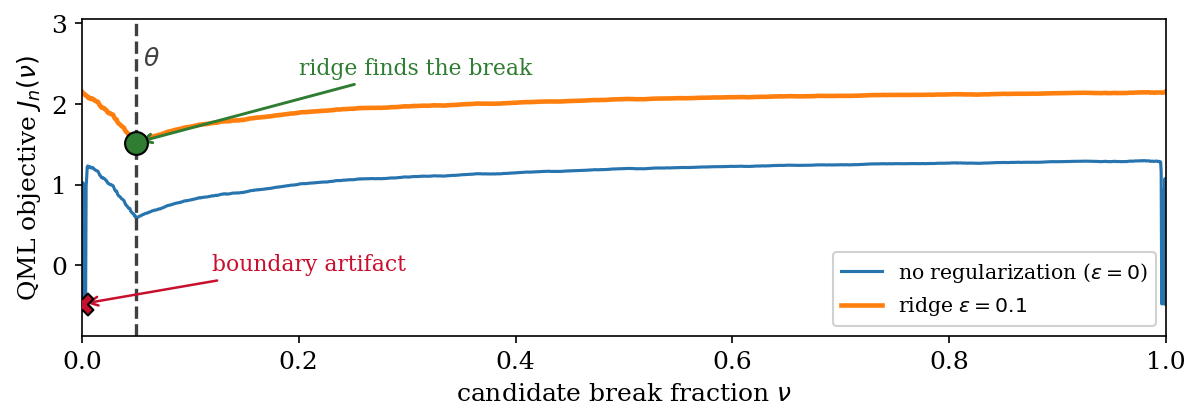}
    \caption{A break inside the endpoint boundary layer ($\theta=0.05$, dashed): without regularization the QML objective dives at the endpoints and its global minimizer is a boundary artifact; with the ridge it attains its minimum at the break.}
    \label{fig:boundary_ridge}
\end{figure}

In modern applications, data are often high-dimensional (large cross-sections, panels, portfolios), where classical low-dimensional change-point methods either break down or require additional structure. The main reason is that generically eigenvalues and eigenvectors are not consistent in the regime~$\frac{p}{n}\to \gamma >0$. For such data it is common that the strongly correlated observed variables have much of their variation driven by a small number of latent factors, which is captured by \emph{factor models}. Without claiming completeness, let us list several examples. In finance, factor models are standard tools for portfolio construction, risk management, and performance attribution (see e.g.~\cite{Ang2014AssetManagement}). In macroeconomics, dynamic factor models are widely used for nowcasting and forecasting from large panels of economic indicators, including real-time monitoring of GDP and business-cycle conditions by central banks (see e.g.~\cite{StockWatson2011DynamicFactorModels} and also~\cite{BaiWang2016LargeFactorModels}). In neuroscience, factor models extract low-dimensional latent structure from large neural population recordings (cf.~e.g.~\cite{CunninghamYu2014DimensionalityReduction}), and in climate science they summarize and forecast large-scale correlated environmental variability (cf.~e.g.~\cite{LiKoopmanLitPetrova2020ElNino}).\\

Breakpoint estimation in high-dimensional factor models was studied by Bai, Han, and Shi~\cite{BaiHanShi2020} using least-squares/PCA-based methods; the closest work to ours is the QML/log-determinant approach of Duan, Bai, and Han~\cite{DuanBaiHan2023}, without ridge regularization, building on the factor-model/PCA theory of Bai~\cite{Bai2003} and Bai--Ng~\cite{BaiNg2002}. Under structural assumptions on the noise --- weak cross-sectional and temporal dependence, moment conditions, often homogeneity --- strong enough to force a quantitative law of large numbers on the empirical second moments, \cite{DuanBaiHan2023} locate the break up to an error of size~$O(1)$.
We therefore apply in Section~\ref{sec:breakpoint_detection_in_factor_models} the abstract framework of Section~\ref{sec:abstract_QML} to a \emph{pervasive} factor model $x_t=\Lambda_\bullet f_t+e_t$ --- the subscript $\bullet\in\{\leq,>\}$ is a placeholder for the regime, with loadings $\Lambda_\leq$ before and $\Lambda_>$ after the break --- in the high-dimensional high-sample-size (HDHSS) regime, which means that the dimension $p=p_n\to\infty$ diverges with the sample size~$n$. Pervasive means that the leading
eigenvalues of~$\Lambda_\bullet \Lambda_\bullet^{\top}$ scale like the underlying dimension~$p$; this is a natural assumption, as factors like the market factor or industrial sectors usually affect a macroscopic proportion of the assets (see Remark~\ref{rem:pervasive_scale} below). We recover and extend the QML breakpoint result of Duan, Bai, and Han~\cite{DuanBaiHan2023} by allowing general error terms~$e_t$ whose blockwise spectral contribution is negligible on the pervasive scale, i.e.~of order~$o(p)$. The error term is not necessarily independent, uncorrelated, idiosyncratic, or even random --- and may carry any non-pervasive, lower-order structure (weak factors, idiosyncratic noise, dependence). In Corollary~\ref{thm:qml_consistency_factor} below, we then show consistency of the QML estimator for this factor model via a combination of Theorem~\ref{thm:QML-robust} and the Davis--Kahan $\sin\theta$ theorem. The factor-model approach carries a second, practical advantage: projecting onto the leading left singular basis of the data --- the leading principal components --- averages out idiosyncratic, name-specific fluctuation and thereby smooths the QML objective. Figure~\ref{fig:gfc_smoothness} displays this on daily U.S.~equity returns spanning 2005--2010: the projected objective is visibly smoother than the SPY-only objective.
Breakpoint estimation in the high-dimensional low-sample-size (HDLSS) regime, where $p\to\infty$ while $n$ remains fixed, remains an open direction; the inconsistency of sample PCA eigenvectors is a central obstacle to extending the present QML theory.\\

\begin{figure}[htp]
    \centering
    \includegraphics[width=0.72\linewidth]{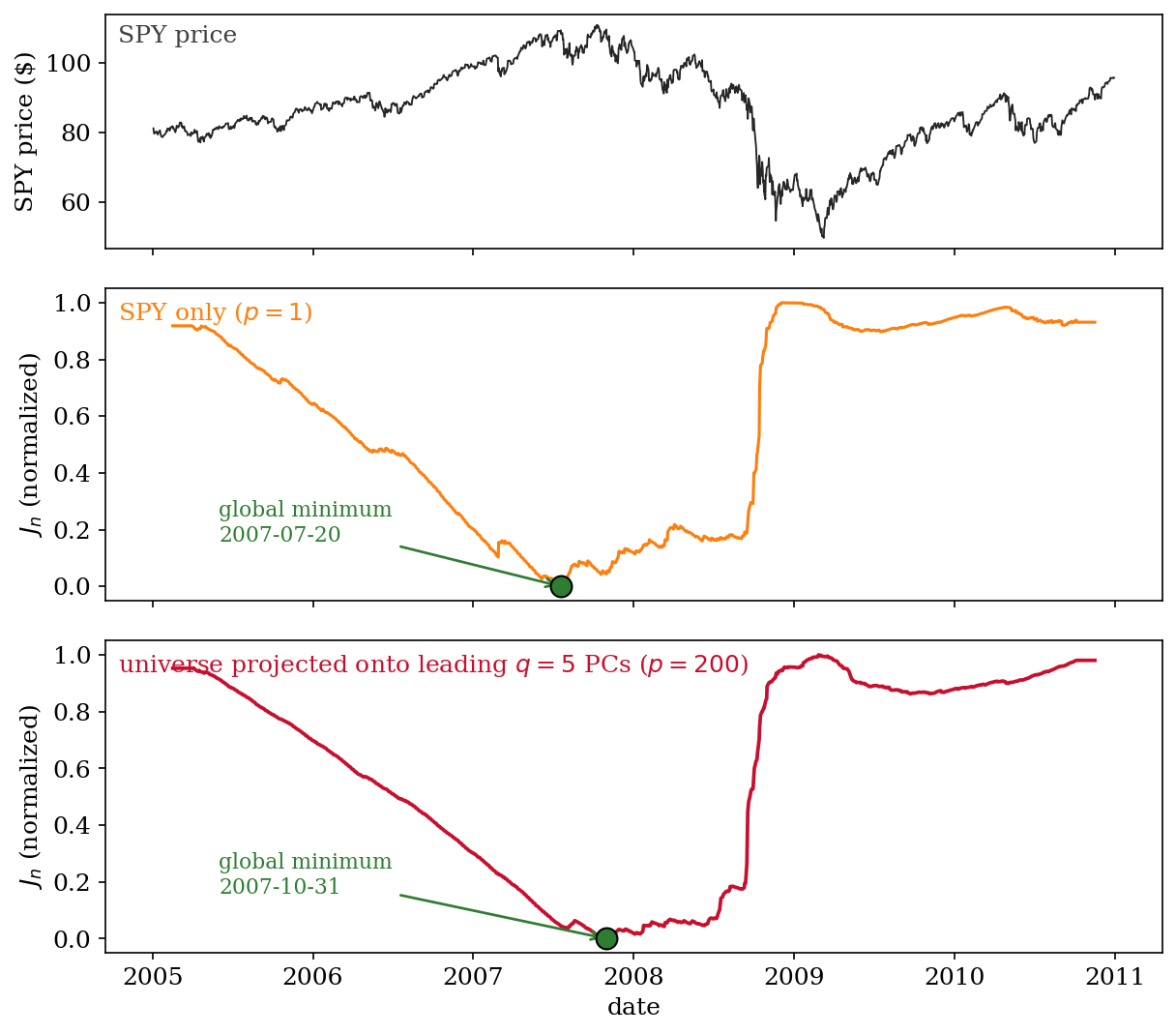}
    \caption{Smoothing by projection on daily U.S.\ equity returns, 2005--2010. Top: SPY price. Middle and bottom: normalized QML objectives $J_n(\nu)$ on the same sample, for SPY alone ($p=1$) and for the large-cap universe ($p=200$ names) projected onto its leading $q=5$ principal components ($\varepsilon=10^{-8}$). The projected objective is visibly smoother.}
    \label{fig:gfc_smoothness}
\end{figure}

The theory of this paper applies to data carrying  one dominant breakpoint: a location at which the between-regime contrast exceeds the within-regime fluctuation (Definition~\ref{def:pathwise_breakpoint}); our theorems require the excess to reach the threshold factor of~\eqref{eq:intro_SNR_threshold}. A dominant breakpoint is unique whenever it exists (Lemma~\ref{lem:unique_breakpoint}), and every minimizer of~$J_n$ converges to it (Theorem~\ref{thm:QML-robust}). In this sense, being a global minimizer of~$J_n$ is a necessary condition for the dominant breakpoint, but not a sufficient one. If no location meets the threshold, the theory makes no statement, and a global minimizer of~$J_n$ may then be produced by a transient anomaly, a corrupted stretch of data, or several breaks of comparable size; and even a genuine breakpoint whose contrast falls short of the threshold can be missed (Proposition~\ref{prop:sharp-A4-fixed-theta}). A minimizer of~$J_n$ should therefore be read as a candidate breakpoint whose signal-to-noise ratio remains to be checked.

The paper is organized as follows. Section~\ref{sec:abstract_QML} develops the abstract, model-free framework and establishes the sharp signal-to-noise consistency threshold for the regularized QML estimator, together with its ridge consequences and the sharpness construction. Section~\ref{sec:breakpoint_detection_in_factor_models} specializes the framework to high-dimensional pervasive factor models. Section~\ref{sec:proofs} contains the proofs, and Appendix~\ref{sec:code} documents the code generating the figures.

\section*{Conventions and notation}

\begin{itemize}
    \item For a finite-dimensional vector~$x$, we denote with~$|x|$ the standard Euclidean norm of~$x$.
    \item For a matrix~$M$, we denote with~$\| M \|$ the operator norm of~$M=(M_{ij})$ with respect to the standard Euclidean norm and with~$\| M \|_{F}$ its Frobenius norm, i.e.
    \begin{align*}
        \|M \|:= \sup_{|x|=1} | M x| \quad \mbox{and} \quad \|M \|_{F} := \sqrt{\sum_{i,j} |M_{ij}|^2}
    \end{align*}
    \item For a symmetric matrix~$M$, $\lambda_{\min}(M)$ and~$\lambda_{\max}(M)$ denote its smallest and largest eigenvalue.
    \item For symmetric matrices we write~$A \succeq 0$ if~$A$ is positive semi-definite (all eigenvalues are non-negative) and~$A \succ 0$ if~$A$ is (strictly) positive definite (all eigenvalues are strictly positive).
    \item A hat denotes an empirical (sample) quantity or an estimator, e.g.~$\widehat \Sigma_{a:b}$,~$\widehat \theta_n$, and~$\widehat U$.
    \item  For a location~$\xi\in(0,1)$ we write~$k_\xi:=\lfloor \xi n\rfloor$ for its discretization, so that the breakpoint is~$\tau=k_\theta$; for a sequence~$(\nu_n)$ we abbreviate~$k_n:=\lfloor \nu_n n\rfloor$.
    \item A subscript or superscript~$\bullet$ is a placeholder for a regime index, $\bullet\in\{\le,>\}$, where~$\le$ refers to the regime before the breakpoint~$\tau$ and~$>$ to the regime after it; e.g.~$\Lambda_\bullet$,~$\Sigma_{F_\bullet}$, and~$\Sigma_\bullet$ denote the corresponding pre-/post-break objects.
    \item We use~$c$ and~$C$ to denote generic positive constants, independent of~$n$ and~$p$, whose value may change from line to line.
    \item For a finite set, e.g.~a macroscopic index block~$I$, we denote with~$|\cdot|$ its cardinality; whether~$|\cdot|$ means the Euclidean norm of a vector or the cardinality of a set will always be clear from the context.
\end{itemize}

\section{Abstract framework for QML-Breakpoint detection}\label{sec:abstract_QML}

The data is given by a $p$-dimensional time series $X=(x_t)_{1\le t\le n}\in\mathbb R^{p\times n}$. We work in a high-sample-size asymptotic regime in which the number of observations~$n$ increases. The analysis of this section is carried out entirely on a fixed, $q$-dimensional signal series (see Definition~\ref{def:signal_separation_space} below) and imposes \emph{no} restriction on the ambient dimension~$p$, which may be held fixed or diverge with~$n$; only the signal dimension~$q$ is fixed. This generality is what lets us instantiate the same framework in the high-dimensional regime~$p=p_n\to\infty$ in Section~\ref{sec:breakpoint_detection_in_factor_models}. We assume that the data suddenly switches regimes at time~$\tau:= \lfloor \theta n \rfloor \in \mathbb{N}$ for some fixed~$\theta \in (0,1)$. The time~$\tau$ is called the breakpoint of the data.  A model-free definition of a \emph{dominant} breakpoint, formulated in terms of the signal series alone, is given in Definition~\ref{def:pathwise_breakpoint} below; the data may contain several breakpoints, but by Lemma~\ref{lem:unique_breakpoint} at most one of them is dominant. The main purpose of this section is to show that a quasi-maximum-likelihood (QML) estimator is able to detect  the dominant breakpoint provided that the data contains a signal that is sufficiently distinct between the regimes.\\

As we work with asymptotics, i.e.~sending the number of observations~$n$ to infinity, let us be more precise about how the data is generated. 
\begin{assumption}\label{ass:generating_data}
    Let~$\theta \in (0, 1)$. We assume that there is a latent two-sided $p$-dimensional sequence $Z= (z_{s})_{s \in \mathbb{Z}}$, fixed and independent of~$n$, which is in the first regime for~$s \leq 0$ and in the second regime for~$s \geq 1$. For a sample of size~$n$, the observed data $X= (x_t)_{1 \leq t \leq n} \in \mathbb R^{p\times n}$ is the length-$n$ window of~$Z$ that straddles the origin at the fraction~$\theta$,
    \begin{align}
        x_t := z_{\,t - \lfloor \theta n \rfloor}, \qquad 1 \leq t \leq n,
    \end{align}
    i.e.~the window~$-\lfloor\theta n\rfloor+1 \leq s \leq \lceil(1-\theta)n\rceil$ re-indexed to start at~$t=1$.
\end{assumption}
In practical terms, Assumption~\ref{ass:generating_data} means that as the sample size~$n$ grows, the observation window of the fixed latent sequence~$Z$ expands on both sides of the break at the origin, while the break itself is held at the fixed sample fraction~$\theta$. This construction has several advantages:
\begin{itemize}
    \item The breakpoint~$\tau_n = \lfloor \theta n\rfloor$ is kept at a fixed fraction~$\theta \in (0,1)$ of the sample.
    \item It guarantees that both regimes have a sample size growing proportional to~$n$.
    \item It provides a clean coupling across different~$n$, simplifying the asymptotic analysis of estimators.
\end{itemize}
This coupling is only a convenient canonical construction: all results below use only the stated blockwise empirical limits and therefore apply unchanged to triangular arrays satisfying the assumptions of the result in question.

 Instead of considering the full $p$-dimensional time series~$(x_t)$, we only consider a $q$-dimensional signal time series~$(y_t)$, carefully chosen to separate the regimes. Choosing the right signal is important as it allows for filtering noise and amplifying the regime separation. The signal time series~$(y_t)$ is the only object on which the QML estimator below acts.

\begin{defn}[Signal time series and signal separation space] \label{def:signal_separation_space}
    The \emph{signal time series}~$(y_t)$ is a $q$-dimensional sequence~$y_t\in\mathbb R^q$, $1\le t\le n$, obtained from the data by a map~$x_t\mapsto y_t$. This map can be nonlinear but the canonical construction is linear. Assuming that the separating signal lives in a $q$-dimensional subspace of~$\mathbb{R}^p$, we fix an orthonormal basis $u_i\in\mathbb R^{p}$, $1\le i\le q$, of this space, and let $U_q=(u_1,\dots,u_q)\in\mathbb R^{p\times q}$ denote the
associated matrix. This subspace~$\Span(U_q)$ is called the \emph{signal separation space}, and the signal time series records the coordinates of the orthogonal projection of~$x_t$ onto it in the basis~$U_q$, i.e.
    \begin{align}
   y_t := U_q^{\top} x_t \in \mathbb{R}^q, \qquad 1\le t\le n.
\end{align}
Working with the coordinate vector~$y_t\in\mathbb R^q$ rather than the projection in~$\mathbb R^p$ is what reduces the ambient dimension to the fixed number~$q$. The signal-separation space might depend on the data~$X$ and therefore also on~$n$ or~$p$.
\end{defn}

The choice of the signal time series --- equivalently, of the map~$x_t\mapsto y_t$ that produces it --- is crucial for the success of breakpoint detection, and a priori knowledge is valuable. Many different constructions of the signal have been proposed, ranging from contrast-based methods like CUSUM (e.g.~\cite{Page1954,BassevilleNikiforov1993}) and MOSUM (e.g.~\cite{chu1995mosum,CsorgoHorvath1997}) to global dimension reduction via PCA/SVD factors (e.g.~\cite{BaiHanShi2020,DuanBaiHan2023}). Other popular alternatives include data-driven projection directions \cite{WangSamworth2018} and nonparametric feature-map discrepancies (kernel/energy-type signals) \cite{GrettonEtAl2012,MattesonJames2014}. The framework also accommodates feature engineering by absorbing it into the construction of~$x_t$. The raw inputs can be augmented with engineered features (e.g.~quadratic interactions, lagged variables, kernel evaluations) designed to amplify the between-regime contrast. We assume that the data time series~$x_t$ already contains those augmented features. We want to point out one trivial case: choosing~$y_t= x_t$, i.e.~the identity map, is a valid choice for defining the signal time series.   

\begin{remark}[PCA/SVD signals and their limitations]\label{rem:data_driven_signal_space}
     If a priori knowledge is not available, one data-driven choice of the signal separation space is as follows (see for example~\cite{Bai2003,BaiHanShi2020,DuanBaiHan2023}): The signal separation space is spanned by the left leading singular basis vectors obtained by the singular value decomposition of the data matrix~$X$. Equivalently this corresponds to choosing the leading eigenspace of the empirical covariance matrix~$\frac{1}{n} X X^{\top}$. We use this approach when studying breakpoint detection in factor models (see Section~\ref{sec:breakpoint_detection_in_factor_models} below for more details).    
     A possible problem is that PCA/SVD signal spaces can fail to separate regimes when the directions that change
across regimes are \emph{not} aligned with the leading principal components (see~\cite{AueRiceSonmez2018}).
 Concretely, in the retained signal coordinates $y_t=U_q^{\top}x_t$ the detector observes only the compressed regime moment $U_q^{\top}\Sigma\,U_q$ of each regime second moment~$\Sigma$. Writing $D=\Sigma_1-\Sigma_2$ for the difference between the two regimes, the break is visible in the signal time series if and only if $U_q^{\top} D\, U_q\neq 0$. Choosing the leading-$q$ eigenspace as signal separation space therefore hides the break exactly when~$D$ is supported on directions orthogonal to that eigenspace --- typically trailing, low-variance directions that PCA discards --- whereas a change in the variance along a \emph{retained} leading direction leaves $U_q^{\top} D\, U_q\neq 0$ and remains visible.

\end{remark}

As we send~$n \to \infty$, let us introduce the notion of macroscopic intervals.
\begin{defn}[Macroscopic intervals]
For~$0 \leq a < b \leq 1$ let~$I(a,b)$ denote the macroscopic interval
\begin{align}
    I(a,b) = \{\lfloor a n \rfloor+1 , \ldots , \lfloor b n \rfloor \}.
\end{align}
    We define the family $\mathcal{I}$ of all macroscopic intervals by
\begin{align}
    \mathcal{I}
    := \Big\{ I(a,b)
          \;\Big|\; 0 \leq a < b \leq 1 \Big\}.
\end{align}
Let $\theta\in(0,1)$ denote the (unknown) breakpoint. We define the sets of macroscopic
intervals that lie entirely in the left- and right-hand regimes by
\[
    \mathcal{I}_{\leq}
    := \bigl\{ I\in\mathcal{I} : I \subset \{1,\dots,\lfloor \theta n \rfloor\} \bigr\},
    \qquad
    \mathcal{I}_{>}
    := \bigl\{ I\in\mathcal{I} : I \subset \{\lfloor \theta n \rfloor+1,\dots,n\} \bigr\}.
\]
\end{defn}

\begin{defn}\label{def_singal_and_noise}

For a nonempty macroscopic interval $ I= \{\lfloor a n\rfloor +1, \ldots, \lfloor b n \rfloor \}\in\mathcal{I}$ we define the empirical signal second moment matrix as
\begin{align}
   \widehat \Sigma_{a : b}
   := \frac{1}{|I|}\sum_{t \in I} y_t y_t^{\top} \in\mathbb R^{q\times q}.
\end{align}    

We distinguish two families of macroscopic blocks.  The break-anchored
blocks are
\begin{align}
\mathcal I_{\theta}^{\rm br}
:={}&
\bigl\{I(s,\theta):0\leq s<\theta\bigr\}
\cup
\bigl\{I(\theta,s):\theta<s\leq1\bigr\},
\end{align}
and the regime-cut blocks are all pieces obtained by cutting either regime once,
\begin{align}
\mathcal I_{\theta}^{\rm cut}
:={}&
\mathcal I_{\theta}^{\rm br}
\cup
\bigl\{I(0,s):0<s<\theta\bigr\}
\cup
\bigl\{I(s,1):\theta<s<1\bigr\}.
\end{align}
We define the spectral ceiling and floor by
\begin{align}\label{equ:def_M_m}
M
&:=
\sup_{I(a,b)\in\mathcal I_{\theta}^{\rm br}}
\limsup_{n\to\infty}\bigl\|\widehat\Sigma_{a:b}\bigr\|,
\\
m
&:=
\inf_{I(a,b)\in\mathcal I_{\theta}^{\rm cut}}
\liminf_{n\to\infty}
\lambda_{\min}\!\bigl(\widehat\Sigma_{a:b}\bigr).
\end{align}

We additionally define 
\begin{align}
  & \Delta_{\mathrm{within}}
  :=  \sup_{0\leq a_i <b_i \leq \theta, i =1,2 \text{ or } \theta \leq a_i < b_i \leq 1 , i =1,2}  \limsup_{n \to \infty}
    \| \widehat \Sigma_{a_1: b_1} - \widehat \Sigma_{a_2: b_2} \|_{F} , \label{eq:Delta-within} \\
  & \Delta_{\mathrm{between}}
  :=  \inf_{0\leq a_1 <b_1 \leq \theta, \  \theta \leq a_2 < b_2 \leq 1}  \liminf_{n \to \infty}
    \| \widehat \Sigma_{a_1: b_1} - \widehat \Sigma_{a_2 : b_2}  \|_{F} . \label{eq:Delta-between}
\end{align}
For later use, if $\xi\in(0,1)$ is treated as a candidate split, we denote by
$\Delta_{\mathrm{within}}(\xi)$ and $\Delta_{\mathrm{between}}(\xi)$
the quantities defined in~\eqref{eq:Delta-within} and~\eqref{eq:Delta-between},
respectively, with the regime boundary $\theta$ replaced by $\xi$. Thus,
$\Delta_{\mathrm{within}}=\Delta_{\mathrm{within}}(\theta)$ and
$\Delta_{\mathrm{between}}=\Delta_{\mathrm{between}}(\theta)$.

\end{defn}

\begin{remark}
    We want to point out that we do not make structural assumptions on how the data matrix is generated (e.g.~via a factor model), or require that the data follows a strong law of large numbers. 
\end{remark}

\begin{remark}\label{rem:interpretaion_signal_noise}
The empirical signal second moment matrices will be the pivotal tool to detect the breakpoint. The quantity~$\Delta_{\text{within}}$ measures the worst-case fluctuation of the empirical signal second moment matrix within a regime; it therefore measures the noise. The quantity~$\Delta_{\text{between}}$ measures the minimal difference between signal second moment matrices of different regimes. It therefore measures the strength of the regime separation signal.
\end{remark}

\begin{defn}[Dominant breakpoint]\label{def:pathwise_breakpoint}
A candidate split~$\xi\in(0,1)$ is called a \emph{dominant breakpoint} of the signal time series if
\begin{align}\label{eq:def_breakpoint}
\Delta_{\mathrm{between}}(\xi)>\Delta_{\mathrm{within}}(\xi),
\end{align}
i.e.~if blocks on opposite sides of~$\xi$ differ asymptotically more than any two blocks on the same side of~$\xi$. 
\end{defn}
The data may contain several breakpoints, but the dominant breakpoint is unique. 
\begin{lem}[At most one dominant breakpoint]\label{lem:unique_breakpoint}
There is at most one dominant breakpoint, i.e.~at most one~$\xi\in(0,1)$ satisfying~\eqref{eq:def_breakpoint}. Therefore, we reserve the symbol~$\theta \in (0,1)$ to denote the dominant breakpoint in the signal time series. 
\end{lem}
\begin{proof}
We argue by contradiction and let $0<\nu_1<\nu_2<1$ be two dominant breakpoints. The macroscopic blocks $I(0,\nu_1)$ and $I(\nu_1,\nu_2)$ lie on opposite sides of~$\nu_1$ but on the same side of~$\nu_2$, so the definitions~\eqref{eq:Delta-within} and~\eqref{eq:Delta-between} give
\[
\Delta_{\mathrm{between}}(\nu_1)
\leq\liminf_{n\to\infty}\big\|\widehat\Sigma_{0:\nu_1}-\widehat\Sigma_{\nu_1:\nu_2}\big\|_F
\leq\limsup_{n\to\infty}\big\|\widehat\Sigma_{0:\nu_1}-\widehat\Sigma_{\nu_1:\nu_2}\big\|_F
\leq\Delta_{\mathrm{within}}(\nu_2).
\]
Exchanging the roles of the two candidates and using the blocks $I(\nu_1,\nu_2)$ and $I(\nu_2,1)$ gives $\Delta_{\mathrm{between}}(\nu_2)\leq\Delta_{\mathrm{within}}(\nu_1)$. If both~$\nu_1$ and~$\nu_2$ satisfied~\eqref{eq:def_breakpoint}, chaining these inequalities would yield
\[
\Delta_{\mathrm{within}}(\nu_1)
<\Delta_{\mathrm{between}}(\nu_1)
\leq\Delta_{\mathrm{within}}(\nu_2)
<\Delta_{\mathrm{between}}(\nu_2)
\leq\Delta_{\mathrm{within}}(\nu_1),
\]
a contradiction.
\end{proof}
\begin{remark}
  A dominant breakpoint need not exist: for scalar data taking the second-moment levels $a\to b\to a$ on three equal segments, every~$\xi\in(0,1)$ has level-$a$ blocks on both sides, so $\Delta_{\mathrm{between}}(\xi)=0$ for all~$\xi$, although the data has two genuine breaks, neither of which is dominant.
\end{remark}

In this section we work with the following assumptions:

\begin{assumption}
\begin{itemize}
    \item[(A1)]\assumlabel{ass:bounded_second_moment}{(A1)} {\bf Bounded operator norm.} We assume that $M<\infty$.
    \item[(A2)]\assumlabel{ass:SLLN_in_each_regime}{(A2)} {\bf SLLN in each regime.} We assume that there are positive semidefinite matrices~$\Sigma_{\leq} \neq \Sigma_{>} \in \mathbb{R}^{q \times q}$ such that for all~$0 \leq a_1 < b_1 \leq \theta$ and~$\theta \leq a_2 < b_2 \leq 1$ it holds
    \[
        \lim_{n \to \infty} \widehat \Sigma_{a_1: b_1} = \Sigma_{\leq} \quad \mbox{and} \quad         \lim_{n \to \infty} \widehat \Sigma_{a_2: b_2} = \Sigma_{>}.
    \]
    \item[(A3)]\assumlabel{ass:signal_weak_domination}{(A3)} {\bf Weak signal-separation condition.}  For~$\varepsilon \geq 0$
    \[
        \Delta_{\text{between}}  > \frac{M+\varepsilon}{m+\varepsilon} \Delta_{\text{within}}.
    \]
    \item[(A4)]\assumlabel{ass:aposeterior_signal_separation}{(A4)} {\bf A posteriori strong signal-separation condition.} For~$\varepsilon \geq 0$
\[
 \Delta_{\text{between}}  >   
  \frac{1}{2 \sqrt{\theta (1- \theta)}}\frac{ M+\varepsilon}{m+\varepsilon} \Delta_{\text{within}}.
\]
    \item[(A5)]\assumlabel{ass:apriori_signal_separation}{(A5)} {\bf A priori strong signal-separation condition.} For $0 \leq \delta \leq \theta \leq 1- \delta$ and~$\varepsilon \geq 0$
\[
\Delta_{\text{between}}  >   
  \frac{1}{2 \sqrt{\delta (1- \delta)}}\frac{ M+\varepsilon}{m+\varepsilon} \Delta_{\text{within}}.
\]
    \item[(PD)]\assumlabel{ass:positive_definiteness}{(PD)} {\bf Positive definiteness.} We assume that $m>0$.
\end{itemize}
\end{assumption}

\begin{remark}
   The second moment bound of Assumption~\ref{ass:bounded_second_moment} is standard and does not rule out the analysis of heavy-tailed data~$(x_t)$ as it applies to the signal time series~$(y_t)$. A nonlinear mapping $x_t \mapsto y_t$ can ensure Assumption~\ref{ass:bounded_second_moment}, for example by enforcing a uniform bound on~$\sup_{t}|y_t|$ via truncation.
\end{remark}

\begin{remark}[Second moments versus covariances]\label{rem:second_moment_vs_covariance}
Throughout the manuscript we work with the empirical \emph{second moment} matrices
\[
\widehat\Sigma_{a:b} \;=\; \tfrac{1}{|I|}\sum_{t\in I}y_t y_t^{\top},
\]
i.e.~uncentered averages of $y_t y_t^{\top}$. The framework extends naturally to the centered \emph{covariance} matrices
\[
\widehat\Sigma_{a:b}^{\rm cov} \;:=\; \widehat\Sigma_{a:b} \;-\; \bar y_{a:b}\,\bar y_{a:b}^{\top},
\qquad
\bar y_{a:b} \;:=\; \tfrac{1}{|I|}\sum_{t\in I}y_t,
\]
which differ from the second moments by a rank-one correction. Using the same argument, i.e.~quantifying the Jensen gap, naturally gives the correct definition of signal~$\Delta_{\rm between}$ and noise~$\Delta_{\rm within}$, leading to variants of the main results Theorem~\ref{thm:classical_QML_consistency}, Theorem~\ref{thm:QML-robust}, and Corollary~\ref{thm:qml_consistency_factor}. We chose to work with the second moments throughout because they have clean additivity on macroscopic intervals (cf.~\eqref{equ:decomp_empirical_second_moments} in the proof of Theorem~\ref{thm:QML-robust}), which simplifies the calculations.
\end{remark}

\begin{remark}\label{rem:assumptions_as_signal_to_noise_ratios}
Building upon Remark~\ref{rem:interpretaion_signal_noise}, the ratio~$\frac{\Delta_{\text{between}}}{\Delta_{\text{within}}}$ can be interpreted as the signal-to-noise ratio, and therefore the Assumptions (A2)--(A5) quantify different signal-to-noise ratios. We observe that Assumption (A2) implies 
\begin{align}
\Delta_{\text{between}}< \infty \quad \mbox{and} \quad \Delta_{\text{within}} =0,
\end{align}
having the interpretation that the signal-to-noise ratio is infinite. We immediately see that (A2) implies~(A3),~(A4), and (A5) for all choices of~$\delta>0$ and~$\varepsilon>0$. Moreover, under~(A2) every block in $\mathcal I_{\theta}^{\rm br}$ and $\mathcal I_{\theta}^{\rm cut}$ converges to the corresponding regime matrix, and therefore
\[
M=\max\{\|\Sigma_{\leq}\|,\|\Sigma_{>}\|\},
\qquad
m=\min\{\lambda_{\min}(\Sigma_{\leq}),\lambda_{\min}(\Sigma_{>})\}.
\]
In particular, (A2) implies~(A1). Additionally, we see that if~$\theta \in [\delta, 1 - \delta]$ then~$\delta(1-\delta) \leq \theta(1-\theta)$, and hence the condition~(A5) implies~(A4). Finally, since~$2\sqrt{\theta(1-\theta)} \leq 1$, the prefactor in~(A4) is at least one, so~(A4) implies~(A3). Altogether we obtain the implication chain
\[
    \textup{(A5)}\implies \textup{(A4)}\implies \textup{(A3)},
\]
where the first implication holds for~$\theta \in [\delta, 1-\delta]$.
\end{remark}

\begin{defn}[QML objective function and QML breakpoint estimator~$\hat \theta_n$]\label{def:qml_objective_function}
For fixed~$\varepsilon \geq 0$ and every~$\nu\in(0,1)$ such that both empirical blocks are nonempty, equivalently
$1\leq\lfloor \nu n\rfloor\leq n-1$, we define the QML objective by
\begin{align}
  J_{n}(\nu)
  := \nu \log\det\big(  \widehat \Sigma_{0: \nu } + \varepsilon I_q  \big)
   + (1- \nu) \log\det\big( \widehat \Sigma_{\nu : 1} + \varepsilon I_q\big).
\end{align}

Here the empirical second moments are understood according to Definition~\ref{def_singal_and_noise}, in particular with the floor convention in their block endpoints. Thus the formula also defines~$J_n$ at off-grid fractions whenever both blocks are nonempty. For breakpoint estimation, the data are discrete and it suffices to evaluate~$J_n$ at the grids
\[
    \mathcal D_n:=\left\{\frac{k}{n}:1\leq k\leq n-1\right\},
    \qquad
    \mathcal D_n(\delta):=\mathcal D_n\cap[\delta,1-\delta],
    \quad 0<\delta<1/2 .
\]
The untrimmed and trimmed QML estimators of the breakpoint~$\theta$ are obtained by minimizing~$J_n$ and formally defined via
\[
    \hat \theta_n\in\arg\min_{\nu\in\mathcal D_n}J_n(\nu),
    \qquad
    \hat \theta_{n,\delta}\in\arg\min_{\nu\in\mathcal D_n(\delta)}J_n(\nu).
\]

\end{defn}

There are three important reasons to include the~$\varepsilon$-regularization~$+ \varepsilon I_q$ in the definition of the QML-objective function~$J_{n}$:
\begin{remark}\label{rem:epsilon_regularization}
    The~$\varepsilon$-regularization~$+ \varepsilon I_q$ in the definition of the QML-objective function~$J_{n}$ allows eigenvalues of the signal second moment matrix to take on the value~$0$. For example, this allows one to model the situation where a signal completely vanishes in the other regime. The main financial application is the rank collapse in a crisis regime, where at the breakpoint all assets are strongly correlated and the cross-sectional correlations are negligible in comparison.
\end{remark}
\begin{remark}\label{rem:ridge_regularization_and_SNR}
 Adding the~$\varepsilon$-regularization~$+ \varepsilon I_q$ in the definition of the QML-objective function~$J_{n}$ gives one additional degree of freedom to choose from. By choosing~$\varepsilon$ appropriately, one can weaken the signal-separation conditions. In particular (A3) can be weakened to
 \begin{itemize}
     \item[(A3*)] We assume that
     \begin{align}
         \Delta_{\text{between}}  > \Delta_{\text{within}} .
     \end{align}
 \end{itemize}
Then (A3*) implies~(A3) for
\begin{align}
    \varepsilon\geq \max \left\{0, \frac{ M \Delta_{\text{within}}- m\Delta_{\text{between}}}{\Delta_{\text{between}}-\Delta_{\text{within}}} \right\}.
\end{align}
Thus~\textup{(A3*)} is the best condition obtained by
optimizing the sufficient local condition~\textup{(A3)} over the ridge
parameter.

 The assumption (A4) can be weakened to:
\begin{itemize}
    \item[(A4*)]\label{ass:weak_signal_to_noise_optimized_epsilon} We assume that   
\begin{align}\tag{(A4*)}\label{equ:strong_a_posterior_signal_separation_condition_no_M_m}
  \Delta_{\text{between}}  >   
  \frac{1}{2 \sqrt{\theta (1- \theta)}} \Delta_{\text{within}}.
\end{align}
\end{itemize}
Then (A4*) implies~(A4) for
\begin{align}
    \varepsilon\geq \max \left\{0, \frac{ M \Delta_{\text{within}}- 2 \sqrt{\theta (1- \theta)}\, m\Delta_{\text{between}}}{2 \sqrt{\theta (1- \theta) }\Delta_{\text{between}}-\Delta_{\text{within}}} \right\}.
\end{align}
In contrast to the preceding local condition, the
ridge-optimized global threshold in~\textup{(A4*)} is shown to be sharp in
Proposition~\ref{prop:sharp-A4-fixed-theta} and
Remark~\ref{rem:sharp-A4-fixed-theta}.
We want to point out that in the special case~$\theta = 0.5$ the right-hand side of Assumption~\ref{equ:strong_a_posterior_signal_separation_condition_no_M_m} equals~1. The assumption~(A5) can be weakened similarly.
\end{remark}

\begin{remark}[The ridge removes the endpoint boundary layers]\label{rem:untrimmed_minimization}
The third reason is specific to the breakpoint problem. For~$\varepsilon=0$ the minimization of~$J_n$ has to be restricted to a trimmed grid~$\mathcal D_n(\delta)$: a block second moment built from fewer than~$q$ samples is rank-deficient, so~$J_n=-\infty$ on the short endpoint windows, and, more generally, the empirical block second moments near the endpoints need not remain uniformly positive definite, so that no uniform control of~$J_n$ is available there. For~$\varepsilon>0$ the ridge bounds the spectrum of the regularized block second moments from below by~$\varepsilon$, uniformly over all windows, which removes these endpoint boundary layers (see Lemma~\ref{lem:boundary_energy} and Lemma~\ref{lem:boundary_logdet} below). As a consequence, the QML criterion can be minimized over the full grid~$\mathcal D_n$ (see Theorem~\ref{thm:classical_QML_consistency}, Theorem~\ref{thm:QML-robust}, and Corollary~\ref{thm:qml_consistency_factor} below).
\end{remark}

\begin{remark}[Interpretation of QML objective function]\label{rem:interpretation_qml_objective_function}
The name QML estimator comes from the fact that~$J_n$ is the profile log-likelihood of the following Gaussian block model of a regime change (for simplicity, $\varepsilon=0$). For a candidate breakpoint~$1\leq t<n$, assume the data~$Y:=(y(s))_{1\leq s\leq n}$, written in signal coordinates, is generated with unknown covariance matrices $\Sigma_{\leq t}$, $\Sigma_{>t}$ as
\[
\{y(s)\}_{ 1\leq s \leq t} \stackrel{\mathrm{i.i.d.}}{\sim}\mathcal N(0,\Sigma_{\leq t}),
\quad \mbox{and} \quad
\{y(s)\}_{t < s \leq n}\stackrel{\mathrm{i.i.d.}}{\sim}\mathcal N(0,\Sigma_{>t}).
\]
The Gaussian log-likelihood $\ell(Y\mid t,\Sigma_{\leq t},\Sigma_{>t})$ is maximized uniquely at $\Sigma_{\leq t}=\widehat\Sigma_{0:\frac tn}$ and $\Sigma_{>t}=\widehat\Sigma_{\frac tn:1}$, and a standard trace computation (the quadratic terms evaluate to $tq$ and $(n-t)q$ at the maximizer) gives
\begin{align}
    \max_{\Sigma_{\leq t} , \Sigma_{>t}} \ell(Y \mid t, \Sigma_{\leq t} , \Sigma_{>t} )
   =  - \frac{n}{2} J_{n}\!\left(\tfrac{t}{n}\right) - \frac{nq}{2} \left(1 + \log(2 \pi)\right).
\end{align}
Minimizing the QML objective over~$t$ is therefore maximizing the profile likelihood over the candidate breakpoint.
\end{remark}

First we consider the classical case, which means that the empirical second moment matrices satisfy the SLLN of Assumption~\ref{ass:SLLN_in_each_regime}.
\begin{defn}[Population targets for left/right segments]\label{def:population_targets}
For~$\nu \in (0,1)$ we define the population targets for the left and the right segment by
\ifwidedisplays
\[
\Sigma_{0: \nu}
:=
\begin{cases}
\Sigma_{\le}, & \nu \le \theta,\\[4pt]
\frac{\theta}{\nu}\,\Sigma_{\le}
+\frac{\nu-\theta}{\nu}\,\Sigma_{>}, & \nu > \theta,
\end{cases}
\qquad
\Sigma_{\nu:1}
:=
\begin{cases}
\frac{\theta-\nu}{1-\nu}\,\Sigma_{\le}
+\frac{1-\theta}{1-\nu}\,\Sigma_{>}, & \nu < \theta,\\[4pt]
\Sigma_{>}, & \nu \ge \theta.
\end{cases}
\]
\else
\[
\Sigma_{0: \nu}
:=
\begin{cases}
\Sigma_{\le}, & \nu \le \theta,\\[6pt]
\displaystyle
\frac{\theta}{\nu}\,\Sigma_{\le}
+\frac{\nu-\theta}{\nu}\,\Sigma_{>}, & \nu > \theta,
\end{cases}
\quad
\Sigma_{\nu:1}
:=
\begin{cases}
\displaystyle
\frac{\theta-\nu}{1-\nu}\,\Sigma_{\le}
+\frac{1-\theta}{1-\nu}\,\Sigma_{>}, & \nu < \theta,\\[6pt]
\Sigma_{>}, & \nu \ge \theta.
\end{cases}
\]
\fi
     For~$\varepsilon \geq 0$, we define the limiting QML-objective function ~$J_\infty (\nu)$ for~$\nu \in (0, 1)$ via
 \begin{align}
  J_{\infty}(\nu)
  := \nu \log\det\left( \Sigma_{0: \nu} + \varepsilon I_q \right)
   + (1- \nu) \log\det\big(  \Sigma_{\nu :1}+ \varepsilon I_q\big).
 \end{align}
We define its values at~$\nu=0$ and~$\nu=1$ by continuous extension.
\end{defn}

\begin{theorem}\label{thm:classical_QML_consistency}
We assume~\ref{ass:SLLN_in_each_regime} (which also immediately implies~\ref{ass:bounded_second_moment}). If $\varepsilon=0$, we additionally assume~\ref{ass:positive_definiteness}. Then:
 \begin{enumerate}[label=(\roman*)]
     \item The limiting QML objective function~$J_\infty$ is strictly concave on $(0, \theta)$ and on~$(\theta, 1)$ and has a cusp at its unique minimizer~$\theta$. More precisely, for all~$\nu \in [0,1]$
     \begin{align}\label{equ:cusp_separation_J_infty}
         J_{\infty} (\nu) - J_{\infty}(\theta) \geq \frac{\min\{\theta,\,1-\theta\}}{2(M +\varepsilon)^2}\,\|\Sigma_{>} - \Sigma_{\leq}\|_F^2 \
 |\nu - \theta|.
     \end{align}
     \item If $\varepsilon>0$, then
     \begin{align}\label{equ:uniform_convergence_J_n_untrimmed}
         \lim_{n \to \infty}\ \sup_{\nu \in \mathcal D_n} |J_{n} (\nu) - J_{\infty} (\nu)| = 0.
     \end{align}
     If $\varepsilon=0$, then, for every~$0 < \delta < \frac{1}{2}$,
     \begin{align}\label{equ:uniform_convergence_J_n}
         \lim_{n \to \infty}\ \sup_{\nu \in \mathcal D_n(\delta)} |J_{n} (\nu) - J_{\infty} (\nu)| = 0.
     \end{align}
    \item If $\varepsilon>0$, every untrimmed QML breakpoint estimator $\hat \theta_{n}\in \arg \min_{\nu \in \mathcal D_n} J_n(\nu)$ is consistent, i.e.~$\lim_{n \to \infty} \hat \theta_n = \theta.$ \newline
If $\varepsilon=0$ and $\delta \in (0, \min\left\{ \theta, 1 - \theta \right\})$, then every trimmed QML breakpoint estimator $\hat \theta_{n,\delta}\in \arg \min_{\nu \in \mathcal D_n(\delta)} J_n(\nu)$ is consistent, i.e.~$\lim_{n \to \infty} \hat \theta_{n,\delta} = \theta.$
 \end{enumerate}
\end{theorem}

 Though the case~$\varepsilon=0$ of Theorem~\ref{thm:classical_QML_consistency} is classical, its part for~$\varepsilon>0$ is new: the ridge removes the endpoint boundary layers, so the QML objective can be minimized over the full grid and no trimming of the candidate set is needed (see Remark~\ref{rem:untrimmed_minimization}).
In Section~\ref{subsec:heuristic-eps} we give a heuristic argument for Theorem~\ref{thm:classical_QML_consistency}, built from the two ingredients that also drive the proof of its robust counterpart (Theorem~\ref{thm:QML-robust} below): the concavity of~$\log\det$ and the explicit two-sided Jensen-gap estimate of Lemma~\ref{lem:jensen_gap_logdet}. The rigorous proof is given in  Section~\ref{sec:proof_consistency_infinite_signal_to_noise_ratio}. 

\begin{remark}[Why consistency without a rate]\label{rem:no_rate}
Under stronger, quantitative assumptions the breakpoint can be located much more precisely: for pervasive factor models with i.i.d.\ factors and bounded higher moments, \cite{DuanBaiHan2023} estimate the break index up to a bounded error, i.e.\ $|\hat\theta_n-\theta|=O_p(1/n)$; see also the classical change-point rates in~\cite{Hinkley1970,Bai1997}. Theorem~\ref{thm:classical_QML_consistency} only concludes $\hat\theta_n\to\theta$. This is not a shortcoming of the method but the exact price of the qualitative Assumption~\ref{ass:SLLN_in_each_regime}, which demands convergence of the block second moments without quantifying its speed. The following example shows that under our assumptions no rate of convergence can hold.

Let $q=1$ and let the data take just two values: fix $0<a\neq b$, a null sequence $\delta_n\downarrow0$, and set
\[
y_t^{(n)}:=\sqrt a \quad\text{for } t\leq\lfloor(\theta+\delta_n)n\rfloor,
\qquad
y_t^{(n)}:=\sqrt b \quad\text{for } t>\lfloor(\theta+\delta_n)n\rfloor,
\]
so the jump sits at $\theta+\delta_n$, slightly to the right of $\theta$. Every fixed macroscopic interval contained in $(0,\theta)$ or in $(\theta,1)$ eventually contains no jump, so Assumption~\ref{ass:SLLN_in_each_regime} holds with breakpoint~$\theta$, $\Sigma_\leq=a$, $\Sigma_>=b$, and moreover $\Delta_{\mathrm{within}}=0$ and $\Delta_{\mathrm{between}}=|a-b|>0$: all assumptions hold with constants independent of the sequence~$\delta_n$. On the other hand, at every finite~$n$ the data are exactly two-level with jump at $\theta+\delta_n$, and an elementary computation shows that $J_n$ is minimized at this jump up to $O(1/n)$, i.e.\ $\hat\theta_n=\theta+\delta_n+O(1/n)$. Since $\delta_n$ may tend to zero arbitrarily slowly, no rate for $|\hat\theta_n-\theta|$ can be deduced: consistency is the exact conclusion, and Theorem~\ref{thm:classical_QML_consistency} is sharp in this sense. As the example satisfies even Assumption~\ref{ass:SLLN_in_each_regime}, the same applies a fortiori to Theorem~\ref{thm:QML-robust}.

The example also shows where a rate must come from: the estimation error $\delta_n$ is exactly the speed at which the block second moments in Assumption~\ref{ass:SLLN_in_each_regime} converge. Quantifying that speed --- which is what the stochastic assumptions of~\cite{DuanBaiHan2023} do --- is precisely what recovers a rate. In short: rates are the price of unquantified hypotheses, and the example shows the price is unavoidable.
\end{remark}

Using Remark~\ref{rem:assumptions_as_signal_to_noise_ratios}, the interpretation of Theorem~\ref{thm:classical_QML_consistency} is that the QML breakpoint estimator is consistent under infinite signal-to-noise ratio. However, data is usually corrupted and the corruption can be systemic over the whole dataset. There is no reason for the noise to be independent from the data making it hard or impossible to filter it out. One needs to model the in-regime empirical second moments as
$\widehat \Sigma_{a:b} = \Sigma_{\leq} + E_{a:b}$, where~$E_{a:b}$ denotes the measurement error, and there is no reason to believe that~$\lim_{n \to \infty} E_{a:b} = 0$, implying that Assumption~\ref{ass:SLLN_in_each_regime} is too strong; there might be no underlying SLLN at all. We henceforth take on a signal-theoretic point of view and show that a high enough signal-to-noise ratio should be sufficient to detect the breakpoint, leading us to the main result of this article.

\begin{theorem}[Consistency of the QML breakpoint estimator]
\label{thm:QML-robust}
We consider~$\varepsilon \geq 0$, and if~$\varepsilon =0$ we assume~\ref{ass:positive_definiteness}. It holds:
\begin{enumerate}[label=(\roman*)]
    \item Under the Assumption~\ref{ass:bounded_second_moment} and~\ref{ass:signal_weak_domination}, there are~$\bar\delta, c>0$ such that
    \begin{align}
        \liminf_{n \to \infty} J_n(\nu) - J_n(\theta) \geq c |\nu - \theta|, \quad \mbox{for all~$\nu \in (\theta -\bar\delta, \theta+ \bar\delta)\cap(0,1)$.}
    \end{align}
    This means that asymptotically the breakpoint~$\theta$ is a local minimum of the limiting QML objective function~$
    \liminf_{n} J_n$.\medskip
    
    \item Under the Assumption~\ref{ass:bounded_second_moment} and~\ref{ass:aposeterior_signal_separation}, there is~$c>0$ such that
        \begin{align}\label{equ:minimizer_a4}
        \liminf_{n \to \infty} J_n(\nu) - J_n(\theta) \geq c |\nu - \theta|, \quad \mbox{for all~$\nu \in (0,1)$.}
    \end{align}
    
      In particular, if $\varepsilon>0$, every untrimmed QML breakpoint estimator $\hat \theta_{n}\in \arg \min_{\nu \in \mathcal D_n} J_n(\nu)$ is consistent, i.e.~$\lim_{n \to \infty} \hat \theta_n = \theta$. \newline
      If $\varepsilon=0$, then, for every~$\delta \in (0,\min\{\theta,1-\theta\})$, every trimmed QML breakpoint estimator $\hat \theta_{n,\delta}\in \arg \min_{\nu \in \mathcal D_n(\delta) } J_n(\nu)$ is consistent, i.e.~$\lim_{n \to \infty} \hat \theta_{n,\delta} = \theta$.
    
\end{enumerate}
\end{theorem}

The proof of Theorem~\ref{thm:QML-robust} is carried out in Section~\ref{sec:proof_consistency_finite_signal_to_noise}, building on the structure of the proof of Theorem~\ref{thm:classical_QML_consistency} and on the same two ingredients (see Section~\ref{subsec:heuristic-eps}).  The constants~$c$ are explicit. Writing $S:=\Delta_{\text{between}}^2/(M+\varepsilon)^2$ and $W:=\Delta_{\text{within}}^2/(m+\varepsilon)^2$ for the normalized signal and noise levels, one may take $c=(S-W)/4$ in part~(i), see~\eqref{eq:const_local}, and $c$ as in~\eqref{eq:const_global} in part~(ii), namely $c=\tfrac12\min\{\theta,1-\theta\}\,S$ if $W=0$ and $c=\bigl(2\sqrt{\theta(1-\theta)SW}-W\bigr)/(2\max\{\theta,1-\theta\})$ if $W>0$. Both constants depend only on the gap in Assumptions~(A3) and~(A4), i.e.~on the amount by which the signal-to-noise ratio~$\Delta_{\text{between}}/\Delta_{\text{within}}$ exceeds the threshold on the respective right-hand side.\\

Proposition~\ref{prop:sharp-A4-fixed-theta} below shows
that the ridge-optimized signal-to-noise threshold in~\textup{(A4*)} is
necessary for a uniform global-recovery guarantee. In this sense
Theorem~\ref{thm:QML-robust} is optimal. This does not assert that the
additional conditioning factor in the fixed-$\varepsilon$
condition~\textup{(A4)} is itself necessary.

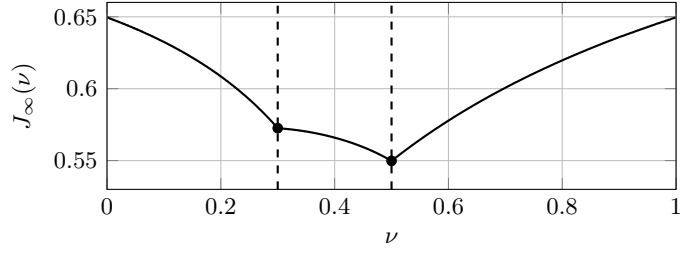
\begin{figure}[t]
\centering
\begin{tikzpicture}
\begin{axis}[
  width=0.72\textwidth,
  height=0.32\textwidth,
  xmin=0, xmax=1,
  ymin=0.53, ymax=0.66,
  xlabel={$ \nu $},
  ylabel={$ J_\infty(\nu) $},
  grid=both,
  ticklabel style={font=\small},
  label style={font=\small},
]

\def\theta{0.30}
\def\nustar{0.50}
\def\K{1.0}
\def\r{1.05}
\def\eps{0.1}

\pgfmathsetmacro{\a}{1+\K+\r}
\pgfmathsetmacro{\b}{1+\K}
\pgfmathsetmacro{\c}{1}

\addplot[thick, domain=0.001:\theta, samples=400]
({x},{
  x*ln(\a+\eps)
  + (1-x)*ln( ( ((\theta - x)*\a + (\nustar-\theta)*\b + (1-\nustar)*\c)/(1-x) ) + \eps )
});

\addplot[thick, domain=\theta:\nustar, samples=600]
({x},{
  x*ln( ( (\theta*\a + (x-\theta)*\b)/x ) + \eps )
  + (1-x)*ln( ( ((\nustar-x)*\b + (1-\nustar)*\c)/(1-x) ) + \eps )
});

\addplot[thick, domain=\nustar:0.999, samples=400]
({x},{
  x*ln( ( (\theta*\a + (\nustar-\theta)*\b + (x-\nustar)*\c)/x ) + \eps )
  + (1-x)*ln(\c+\eps)
});

\addplot[dashed, thick] coordinates {(\theta,0.53) (\theta,0.66)};
\addplot[dashed, thick] coordinates {(\nustar,0.53) (\nustar,0.66)};

\addplot[only marks, mark=*, mark size=1.8pt]
coordinates {(\theta,0.5725671) (\nustar,0.5498060)};

\end{axis}
\end{tikzpicture}

\caption{Illustration of the limiting objective $J_\infty(\nu)$ in the deterministic three-level example from Proposition~\ref{prop:sharp-A4-fixed-theta}
(with $\theta=0.30$, $\nu_\star=0.5$, $K=1$, $r=1.05$, $\varepsilon=0.1$). Dashed lines indicate $\theta$ and $\nu_\star$.}
\label{fig:illustration_optimality}
\end{figure}

\begin{proposition}
\label{prop:sharp-A4-fixed-theta}
For $\theta\in(0,\tfrac12)$ and~$0< K < r < \infty$ we define the deterministic signal time series $(y_t)_{1 \leq t \leq n}$ via
\begin{equation}\label{eq:triangular-array-sharpA4}
y_t
=
\begin{cases}
\sqrt{1 + K + r}, & 1 \le t\le \lfloor \theta n\rfloor,\\
\sqrt{1+K}, & \lfloor \theta n\rfloor < t\le \lfloor 0.5  n\rfloor,\\
\sqrt{1}, & \lfloor 0.5  n\rfloor < t\le n .
\end{cases}
\end{equation}
Then at~$\theta$ and~$0.5$ it holds
\begin{align}\label{equ:optimality_signal_and_noise}
  \Delta_{\mathrm{within}} (\theta) = K,
\quad
    \Delta_{\mathrm{between}} (\theta) =r,
\quad
  \frac{\Delta_{\mathrm{between}} (\theta)}{\Delta_{\mathrm{within}} (\theta)}=\frac{r}{K}, \quad \mbox{and}\\
  \Delta_{\mathrm{within}} (0.5) = r,
\quad
    \Delta_{\mathrm{between}} (0.5) =K,
\quad   \frac{\Delta_{\mathrm{between}} (0.5)}{\Delta_{\mathrm{within}} (0.5)}=\frac{K}{r}.\label{equ:optimality_signal_and_noise_at_05}
\end{align}
Moreover, if 
\begin{align}\label{equ:counterexample_conditions}
 \frac{r}{K} = \frac{\Delta_{\mathrm{between}} (\theta)}{\Delta_{\mathrm{within}}(\theta)} < \frac{1}{2 \sqrt{\theta (1- \theta)}}
\end{align}
then, for every regularization parameter $\varepsilon >0$,  the pointwise limit $J_\infty(\nu):=\lim_{n}J_n(\nu)$ exists for every $\nu\in(0,1)$ and has a unique minimizer at $0.5$. In particular, it satisfies
\begin{align}\label{equ:counterexample_energy}
J_\infty(0.5) < J_\infty(\theta).
\end{align}
\end{proposition}

For an illustration of Proposition~\ref{prop:sharp-A4-fixed-theta} we refer to Figure~\ref{fig:illustration_optimality}. The proof is given in Section~\ref{sec:proof_optimality}.
\begin{remark}[Optimality of (A4*) for QML-breakpoint detection]
\label{rem:sharp-A4-fixed-theta}
In the situation of
Proposition~\ref{prop:sharp-A4-fixed-theta} there are two structural changes,
at~$\theta$ and at~$\nu_\star=\tfrac12$. The designation of~$\theta$ as the
genuine break is not a convention: according to the
signal-to-noise criterion $\Delta_{\mathrm{between}}(\nu)/\Delta_{\mathrm{within}}(\nu)>1$ at a candidate~$\nu$ (the candidate-level form of \textup{(A3$^\ast$)} in Remark~\ref{rem:ridge_regularization_and_SNR}; the display~$\eqref{eq:Delta-within}$--$\eqref{eq:Delta-between}$ defines the $\nu$-dependent quantities),
equations~\eqref{equ:optimality_signal_and_noise}
and~\eqref{equ:optimality_signal_and_noise_at_05} give
\begin{align}
  \frac{\Delta_{\mathrm{between}} (\theta)}{\Delta_{\mathrm{within}} (\theta)} >1
  \quad \mbox{and} \quad
  \frac{\Delta_{\mathrm{between}} (\nu_\star)}
       {\Delta_{\mathrm{within}} (\nu_\star)} < 1.
\end{align}
Nevertheless, under the insufficient global SNR
in~\eqref{equ:counterexample_conditions}, the QML objective has its unique
global minimum at the non-dominant change~$\nu_\star$ for every
$\varepsilon>0$. Although~\textup{(A3*)} holds and a sufficiently large ridge
therefore makes~$\theta$ a local minimum, that local condition cannot force
global recovery.

More precisely, fix $\theta\in(0,\tfrac12)$ and write
$c_\theta:=\{2\sqrt{\theta(1-\theta)}\}^{-1}$. For every
$s\in(1,c_\theta)$, choosing $r=sK$ in
Proposition~\ref{prop:sharp-A4-fixed-theta} produces a scalar deterministic
array with
$\Delta_{\mathrm{between}}(\theta)/\Delta_{\mathrm{within}}(\theta)=s$
whose QML objective is minimized at~$\nu_\star$ for every ridge parameter.
Conversely,~\textup{(A4*)} and
Remark~\ref{rem:ridge_regularization_and_SNR} provide a ridge for which
Theorem~\ref{thm:QML-robust} guarantees global recovery whenever the ratio
exceeds~$c_\theta$. Hence the threshold
\[
\frac{\Delta_{\mathrm{between}}}{\Delta_{\mathrm{within}}}\;>\;\frac{1}{2\sqrt{\theta(1-\theta)}}
\]
is exact for a uniform ridge-optimized global-recovery
guarantee.
\end{remark}

\subsection{Proof heuristic and the role of the regularization parameter $\varepsilon$}
\label{subsec:heuristic-eps}

This subsection explains the heuristic behind the proof of Theorem~\ref{thm:classical_QML_consistency} and Theorem~\ref{thm:QML-robust}. For that purpose let us assume that
\begin{align}\label{equ:heuristic_QML_consistency}
    \widehat \Sigma_{a:b} \approx \Sigma_{a:b} \qquad \mbox{for all } 0 \leq a < b \leq 1.
\end{align}
This heuristic operates within the SLLN picture of Theorem~\ref{thm:classical_QML_consistency}, in which the population targets~$\Sigma_{a:b}$ are defined; the proof of the robust Theorem~\ref{thm:QML-robust} does not rely on these limiting objects and instead works directly with the empirical block moments~$\widehat\Sigma_{a:b}$ (see Section~\ref{sec:proof_consistency_finite_signal_to_noise}).
Strict concavity of the map $A\mapsto \log\det(A)$ then yields that for~$\nu \in (\theta,1)$ (the case~$\nu \in (0,\theta)$ is symmetric)
\begin{align}
    J_n(\nu) & \approx J_\infty (\nu) = \nu \log \det \left( \Sigma_{0:\nu} \right) + (1- \nu) \log \det \left(\Sigma_{\nu:1}\right) \\
        & = \nu \log \det \left( \frac{\theta}{\nu}\Sigma_{\leq} + \frac{\nu - \theta}{\nu} \Sigma_{>} \right) +  (1- \nu) \log \det \left(\Sigma_{>}\right)  \\
        & >   \theta \log \det \left( \Sigma_{\leq} \right) +   (1- \theta) \log \det \left(\Sigma_{>}\right) = J_\infty(\theta) \approx J_n(\theta).
\end{align}
Therefore, the breakpoint~$\theta$ is a global minimizer of~$J_\infty$, and as~$J_n$ converges to~$J_\infty$, the QML estimator should be consistent.

The proof of Theorem~\ref{thm:classical_QML_consistency} and Theorem~\ref{thm:QML-robust} then consists of quantifying the concavity estimate, more precisely by calculating and comparing two Jensen gaps. The main ingredient in Theorem~\ref{thm:QML-robust} is then the estimate (cf.~\eqref{eq:jn-lower-pre} below):
\[
\liminf_{n\to\infty}\bigl(J_n(\nu;\varepsilon)-J_n(\theta;\varepsilon)\bigr)
\ \ge\ \frac12\,T_+(\nu)\,(\nu-\theta),
\]
where
\begin{align}
\label{equ:curcial_estimate_robust_QML}
T_+(\nu)
=\frac{\theta}{\nu(M+\varepsilon)^2}\,\Delta_{\mathrm{between}}^2
-\frac{1-\nu}{(1-\theta)(m+\varepsilon)^2}\,\Delta_{\mathrm{within}}^2.
\end{align}
Thus the ``cusp sharpness'' at the true breakpoint~$\nu= \theta$ is governed by the contrast
between two quadratic contributions: a positive signal term of size $\Delta_{\mathrm{between}}^2/(M+\varepsilon)^2$ and a negative noise term of size $\Delta_{\mathrm{within}}^2/(m+\varepsilon)^2$.
The sufficient signal-to-noise conditions (A3)--(A5) are precisely those that force
$T_+(\nu)$, together with its left-hand analogue, to stay uniformly positive, either in a neighborhood of $\theta$ for a local minimum, or on
the whole open interval~$(0,1)$ for the global energy gap under~(A4).\medskip

\begin{remark}[The QML objective is biased toward a central split]\label{rem:centering_bias}
The estimate~\eqref{equ:curcial_estimate_robust_QML} also reveals a structural bias of the QML criterion toward a balanced split at the center~$\nu=\tfrac12$. Indeed, the geometric factor $\tfrac{1}{2\sqrt{\theta(1-\theta)}}$ entering~(A4) and~(A4$^\ast$), viewed as a function of the break fraction~$\theta$, is minimized at $\theta=\tfrac12$, where it equals~$1$, and grows without bound as $\theta\to0$ or $\theta\to1$. Hence a central breakpoint is the easiest to recover: a signal-to-noise ratio $\Delta_{\text{between}}/\Delta_{\text{within}}$ only slightly above~$1$ already certifies it as the global minimizer. An off-center breakpoint, by contrast, must clear the higher bar $\tfrac{1}{2\sqrt{\theta(1-\theta)}}>1$. Proposition~\ref{prop:sharp-A4-fixed-theta} shows that this geometric bar is sharp for ridge-optimized global recovery at every off-center break fraction (see also Remark~\ref{rem:sharp-A4-fixed-theta}).
\end{remark}

\begin{remark}[Correct choice of the regularization parameter~$\varepsilon>0$]
For a candidate~$\nu>\theta$, the best choice of~$\varepsilon >0$ is the one that maximizes~$T_+(\nu)$ in equation~\eqref{equ:curcial_estimate_robust_QML}; for~$\nu<\theta$, the left-hand analogue applies. The parameter $\varepsilon$ influences the size of $(M+\varepsilon)^{-2}$ and $(m+\varepsilon)^{-2}$. Let us first consider the two extremes. Choosing $\varepsilon \ll 1$ very small: If $m$ is small or zero, then
$(m+\varepsilon)^{-2}$ is very large and the negative noise contribution gets amplified, potentially
flattening or even moving the minimum away from~$\theta$. Choosing $\varepsilon \gg 1$ very large: If $\varepsilon$ dominates~$M$ and~$m$ then for a minimum of~$ \liminf_n J_n$ one only needs a stronger signal than noise, i.e.~$\Delta_{\text{between}} > \Delta_{\text{within}}$. The problem is that the contrast of the cusp scales like~$\frac{1}{\varepsilon}$, making the objective function nearly flat around~$\theta$ and therefore the QML estimator sensitive to noise. The best performance is expected at an intermediate ``sweet spot'' where
$\varepsilon$ is large enough to remove the small-eigenvalue bottleneck, but still small enough that between-regime separation produces a visibly
stronger contrast than within-regime heterogeneity.  
\end{remark}

\section{QML-breakpoint detection in factor models}\label{sec:breakpoint_detection_in_factor_models}

In this section we apply the abstract framework of Section~\ref{sec:abstract_QML} to high-dimensional data. The data is generated by a factor model whose loadings switch at the breakpoint (Definition~\ref{def:pervasive_factor_model} below). The detection strategy consists of a single preprocessing step:  we take the leading $q$ PCA scores of the $p$-dimensional observations and rescale them by $p^{-1/2}$, obtaining a $q$-dimensional signal series to which the framework of Section~\ref{sec:abstract_QML} applies.

The key structural hypothesis is that the factors are \emph{pervasive} --- each moves a macroscopic proportion of the coordinates, so the signal eigenvalues grow like the dimension~$p$ (see Remark~\ref{rem:pervasive_scale} below) --- while the error term only needs to be of lower order~$o(p)$ and may otherwise be completely general; the order-$p$ spectral gap then lets the Davis--Kahan inequality align the estimated projection with the pervasive directions. The main results are Proposition~\ref{prop:projected_slln_factor}, which verifies the assumptions of the abstract framework for the projected series --- vanishing within-regime noise, positive between-regime contrast --- and Corollary~\ref{thm:qml_consistency_factor}, which concludes consistency of the untrimmed QML breakpoint estimator for every ridge $\varepsilon>0$. Throughout we work in the high-dimensional high-sample-size (HDHSS) regime: $p=p_n\to\infty$ jointly with~$n$, while the numbers of factors stay fixed.\\

Let us start with the definition of a factor model and what it means to have a regime change at the breakpoint~$\tau$.

\begin{defn}[Regime-switching pervasive factor model]\label{def:pervasive_factor_model}
Let~$p=p_n$ denote the dimension of the observations, allowed to depend on the number~$n$ of observations. Let $X=(x_t)_{1\le t\le n}\in \R^{p\times n}$ denote the observed data. We say that $X$ follows a regime-switching pervasive factor model with breakpoint~$\tau = \lfloor \theta n\rfloor$, for some fixed~$\theta \in (0,1)$, if
\begin{equation}\label{eq:factor_model_controlled_error}
x_t=
\begin{cases}
\Lambda_{\le} f_t + e_t, & 1\le t\le \tau,\\[1mm]
\Lambda_{>} f_t + e_t, & \tau < t\le n,
\end{cases}
\end{equation}
where:
\begin{itemize}
    \item $\Lambda_{\le}=\Lambda_{\le,n}\in \R^{p\times q_{\le}}$ and $\Lambda_{>}=\Lambda_{>,n}\in \R^{p\times q_{>}}$ are deterministic loading matrices, with $q_{\le},q_{>}$ fixed, not depending on~$n$;
    \item $f_t$ are the factors, with $f_t \in \R^{q_\le}$ for $t\le\tau$ and $f_t\in\R^{q_>}$ for $t>\tau$;
    \item $e_t\in \R^p$ is a general error vector absorbing all non-pervasive, lower-order structure, e.g.~weak factors, idiosyncratic noise, dependence, or correlation with the factors.
\end{itemize}
We suppress the dependence of $p, \Lambda_\bullet$, $f_t$, $e_t$ on~$n$ whenever harmless. For convenience we write $s_t:=\Lambda_{\le} f_t$ for $t\le\tau$ and $s_t:=\Lambda_{>} f_t$ for $t>\tau$, so that $x_t=s_t+e_t$. For a macroscopic interval $I=I(a,b)$ we set
\[
S_{a:b}:=\frac1{|I|}\sum_{t\in I}s_ts_t^{\top},
\qquad
E_{a:b}:=\frac1{|I|}\sum_{t\in I}e_te_t^{\top}.
\]

\end{defn}

Classical results on QML-based breakpoint estimation in factor models --- most prominently Duan, Bai, and Han~\cite{DuanBaiHan2023} --- concern the unregularized objective~($\varepsilon=0$) under the structural assumptions on the idiosyncratic error~$e_t$ recalled in the introduction, which force a quantitative law of large numbers on the block second moments. In contrast (see Corollary~\ref{thm:qml_consistency_factor} below), we establish consistency allowing~$e_t$ to be \emph{completely general}, requiring only that its blockwise energy is negligible at the pervasive scale (see (F4) below). \\

We now describe the signal time series~$(y_t)_{1 \leq t \leq n}$ used for the detection of the breakpoint. As we cannot rely on domain knowledge for generic data, we define the signal space in a data-driven way (see also Remark~\ref{rem:data_driven_signal_space}). We fix a \emph{working rank} $q$, theoretically the number of pervasive directions (see Assumption~(F1) below), chosen in practice by thresholding the singular values (cf.\ Remark~\ref{rem:estimating_rank}). Let $\widehat U \in \R^{p\times q}$ denote the matrix whose columns are the leading $q$ eigenvectors of the empirical second moment matrix of the full dataset
\[
   \frac1n\sum_{t=1}^n x_t x_t^{\top}\in \R^{p\times p}.
\]
We define the  \emph{$p^{-1/2}$-rescaled PCA-score time series} by
\begin{equation}\label{eq:projected_signal_factor_section}
y_t:= \frac1{\sqrt p}\,\widehat U^{\top}x_t\in \R^q,\qquad 1\le t\le n.
\end{equation}
 The entries of~$y_t$ are the coordinates of the orthogonal projection~$\widehat U\widehat U^{\top}x_t$ in the basis~$\widehat U$.

For a macroscopic interval $I=I(a,b)$ the empirical second moment of  the rescaled PCA scores~$y_{t}$ on~$I$ is then defined as in Definition~\ref{def_singal_and_noise} as
\begin{equation}\label{eq:projected_block_moment_factor}
  \widehat \Sigma_{a:b} :=\frac1{|I|}\sum_{t\in I}y_ty_t^{\top} \in\R^{q\times q}.
\end{equation}
The normalization by $1/p$ makes $\widehat \Sigma_{a:b}$ an order-one $q\times q$ covariance (Remark~\ref{rem:pervasive_scale}).  The PCA step serves two purposes:
\begin{itemize}
    \item \emph{Dimension reduction:} going from~$p$ dimensions to the fixed dimension~$q$ makes the asymptotic analysis feasible even as~$p=p_n\to\infty$;
    \item \emph{Denoising:}  replacing $x_t$ by its orthogonal projection $\widehat U\widehat U^{\top}x_t$ onto the leading left-singular eigenspace of $X$ is a standard denoising device; the scores $\widehat U^{\top}x_t$ describe the data in its natural feature coordinates; this is the smoothing displayed in Figure~\ref{fig:gfc_smoothness}.
\end{itemize}

We now collect the asymptotic conditions needed to detect the breakpoint: pervasiveness, a factor SLLN, the regime change, and an asymptotic control on the error.  Stacking the two regimes, the data of a regime-switching model is itself a factor model with loadings~$[\,\Lambda_{\le}, \Lambda_{>}\,]$,
\[
X=[\,\Lambda_{\le}, \Lambda_{>}\,]\begin{bmatrix}F_{\le} & 0\\ 0 & F_{>}\end{bmatrix}+[\,E_{\le}, E_{>}\,],
\]
which is why the pervasiveness condition~(F1) below is formulated for the combined loadings.

\begin{assumption}\label{ass:assumptions_factor}
In addition to the model of Definition~\ref{def:pervasive_factor_model}, we assume:
\begin{itemize}
\item[(F1)] {\bf Pervasive loadings.}
   The combined normalized loading Gram matrix converges to a positive-semidefinite limit~$\Sigma_\Lambda \succeq 0$, i.e.
   \[
   \lim_{p \to \infty}\tfrac1p[\,\Lambda_{\le}\ \Lambda_{>}\,]^{\top}[\,\Lambda_{\le}\ \Lambda_{>}\,] = \Sigma_\Lambda \in \mathbb{R}^{(q_\le+q_>) \times (q_\le+q_>)}.
   \]
We define $q:=\operatorname{rank}(\Sigma_\Lambda) \leq q_\le+q_>$ as the pervasive rank. We have a strict inequality  exactly when the two regimes share at least one pervasive direction, i.e.\ a factor unchanged across the breakpoint.

\item[(F2)] {\bf Factor SLLN.}
   There exist positive definite matrices~$\Sigma_{F_{\le}} \succ 0$ and~$\Sigma_{F_{>}} \succ 0$ such that for every macroscopic interval $I(a,b)$ contained in the first regime and $I(c,d)$ contained in the second regime 
   \begin{align}
     \lim_{n \to \infty} \tfrac1{|I(a,b)|}\sum_{t\in I(a,b)} f_tf_t^{\top} = \Sigma_{F_{\le}} \quad \mbox{and} \quad   \lim_{n \to \infty} \tfrac1{|I(c,d)|}\sum_{t\in I(c,d)} f_tf_t^{\top} = \Sigma_{F_{>}}.
   \end{align}

\item[(F3)] {\bf Regime change in the leading part.}
  For $\bullet\in\{\le,>\}$ write $\Sigma_\bullet^{\rm sig}:=\Lambda_\bullet\Sigma_{F_\bullet}\Lambda_\bullet^{\top}$. We assume that the two regimes stay asymptotically separated on the pervasive scale, i.e.~
  \begin{align}\label{equ:regime_separation}
   \liminf_{n\to\infty}\ \frac{1}{p_n}\bigl\| \Sigma_\le^{\rm sig}-\Sigma_>^{\rm sig} \bigr\|_F>0 .
  \end{align}

\item[(F4)] {\bf Negligible error on the pervasive scale.}
   For every fixed macroscopic interval~$I(a,b)$, $ 0 \leq a < b \leq 1$, 
   \[
   \lim_{n \to \infty} \frac{1}{p_n} \bigl\| E_{a:b}\bigr\| = 0.
   \]
\end{itemize}
\end{assumption}

Let us explain the intuition of the pervasive order-$p$ scaling of the eigenvalues in Assumption (F1).
\begin{remark}[Pervasive scaling in financial factor models]\label{rem:pervasive_scale}
To illustrate pervasive scaling in financial factor models, let us consider the simplest model, where there is only a single market factor with loading vector $\Lambda=(1,\dots,1)^{\top}\in\R^{p\times1}$. This means that assets move as the market factor plus some idiosyncratic noise. Then
\[
\|\Lambda\|=\sup_{|f|=1}|\Lambda f|=\sqrt p,\qquad \|\Lambda\|^2=p,
\]
and the covariance contribution $\Lambda\Lambda^{\top}$ is the all-ones matrix, which has a single non-zero eigenvalue equal to~$p$. Thus a pervasive factor contributes an eigenvalue of order~$p$. It is important that all pervasive factors are caught when choosing the dimension~$q$, as we need in the proof of Proposition~\ref{prop:projected_slln_factor} that the associated~$q$-eigengap is also of order~$p$. 
\end{remark}

Let us now turn to the main result of this section.

\begin{proposition}\label{prop:projected_slln_factor}
Consider a regime-switching pervasive factor model~\eqref{eq:factor_model_controlled_error} satisfying \textup{(F1)--(F4)}, and let $(y_t)$ be the  $p^{-1/2}$-rescaled PCA-score series~\eqref{eq:projected_signal_factor_section}, with  block second moments $\widehat\Sigma_{a:b}$ given by~\eqref{eq:projected_block_moment_factor}. Then:
\begin{enumerate}[label=\textup{(\roman*)}]
  \item Assumption~\textup{(A1)} holds; that is, the
  break-anchored spectral ceiling in~\eqref{equ:def_M_m} satisfies $M<\infty$;
  \item the within- and between-regime separations of Section~\ref{sec:abstract_QML} satisfy
  \[
  \Delta_{\textup{within}}=0
  \quad\text{and}\quad
  \Delta_{\textup{between}}=  \liminf_{n\to\infty}\ \frac{1}{p_n}\bigl\| \Sigma_\le^{\rm sig}-\Sigma_>^{\rm sig}\bigr\|_F ,
\]
and therefore, together with~\textup{(F3)}, Assumption~\textup{(A4)} holds for every $\varepsilon>0$.
\end{enumerate}
\end{proposition}

The proof of Proposition~\ref{prop:projected_slln_factor} is given in Section~\ref{sec:proof_qml_consistency_factor}. A direct combination with Theorem~\ref{thm:QML-robust} yields consistency of the QML-breakpoint estimator:

\begin{corollary}[HDHSS consistency of QML-breakpoint estimation in pervasive factor models]\label{thm:qml_consistency_factor}

We consider a regime-switching pervasive factor model~\eqref{eq:factor_model_controlled_error} satisfying \textup{(F1)--(F4)}. Then, for every $\varepsilon>0$, every untrimmed regularized QML breakpoint estimator $\widehat\theta_n\in\arg\min_{\nu\in\mathcal D_n}J_n(\nu)$ satisfies $\widehat\theta_n \to \theta$.

\end{corollary}
\begin{remark}[Beyond infinite SNR]\label{rem:robustness_radius}
Assumption~(F4) makes the error budget vanish at the pervasive scale, and Proposition~\ref{prop:projected_slln_factor} accordingly lands in the infinite signal-to-noise regime $\Delta_{\mathrm{within}}=0$ of the abstract framework. Nothing in the proof is tied to this extreme case: all estimates are quantitative, and a sufficiently small non-vanishing error budget yields a positive $\Delta_{\mathrm{within}}$ and a reduced $\Delta_{\mathrm{between}}$ that still satisfy the finite signal-to-noise condition~\textup{(A4)}, so Theorem~\ref{thm:QML-robust} continues to give consistency; we do not track the explicit constants.
\end{remark}

\begin{remark}[Estimating the working rank]\label{rem:estimating_rank}
In practice the pervasive rank $q$ is unknown and is estimated from the data, for example by thresholding the singular values of the data matrix. The message is: overestimation is harmless, underestimation is not. If the estimated rank eventually stabilizes at a value $\ge q$, the projected series still contains the consistent $q$-dimensional pervasive component, while the extra directions contribute only $o(1)$ to the normalized block second moments, by \textup{(F4)} and the Davis--Kahan estimate; the proof of Proposition~\ref{prop:projected_slln_factor} goes through unchanged and the regularized QML breakpoint estimator (any $\varepsilon>0$) remains consistent. If the rank is underestimated, the projection can collapse the two regimes onto each other, and the break may become undetectable.
\end{remark}

\section{Proofs}\label{sec:proofs}

The conceptual motivation for the arguments, including the two competing
Jensen gaps and the role of the ridge parameter, was given in
Section~\ref{subsec:heuristic-eps}.  Here we supply the technical details in
two parts.  Subsection~\ref{sec:proof_auxiliary_results} collects the
auxiliary tools: Lemmas~\ref{lem:opnorm_Sigma_ab}--\ref{lem:Lipschitz_J}
provide the boundedness, continuity, and endpoint control needed to pass from
pointwise estimates to minimizers of~$J_n$, while
Lemma~\ref{lem:jensen_gap_logdet} quantifies the curvature of~$\log\det$.
Subsection~\ref{sec:proof_main_results} proves the main results:
Theorem~\ref{thm:classical_QML_consistency} gives classical SLLN-based
consistency, Theorem~\ref{thm:QML-robust} gives finite-signal-to-noise
consistency, Proposition~\ref{prop:sharp-A4-fixed-theta} establishes
sharpness of the global threshold, and
Proposition~\ref{prop:projected_slln_factor} verifies the abstract assumptions
for the pervasive factor model.  Combining the last proposition with
Theorem~\ref{thm:QML-robust} yields the factor-model consistency result,
Corollary~\ref{thm:qml_consistency_factor}.

\subsection{Auxiliary results}\label{sec:proof_auxiliary_results}

\begin{lem}[Block additivity of empirical second moments]\label{lem:block_additivity}
Let $(z_t)_{1\le t\le n}$ be vectors in $\mathbb R^d$ and, for a nonempty index set $I\subset\{1,\dots,n\}$, set $\widehat\Sigma_I:=\frac1{|I|}\sum_{t\in I}z_tz_t^{\top}$. If $I,J$ are nonempty and disjoint, then
\begin{align}\label{eq:block_additivity}
\widehat\Sigma_{I\cup J}
=\frac{|I|}{|I|+|J|}\,\widehat\Sigma_I+\frac{|J|}{|I|+|J|}\,\widehat\Sigma_J ,
\end{align}
and consequently
\begin{align}\label{eq:block_additivity_solved}
\widehat\Sigma_J
=\frac{|I|+|J|}{|J|}\,\widehat\Sigma_{I\cup J}-\frac{|I|}{|J|}\,\widehat\Sigma_I ,
\qquad
\bigl\|\widehat\Sigma_J-\widehat\Sigma_{I\cup J}\bigr\|
\le\frac{|I|}{|J|}\Bigl(\bigl\|\widehat\Sigma_{I\cup J}\bigr\|+\bigl\|\widehat\Sigma_I\bigr\|\Bigr).
\end{align}
\end{lem}
\begin{proof}
Multiplying out, $(|I|+|J|)\,\widehat\Sigma_{I\cup J}=\sum_{t\in I\cup J}z_tz_t^{\top}=|I|\,\widehat\Sigma_I+|J|\,\widehat\Sigma_J$, which is~\eqref{eq:block_additivity}. Solving for $\widehat\Sigma_J$ gives the identity in~\eqref{eq:block_additivity_solved}; subtracting $\widehat\Sigma_{I\cup J}$ from it yields $\widehat\Sigma_J-\widehat\Sigma_{I\cup J}=\frac{|I|}{|J|}\bigl(\widehat\Sigma_{I\cup J}-\widehat\Sigma_I\bigr)$, and the triangle inequality gives the norm bound.
\end{proof}

\begin{lem}[Bounds on empirical second moments]
\label{lem:opnorm_Sigma_ab}
 If $0\le a<b\le 1$, then
\begin{equation}\label{eq:opnorm_left_block}
\limsup_{n \to \infty} \|\widehat\Sigma_{a:b}\| \leq \Delta_{\text{within}} + M.
\end{equation}
\end{lem}

\begin{proof}[Proof of Lemma~\ref{lem:opnorm_Sigma_ab}]
We first prove~\eqref{eq:opnorm_left_block} for~$0 \leq a  < b \leq \theta$ and for~$\theta \leq a  < b \leq 1$.
The general statement for~$0 \leq a < \theta < b \leq 1$ follows then from the triangle inequality and  Lemma~\ref{lem:block_additivity}, applied to $I(a,\theta)$ and $I(\theta,b)$:
\begin{align}
    \widehat\Sigma_{a:b} = \frac{\lfloor \theta n \rfloor - \lfloor a n \rfloor}{\lfloor b n \rfloor - \lfloor a n \rfloor} \widehat\Sigma_{a:\theta} +  \frac{ \lfloor b n \rfloor  - \lfloor \theta n \rfloor }{\lfloor b n \rfloor - \lfloor a n \rfloor} \widehat\Sigma_{\theta:b} .
\end{align}
Hence, let us only consider the case $0 \leq a  < b \leq \theta$. We observe that
\begin{align}
       \|  \widehat\Sigma_{a:b} \| \leq \| \widehat\Sigma_{a:b}  -  \widehat\Sigma_{0:\theta}  \| + \|  \widehat\Sigma_{0:\theta}  \|,  
\end{align}
from which the desired inequality~\eqref{eq:opnorm_left_block} follows by unwinding the involved definitions.
\end{proof}

\begin{lem}[Modulus of continuity for empirical second moments]
\label{lem:modcont_cov}
For all  $0 < \nu_1 < \nu_2 < 1$,
\begin{align}
\label{eq:modcont_left}
\limsup_{n \to \infty}  \|\widehat\Sigma_{0:\nu_2 }-\widehat\Sigma_{0: \nu_1 } \|  \leq \frac{2}{\nu_2}  ( \Delta_{\text{within}} + M)  |\nu_2-\nu_1|,
\end{align}
and 
\begin{align}
\label{eq:modcont_right}
 \limsup_{n \to \infty} \|\widehat\Sigma_{\nu_2:1}-\widehat\Sigma_{\nu_1 :1} \| \leq \frac{2}{(1 - \nu_1)}  ( \Delta_{\text{within}} + M) |\nu_2-\nu_1|.
\end{align}

\end{lem}

\begin{proof}[Proof of Lemma~\ref{lem:modcont_cov}]
We only consider~$ \|\widehat\Sigma_{0:\nu_2 }-\widehat\Sigma_{0: \nu_1 } \|$ as the term $\|\widehat\Sigma_{\nu_2:1}-\widehat\Sigma_{\nu_1 :1} \|$ can be estimated in the same way.
Set $k_i:=\lfloor\nu_i n\rfloor$, $i=1,2$. For all large enough~$n$,
$1\leq k_1<k_2\leq n-1$, and  \eqref{eq:block_additivity} gives
\[
\widehat\Sigma_{0: \nu_2}
=
\frac{k_1}{k_2}\widehat\Sigma_{0: \nu_1}
+\frac{k_2-k_1}{k_2}\widehat\Sigma_{\nu_1:\nu_2}.
\]
Hence
\[
\widehat\Sigma_{0: \nu_2}-\widehat\Sigma_{0: \nu_1}
=
\frac{k_2-k_1}{k_2}
\Big(\widehat\Sigma_{\nu_1:\nu_2}-\widehat\Sigma_{0: \nu_1}\Big).
\]
Taking operator norms and using the triangle inequality gives
\begin{align}
\big\|\widehat\Sigma_{0: \nu_2}-\widehat\Sigma_{0:  \nu_1}\big\|
& \leq
\frac{k_2-k_1}{k_2}\Big(\|\widehat\Sigma_{0: \nu_1}\|+\|\widehat\Sigma_{\nu_1:\nu_2}\|\Big).
\end{align}
Since $(k_2-k_1)/k_2\to(\nu_2-\nu_1)/\nu_2$, applying
Lemma~\ref{lem:opnorm_Sigma_ab} yields the desired estimate.
\end{proof}

\begin{lem}[Vanishing energy in short windows]\label{lem:boundary_energy}
Assume~\ref{ass:bounded_second_moment} and $\Delta_{\mathrm{within}}<\infty$. Then, for every fixed $\bar\nu\in[0,1]$ and $\rho\in(0,1)$, setting $a_\rho:=(\bar\nu-\rho)\vee0$ and $b_\rho:=(\bar\nu+\rho)\wedge1$,
\begin{align}\label{eq:local_energy}
\limsup_{n\to\infty}
\frac1n\sum_{t=\lfloor a_\rho n\rfloor+1}^{\lfloor b_\rho n\rfloor}\|y_t\|^2
\leq 2\rho\,q\,\big(M+\Delta_{\mathrm{within}}\big).
\end{align}
In particular, combining the windows at $\bar\nu=0$ and $\bar\nu=1$ and sending $\rho\downarrow0$,
\begin{align}\label{eq:boundary_energy}
\lim_{\rho\downarrow0}\limsup_{n\to\infty}
\frac1n\left(\sum_{t=1}^{\lfloor \rho n\rfloor}\|y_t\|^2
+\sum_{t=n-\lfloor \rho n\rfloor+1}^{n}\|y_t\|^2\right)=0.
\end{align}
Consequently, if $k_n,\ell_n\in\{0,\ldots,n\}$ satisfy
$k_n/n\to0$ and $\ell_n/n\to0$, then
\begin{align}\label{eq:boundary_energy_sequential}
\frac1n\left(
\sum_{t=1}^{k_n}\|y_t\|^2
+\sum_{t=n-\ell_n+1}^{n}\|y_t\|^2
\right)\longrightarrow0.
\end{align}
\end{lem}

\begin{proof}[Proof of Lemma~\ref{lem:boundary_energy}]
Fix $\bar\nu\in[0,1]$ and $\rho\in(0,1)$. By construction $0\leq a_\rho<b_\rho\leq1$ and $b_\rho-a_\rho\leq2\rho$. Hence
\[
\frac1n\sum_{t=\lfloor a_\rho n\rfloor+1}^{\lfloor b_\rho n\rfloor}\|y_t\|^2
=\frac{\lfloor b_\rho n\rfloor-\lfloor a_\rho n\rfloor}{n}\,
\Tr\widehat\Sigma_{a_\rho:b_\rho}
\leq
\Big(2\rho+\frac1n\Big)\,q\,\big\|\widehat\Sigma_{a_\rho:b_\rho}\big\|.
\]
 Lemma~\ref{lem:opnorm_Sigma_ab} yields
\[
\limsup_{n\to\infty}
\frac1n\sum_{t=\lfloor a_\rho n\rfloor+1}^{\lfloor b_\rho n\rfloor}\|y_t\|^2
\leq
2\rho\,q\,\big(M+\Delta_{\mathrm{within}}\big),
\]
which is~\eqref{eq:local_energy}. For~\eqref{eq:boundary_energy} it suffices to note that the two boundary index sets are contained in the windows of~\eqref{eq:local_energy} with $\bar\nu=0$ and $\bar\nu=1$, respectively, and to send $\rho\downarrow0$. Finally, for every fixed $\rho>0$, the moving boundary sets in the sequential conclusion above are contained in those same fixed $\rho$-windows for all large enough~$n$. Taking $\limsup_n$ and then sending $\rho\downarrow0$ proves that conclusion.

\end{proof}

\begin{lem}[Short endpoint windows under ridge regularization]\label{lem:boundary_logdet}
Assume~\ref{ass:bounded_second_moment}, $\Delta_{\mathrm{within}}<\infty$, and $\varepsilon>0$. If $\nu_n\in\mathcal D_n$ and $\nu_n\to0$, then $\nu_n\bigl|\log\det\big(\widehat\Sigma_{0:\nu_n}+\varepsilon I_q\big)\bigr|\longrightarrow0$; if $\nu_n\in\mathcal D_n$ and $\nu_n\to1$, then $(1-\nu_n)\bigl|\log\det\big(\widehat\Sigma_{\nu_n:1}+\varepsilon I_q\big)\bigr|\longrightarrow0$.
\end{lem}

\begin{proof}[Proof of Lemma~\ref{lem:boundary_logdet}]
We prove the left endpoint statement; the right endpoint is identical. Write $\nu_n=k_n/n$ with $1\leq k_n\leq n-1$. The lower spectral bound $\widehat\Sigma_{0:\nu_n}+\varepsilon I_q\succeq \varepsilon I_q$ gives
\[
\nu_n\log\det\big(\widehat\Sigma_{0:\nu_n}+\varepsilon I_q\big)
\geq q\nu_n\log \varepsilon \longrightarrow0.
\]
For the upper bound,  equation~\eqref{eq:boundary_energy_sequential} of Lemma~\ref{lem:boundary_energy} gives
\[
e_n:=\frac1n\sum_{t=1}^{k_n}\|y_t\|^2
=\nu_n\,\Tr\widehat\Sigma_{0:\nu_n}
\longrightarrow0.
\]
Overestimating the eigenvalues by the trace, we get\[
\log\det\big(\widehat\Sigma_{0:\nu_n}+\varepsilon I_q\big)
\leq q\log\left(\varepsilon+\Tr\widehat\Sigma_{0:\nu_n}\right)
=q\log\left(\varepsilon+\frac{e_n}{\nu_n}\right).
\]
Overestimating $\log\left(\varepsilon+\frac{e_n}{\nu_n}\right)$ by linearization at $\varepsilon$ shows that the product
$\nu_n\log(\varepsilon+e_n/\nu_n)$ tends to zero since $e_n\to0$. This proves the claim.
\end{proof}

\begin{lem}[Lipschitz continuity of $\log \det$]\label{lem:lipischitz_log_det}
    Let us consider symmetric positive definite matrices $A,B \in \mathbb{R}^{q \times q}$ and define
    \begin{align}
        C(t):=(1-t)A + tB,\qquad  \mbox{for } t\in[0,1].
    \end{align}
    Assume that~$\sup_{t \in [0,1]} \|C(t)^{-1} \| < \infty$. Then
    \begin{align}\label{equ:lipischitz_log_det}
        & \left| \log \det (A) - \log \det (B)\right|  \\
        & \ \leq \sqrt{q} \ \sup_{t \in [0,1]} \|C(t)^{-1} \| \  \| A-B \|_F \leq q \ \sup_{t \in [0,1]} \|C(t)^{-1} \| \ \| A- B\|.
    \end{align}
\end{lem}

\begin{proof}[Proof of Lemma~\ref{lem:lipischitz_log_det}]
Set $\phi(t):=\log\det C(t)$. Jacobi's differential identity $d\,(\log\det X)=\Tr (X^{-1}dX)$ (see e.g.~\cite[Section~A.4.1]{BoydVandenberghe2004}) gives $\phi'(t)=\Tr\bigl(C(t)^{-1}(B-A)\bigr)$, so by the fundamental theorem of calculus
\[
\log\det(A)-\log\det(B)
=\phi(0)-\phi(1)=\int_0^1 \Tr \!\bigl(C(t)^{-1}(A-B)\bigr)\,dt .
\]
Bounding the integrand by Cauchy--Schwarz for the Frobenius inner product, $|\Tr(C(t)^{-1}(A-B))|\le \|C(t)^{-1}\|_F\,\|A-B\|_F\leq \sqrt{q}\,\sup_{t\in[0,1]}\|C(t)^{-1}\|\,\|A-B\|_F$, proves the first inequality of~\eqref{equ:lipischitz_log_det}; the second follows from $\|A-B\|_F\le \sqrt{q}\,\|A-B\|$.
\end{proof}

\begin{lem}[Sequential stability of the QML objective]\label{lem:seq_continuity_J}
Assume~\ref{ass:bounded_second_moment}, $\Delta_{\mathrm{within}}<\infty$, and that either
$\varepsilon>0$ or~\ref{ass:positive_definiteness} holds. If
$\nu_n\to\bar\nu\in(0,1)$, then $J_n(\nu_n)-J_n(\bar\nu)\longrightarrow0$.
\end{lem}

\begin{proof}[Proof of Lemma~\ref{lem:seq_continuity_J}]
Set
\[
k_n:=\lfloor \nu_n n\rfloor,
\qquad
\bar k_n:=\lfloor \bar\nu n\rfloor,
\qquad
r_n:=k_n\wedge\bar k_n,
\qquad
s_n:=k_n\vee\bar k_n.
\]
For all large enough~$n$, $1\leq r_n\leq s_n\leq n-1$.
We first compare the two prefix blocks.  Lemma~\ref{lem:block_additivity}, applied to the shorter prefix and the intervening window $\{r_n+1,\ldots,s_n\}$ (identity~\eqref{eq:block_additivity} if $k_n\geq\bar k_n$, identity~\eqref{eq:block_additivity_solved} otherwise), gives the common bound
\begin{align*}
\big\|\widehat\Sigma_{0:\nu_n}-\widehat\Sigma_{0:\bar\nu}\big\|
&\leq
\frac{|k_n-\bar k_n|}{r_n}
\big\|\widehat\Sigma_{0:\bar\nu}\big\|
+\frac1{r_n}\sum_{t=r_n+1}^{s_n}\|y_t\|^2 .
\end{align*}
Here $r_n/n\to\bar\nu$ and $|k_n-\bar k_n|/n\to0$, while
Lemma~\ref{lem:opnorm_Sigma_ab} bounds the fixed-block norm. Moreover, for
every fixed~$\rho>0$, the window $\{r_n+1,\ldots,s_n\}$ is contained, for all
large~$n$, in the $\rho$-window around~$\bar\nu$. Hence
\eqref{eq:local_energy} and then $\rho\downarrow0$ show that the second term
also tends to zero. Therefore
\[
\big\|\widehat\Sigma_{0:\nu_n}-\widehat\Sigma_{0:\bar\nu}\big\|
\longrightarrow0.
\]
The same argument applied to suffix blocks gives
\[
\big\|\widehat\Sigma_{\nu_n:1}-\widehat\Sigma_{\bar\nu:1}\big\|
\longrightarrow0.
\]

By Definition~\ref{def_singal_and_noise}, each of
$\widehat\Sigma_{0:\bar\nu}$ and
$\widehat\Sigma_{\bar\nu:1}$ is either a member of
$\mathcal I_\theta^{\rm cut}$ or a convex combination of two members of that family  (by~\eqref{eq:block_additivity}).
Consequently,
\[
\liminf_{n\to\infty}\lambda_{\min}(\widehat\Sigma_{0:\bar\nu})\geq m,
\qquad
\liminf_{n\to\infty}\lambda_{\min}(\widehat\Sigma_{\bar\nu:1})\geq m.
\]

We now turn these facts into a statement about the QML objectives. Abbreviate
\[
A_n:=\widehat\Sigma_{0:\nu_n}+\varepsilon I_q,\qquad
\bar A_n:=\widehat\Sigma_{0:\bar\nu}+\varepsilon I_q,\qquad
B_n:=\widehat\Sigma_{\nu_n:1}+\varepsilon I_q,\qquad
\bar B_n:=\widehat\Sigma_{\bar\nu:1}+\varepsilon I_q,
\]
so that $J_n(\nu_n)=\nu_n\log\det A_n+(1-\nu_n)\log\det B_n$ and $J_n(\bar\nu)=\bar\nu\log\det\bar A_n+(1-\bar\nu)\log\det\bar B_n$.

\emph{Common spectral bounds.} By the two displays above,
$\liminf_{n}\lambda_{\min}(\bar A_n)\geq m+\varepsilon$ and
$\liminf_{n}\lambda_{\min}(\bar B_n)\geq m+\varepsilon$, where $m+\varepsilon>0$
by~(PD) if $\varepsilon=0$ and trivially if $\varepsilon>0$. Since
$\|A_n-\bar A_n\|\to0$ and $\|B_n-\bar B_n\|\to0$, the same lower bounds hold
for $A_n$ and $B_n$. Lemma~\ref{lem:opnorm_Sigma_ab} gives
$\limsup_{n}\|\bar A_n\|\leq M+\Delta_{\mathrm{within}}+\varepsilon$, and the
same bound holds for $\bar B_n$, and hence also for $A_n$ and $B_n$.
Consequently, for all large enough~$n$, every convex combination
$C(t)=(1-t)A_n+t\bar A_n$, $t\in[0,1]$, satisfies
\[
\tfrac12(m+\varepsilon)\,I_q
\;\preceq\; C(t)\;\preceq\;
\big(M+\Delta_{\mathrm{within}}+\varepsilon+1\big)\,I_q,
\]
and the same holds for the pair $B_n,\bar B_n$.

\emph{Control of the log-determinants.} Lemma~\ref{lem:lipischitz_log_det} with
$\sup_{t\in[0,1]}\|C(t)^{-1}\|\leq 2/(m+\varepsilon)$ yields
\begin{align*}
|\log\det A_n-\log\det\bar A_n|
&\leq\frac{2q}{m+\varepsilon}\,\|A_n-\bar A_n\|\longrightarrow0,\\
|\log\det B_n-\log\det\bar B_n|
&\leq\frac{2q}{m+\varepsilon}\,\|B_n-\bar B_n\|\longrightarrow0,
\end{align*}
while the two-sided spectral bounds give, for all large~$n$,
\begin{align*}
|\log\det\bar A_n|+|\log\det\bar B_n|
&\leq 2q\max\Big\{\big|\log\tfrac{m+\varepsilon}{2}\big|,\
\log\big(M+\Delta_{\mathrm{within}}+\varepsilon+1\big)\Big\}\\
&=:K<\infty.
\end{align*}

\emph{Convergence of the objectives.} Writing
\begin{align*}
J_n(\nu_n)-J_n(\bar\nu)
&=\nu_n\big(\log\det A_n-\log\det\bar A_n\big)
+(1-\nu_n)\big(\log\det B_n-\log\det\bar B_n\big)\\
&\quad+(\nu_n-\bar\nu)\big(\log\det\bar A_n-\log\det\bar B_n\big),
\end{align*}
the first two terms tend to zero by the previous step, and the last one is
bounded by $|\nu_n-\bar\nu|\,K\to0$. This proves the claim.

\end{proof}

\begin{lem}[Stability of macroscopic blocks near the endpoints]\label{lem:endpoint_large_block}
  Assume~\ref{ass:bounded_second_moment} and $\Delta_{\mathrm{within}}<\infty$. \newline
  It holds
  \begin{align}
    \lim_{\mathcal D_n \ni \nu_n\to 1} \left\|\widehat\Sigma_{\theta:\nu_n}-\widehat\Sigma_{\theta:1}\right\|  = 0 \quad \mbox{and} \quad     \lim_{\mathcal D_n \ni \nu_n\to 0}  \left\|\widehat\Sigma_{\nu_n:\theta}-\widehat\Sigma_{0:\theta}\right\| = 0.
  \end{align}
Consequently, if $\varepsilon>0$ and $\nu_n\to1$, then a direct application of Lemma~\ref{lem:lipischitz_log_det} yields
\[
\left|
\log\det\big(\widehat\Sigma_{\theta:\nu_n}+\varepsilon I_q\big)
-
\log\det\big(\widehat\Sigma_{\theta:1}+\varepsilon I_q\big)
\right|
\leq
\frac{q}{\varepsilon}
\left\|\widehat\Sigma_{\theta:\nu_n}-\widehat\Sigma_{\theta:1}\right\|
\longrightarrow0,
\]
and symmetrically, if $\varepsilon>0$ and $\nu_n\to0$, the same bound holds for the pair $\widehat\Sigma_{\nu_n:\theta}$, $\widehat\Sigma_{0:\theta}$.
\end{lem}

\begin{proof}[Proof of Lemma~\ref{lem:endpoint_large_block}]
We only prove the case $\nu_n\to1$. For all large enough $n$ we have $\nu_n>\theta$. Let
\[
N_n:=n-\lfloor \theta n\rfloor,\qquad
 \ell_n:=n-\lfloor \nu_n n\rfloor,
\qquad
B_n:=\sum_{t=n-\ell_n+1}^{n}y_t y_t^{\top}.
\]
Then $\ell_n/N_n\to0$, and  by~\eqref{eq:block_additivity_solved} with $I=\{n-\ell_n+1,\ldots,n\}$ and $J=I(\theta,\nu_n)$,
\[
\widehat\Sigma_{\theta:\nu_n}
=\frac{N_n}{N_n-\ell_n}\widehat\Sigma_{\theta:1}
-\frac{1}{N_n-\ell_n}B_n .
\]
Therefore,
\[
\left\|\widehat\Sigma_{\theta:\nu_n}-\widehat\Sigma_{\theta:1}\right\|
\leq
\frac{\ell_n}{N_n-\ell_n}\|\widehat\Sigma_{\theta:1}\|
+\frac{\Tr B_n}{N_n-\ell_n}.
\]
The first term tends to zero by Assumption~\ref{ass:bounded_second_moment}, and the second term tends to zero by~\eqref{eq:boundary_energy_sequential}, since $\ell_n/n\to0$ and $(N_n-\ell_n)/n\to1-\theta>0$. The left endpoint is symmetric.
The log-determinant estimates are a direct application of Lemma~\ref{lem:lipischitz_log_det}: for the regularized blocks, every interpolant $(1-t)\bigl(\widehat\Sigma_{\theta:\nu_n}+\varepsilon I_q\bigr)+t\bigl(\widehat\Sigma_{\theta:1}+\varepsilon I_q\bigr)$ dominates $\varepsilon I_q$, so the supremum in~\eqref{equ:lipischitz_log_det} is at most $\varepsilon^{-1}$, and the claimed bounds follow; the case $\nu_n\to0$ is identical.
\end{proof}

\begin{lem}[Lipschitz continuity of the QML-objective functions]\label{lem:Lipschitz_J}
Assume~\ref{ass:bounded_second_moment},
$\Delta_{\mathrm{within}}<\infty$, and that either $\varepsilon>0$
or~\ref{ass:positive_definiteness} holds.
There exists a constant~$C <\infty$ only depending on~$q$, $\varepsilon$,~$m$,~$M$ and~$\Delta_{\text{within}}$ such that for all~$0 < \nu_1 < \nu_2 <1$
\begin{align}\label{equ:J_n_Lipschitz}
    \limsup_{n \to \infty}|J_n(\nu_1) - J_n(\nu_2)| \leq C |\nu_2 - \nu_1|
\end{align}
If, in addition, Assumption~\ref{ass:SLLN_in_each_regime} holds, then
\begin{align}\label{equ:J_infty_Lipschitz}
    |J_\infty(\nu_1) - J_\infty(\nu_2)| \leq C |\nu_2 - \nu_1|.
\end{align}
For every $\eta\in(0,1/2)$, the estimate~\eqref{equ:J_n_Lipschitz} holds uniformly on $[\eta,1-\eta]$ in the sense that
\begin{align}\label{equ:J_n_Lipschitz_uniform}
\limsup_{n\to\infty}\;
\sup_{\nu_1,\nu_2\in[\eta,1-\eta]}
\Big(|J_n(\nu_1)-J_n(\nu_2)|-C|\nu_2-\nu_1|\Big)\leq0.
\end{align}
In contrast to~\eqref{equ:J_n_Lipschitz}, the uniform version~\eqref{equ:J_n_Lipschitz_uniform} may be applied to points $\nu_i=\nu_i(n)$ that vary with $n$, as they do in the proof of Theorem~\ref{thm:classical_QML_consistency} below.
\end{lem}

\begin{proof}[Proof of Lemma~\ref{lem:Lipschitz_J}]
Argument for~\eqref{equ:J_n_Lipschitz}: For every fixed
$\nu\in(0,1)$, Definition~\ref{def_singal_and_noise} shows that each of
$\widehat\Sigma_{0:\nu}$ and $\widehat\Sigma_{\nu:1}$ is either the empirical
second moment of an interval in~$\mathcal I_\theta^{\rm cut}$ or a convex
combination of empirical second moments of two intervals in
$\mathcal I_\theta^{\rm cut}$. Hence their lower spectral limits are at least~$m$.
Lemma~\ref{lem:opnorm_Sigma_ab} bounds their upper spectral limits by
$M+\Delta_{\mathrm{within}}$. Consequently, for all sufficiently large~$n$,
their regularized eigenvalues lie in
\[
\left[\frac{m+\varepsilon}{2},\
2\bigl(M+\Delta_{\mathrm{within}}+\varepsilon\bigr)\right].
\]
To simplify calculations we introduce the notation
\[
L_n(\nu):=\log\det\big(\widehat\Sigma_{0:\nu}+\varepsilon I_q\big),
\qquad
R_n(\nu):=\log\det\big(\widehat\Sigma_{\nu:1}+\varepsilon I_q\big),
\]
which allows us to write
\begin{align}
    J_n(\nu)=\nu L_n(\nu)+(1-\nu)R_n(\nu).
\end{align}
Adding and subtracting terms gives
\begin{align*}
J_n(\nu_2)-J_n(\nu_1)
&= \nu_2 L_n(\nu_2)-\nu_1 L_n(\nu_1) + (1-\nu_2)R_n(\nu_2)-(1-\nu_1)R_n(\nu_1)\\
&=(\nu_2-\nu_1)L_n(\nu_2)+\nu_1\big(L_n(\nu_2)-L_n(\nu_1)\big)\\
&\quad-(\nu_2-\nu_1)R_n(\nu_2)+(1-\nu_1)\big(R_n(\nu_2)-R_n(\nu_1)\big).
\end{align*}
Hence, by the triangle inequality, we obtain
\begin{align}\label{eq:Jn-split}
|J_n(\nu_2)-J_n(\nu_1)| & \leq |\nu_2-\nu_1|\big(|L_n(\nu_2)|+|R_n(\nu_2)|\big) \\
  & \qquad + \nu_1|L_n(\nu_2)-L_n(\nu_1)| + (1-\nu_1)|R_n(\nu_2)-R_n(\nu_1)|.
\end{align}
As the determinant equals the product of eigenvalues, set
\[
C_{\log}:=
2q\max\left\{
\left|\log\!\left(2(M+\Delta_{\mathrm{within}}+\varepsilon)\right)\right|,
\left|\log\!\left(\frac{m+\varepsilon}{2}\right)\right|
\right\}.
\]
The preceding two-sided spectral bound gives
\begin{align}\label{equ:bound_1}
|L_n(\nu_2)|+|R_n(\nu_2)|\leq C_{\log}.
\end{align}
Setting~$A:=\widehat\Sigma_{0:\nu_2}$ and
$B:=\widehat\Sigma_{0:\nu_1}$ and using Lemma~\ref{lem:lipischitz_log_det}, we observe that for large enough~$n$
\begin{align}
|L_n(\nu_2)-L_n(\nu_1)|
& \leq q \sup\limits_{t\in [0,1]} \|(At+(1-t)B+\varepsilon I_q)^{-1}\|\,\|A-B\| \\
& \leq 4q\frac{M+\Delta_{\mathrm{within}}}{\nu_2(m+\varepsilon)}|\nu_2-\nu_1|, \label{equ:bound_2}
\end{align}
where we used the lower bound on the eigenvalues in combination with Lemma~\ref{lem:modcont_cov}. Similarly, we get
\begin{equation}
|R_n(\nu_2)-R_n(\nu_1)|
\leq 4q\frac{M+\Delta_{\mathrm{within}}}{(1-\nu_1)(m+\varepsilon)}|\nu_2-\nu_1|, \label{equ:bound_3}
\end{equation}
When we insert the bounds~\eqref{equ:bound_1},\eqref{equ:bound_2}, and~\eqref{equ:bound_3} into~\eqref{eq:Jn-split}, we get
\[
\begin{aligned}
|J_n(\nu_2)-J_n(\nu_1)|
&\leq |\nu_2-\nu_1|\Bigg(
4q
\frac{M+\Delta_{\mathrm{within}}}{m+\varepsilon}
\left(\frac{\nu_1}{\nu_2}+1\right)\\
&\qquad\qquad
+C_{\log}\Bigg).
\end{aligned}
\]
Incorporating the fact that \(\nu_1/\nu_2\leq 1\), we get the desired estimate~\eqref{equ:J_n_Lipschitz}, i.e.
\[
|J_n(\nu_2)-J_n(\nu_1)|\leq |\nu_2-\nu_1|\,C,
\]
where \(C\) has no dependence on \(\nu_2,\nu_1\).

The argument for~\eqref{equ:J_infty_Lipschitz} is similar after observing
under Assumption~\ref{ass:SLLN_in_each_regime} that the eigenvalues of
$\Sigma_{0:\nu}+\varepsilon I_q$ and $\Sigma_{\nu:1}+\varepsilon I_q$
belong to $[m+\varepsilon,M+\varepsilon]$.

Argument for~\eqref{equ:J_n_Lipschitz_uniform}: Suppose the claim fails for some $\eta\in(0,1/2)$. Then there exist $\alpha>0$, a subsequence (not relabelled), and points $\nu_1(n),\nu_2(n)\in[\eta,1-\eta]$ such that
\[
|J_n(\nu_1(n))-J_n(\nu_2(n))|
\geq C|\nu_2(n)-\nu_1(n)|+\alpha.
\]
By compactness of $[\eta,1-\eta]$, after passing to a further subsequence,
$\nu_i(n)\to\bar\nu_i\in[\eta,1-\eta]$ for $i=1,2$.
Lemma~\ref{lem:seq_continuity_J} gives
\[
J_n(\nu_i(n))-J_n(\bar\nu_i)\longrightarrow0,\qquad i=1,2.
\]
It follows that
\[
\limsup_{n\to\infty}|J_n(\nu_1(n))-J_n(\nu_2(n))|
\leq
\limsup_{n\to\infty}|J_n(\bar\nu_1)-J_n(\bar\nu_2)|
\leq C|\bar\nu_2-\bar\nu_1|,
\]
where the last inequality is~\eqref{equ:J_n_Lipschitz}. This contradicts the
preceding display because
$|\nu_2(n)-\nu_1(n)|\to|\bar\nu_2-\bar\nu_1|$.
\end{proof}

\begin{lem}[Jensen gap for $\log\det$]\label{lem:jensen_gap_logdet}
Let $A,B\in\bR^{q\times q}$ be symmetric positive definite matrices, $\lambda \in [0,1]$, and $C(t):=(1-t)A+tB$ for $t\in[0,1]$.
 Define the Jensen gap as
\begin{align}
    G(A,B;\lambda) : = \log\det\left(C(\lambda)\right)
-\Big((1-\lambda)\log\det A+\lambda\log\det B\Big).
\end{align}
If there are~$\mu_{\min},\mu_{\max}\in\mathbb{R}^+$ such that for all~$t \in (0,1)$ and~$ z \in \mathbb{R}^q$
\begin{align}\label{equ:hypo_jensen_gap}
  \mu_{\min} \|z\|^2 \leq   z^{\top} C(t) z \leq \mu_{\max} \|z\|^2
\end{align}
then
\begin{align}\label{equ:jensen_gap}
\frac{\lambda(1-\lambda)}{2\mu_{\max}^2}  & \|A-B\|_F^2 \le
 G(A,B;\lambda)
 \leq \frac{\lambda(1-\lambda)}{2\mu_{\min}^2}\,\|A-B\|_F^2.
\end{align}
Moreover, setting $D:=B-A$ and $g(t):=\log\det C(t)$, one has
\begin{align}\label{eq:gpp_logdet}
g''(t)=-\bigl\|C(t)^{-1/2}D\,C(t)^{-1/2}\bigr\|_F^2,\qquad t\in[0,1].
\end{align}
\end{lem}
\begin{proof}[Proof of Lemma~\ref{lem:jensen_gap_logdet}]
Set \(D:=B-A\). The function
\[
g:[0,1]\to\mathbb{R},
\qquad
g(t):=\log\det C(t),
\]
is twice continuously differentiable. Jacobi's formula gives
\begin{align}\label{equ:Jacobi_formula}
\frac{d}{dt}\det C(t)
=
\det C(t)\,
\Tr\!\bigl(C(t)^{-1}C'(t)\bigr).
\end{align}
Since \(C'(t)=D\), it follows that
\begin{align}\label{eq:gp_logdet}
g'(t)
=
\Tr\!\bigl(C(t)^{-1}D\bigr).
\end{align}
Differentiating once more and using
\[
\frac{d}{dt}C(t)^{-1}
=
-C(t)^{-1}D\,C(t)^{-1},
\]
we obtain
\begin{align*}
g''(t)
&=
-\Tr\!\bigl(C(t)^{-1}D\,C(t)^{-1}D\bigr) \\
&=
-\Tr\!\left(
  \bigl(C(t)^{-1/2}D\,C(t)^{-1/2}\bigr)^2
 \right) \\
&=
-\bigl\|
 C(t)^{-1/2}D\,C(t)^{-1/2}
 \bigr\|_F^2.
\end{align*}
This proves~\eqref{eq:gpp_logdet}.\par\noindent Here we used the cyclicity of the trace and the fact that
\(C(t)^{-1/2}D\,C(t)^{-1/2}\) is symmetric.

Applying the fundamental theorem of calculus twice gives the
interpolation-error identity
\begin{align}
&g(\lambda)-(1-\lambda)g(0)-\lambda g(1) \nonumber\\
&\quad=
-(1-\lambda)\int_0^\lambda t\,g''(t)\,dt
-\lambda\int_\lambda^1(1-t)\,g''(t)\,dt.
\end{align}
Substituting~\eqref{eq:gpp_logdet}, we obtain
\begin{align}
G(A,B;\lambda)
&=
(1-\lambda)\int_0^\lambda
t\,
\bigl\|
 C(t)^{-1/2}D\,C(t)^{-1/2}
\bigr\|_F^2\,dt \nonumber\\
&\quad+
\lambda\int_\lambda^1
(1-t)\,
\bigl\|
 C(t)^{-1/2}D\,C(t)^{-1/2}
\bigr\|_F^2\,dt.
\label{eq:jensen_gap_logdet_identity}
\end{align}

For every symmetric positive-definite matrix \(C\) and every matrix
\(D\), we have
\begin{align}\label{equ:frobenius_operator_norm_estimate}
\frac{1}{\lambda_{\max}(C)}
\|D\|_F
\leq
\bigl\|C^{-1/2}D\,C^{-1/2}\bigr\|_F
\leq
\frac{1}{\lambda_{\min}(C)}
\|D\|_F.
\end{align}
Consequently, Assumption~\eqref{equ:hypo_jensen_gap} implies
\[
\frac{1}{\mu_{\max}^2}\|D\|_F^2
\leq
\bigl\|C(t)^{-1/2}D\,C(t)^{-1/2}\bigr\|_F^2
\leq
\frac{1}{\mu_{\min}^2}\|D\|_F^2.
\]
Finally,
\[
(1-\lambda)\int_0^\lambda t\,dt
+
\lambda\int_\lambda^1(1-t)\,dt
=
\frac{\lambda(1-\lambda)}{2}.
\]
Inserting these estimates into
\eqref{eq:jensen_gap_logdet_identity} yields
\[
\frac{\lambda(1-\lambda)}{2\mu_{\max}^2}
\|A-B\|_F^2
\leq
G(A,B;\lambda)
\leq
\frac{\lambda(1-\lambda)}{2\mu_{\min}^2}
\|A-B\|_F^2,
\]
which proves~\eqref{equ:jensen_gap}.
\end{proof}

\subsection{Proof of the main results}\label{sec:proof_main_results}

\subsubsection{Proof of Theorem~\ref{thm:classical_QML_consistency}}\label{sec:proof_consistency_infinite_signal_to_noise_ratio} 
Argument for (i): For easier notation we write $A:=\Sigma_{\le}+\varepsilon I_q$, $B:=\Sigma_{>}+\varepsilon I_q$ and $D:=A-B\neq 0$ by assumption (A2). We deduce the concavity of~$J_\infty$ on $(\theta,1)$; the interval~$(0,\theta)$ uses the same argument. Setting
\begin{align}
    X(\nu):=\frac{\theta}{\nu}A+\frac{\nu-\theta}{\nu}B
= B+\frac{\theta}{\nu}D ,
\end{align}
we rewrite~$J_\infty(\nu)$ for $\nu>\theta$ as
\begin{align}
J_\infty(\nu)=\nu\log\det X(\nu) + (1-\nu)\log\det B
=\nu\, g\!\left(\tfrac{\theta}{\nu}\right)+(1-\nu)\log\det B,
\end{align}
where $g(s):=\log\det\bigl((1-s)B+sA\bigr)$: the matrix $X(\nu)$ runs along the affine interpolation of the Jensen-gap Lemma~\ref{lem:jensen_gap_logdet} (with the roles of $A$ and $B$ interchanged) at parameter $s=\theta/\nu$. Differentiating twice, the $g'$-terms cancel, and the second-derivative identity~\eqref{eq:gpp_logdet} in Lemma~\ref{lem:jensen_gap_logdet} yields
\[
J_\infty''(\nu)
=\frac{\theta^2}{\nu^3}\, g''\!\left(\tfrac{\theta}{\nu}\right)
=
-\frac{\theta^2}{\nu^3}\,\big\|X(\nu)^{-1/2}D\,X(\nu)^{-1/2}\big\|_F^2.
\]
The Frobenius norm term is strictly positive since $D\neq 0$ and $X(\nu)$ is positive definite, hence proving the strict concavity of~$J_\infty$ on $(\theta,1)$.\smallskip

Argument for~\eqref{equ:cusp_separation_J_infty}: We observe that by Assumption~\ref{ass:bounded_second_moment} and~\ref{ass:SLLN_in_each_regime} it holds that for all~$z\in \mathbb{R}^q$ 
\begin{align}
    z^{\top} A z \leq (M+\varepsilon) |z|^2 \quad \mbox{and} \quad z^{\top} B z \leq (M+\varepsilon) |z|^2,
\end{align}
and therefore also 
\begin{align}
    z^{\top} ((1-t)A + tB) z \leq (M+\varepsilon) |z|^2 \quad \mbox{for all } t \in [0,1]. 
\end{align}
The complementary lower spectral hypothesis of
Lemma~\ref{lem:jensen_gap_logdet} also holds: the ridge supplies it
when~$\varepsilon>0$, while~(PD) supplies it when~$\varepsilon=0$.
This allows us to apply the lower bound of Lemma~\ref{lem:jensen_gap_logdet}, which yields, for the Jensen gap $G(A,B;\lambda)$ defined there and $\lambda\in[0,1]$, the estimate
\begin{equation}\label{eq:jensen_gap_lower}
G(A,B;\lambda)\;\ge\;\frac{\lambda(1-\lambda)}{2(M +\varepsilon)^2}\,\|A-B\|_F^2.
\end{equation}

\medskip
For $\nu>\theta$, set $\lambda=(\nu-\theta)/\nu\in(0,1)$ so that
\[
\frac{\theta}{\nu}A+\frac{\nu-\theta}{\nu}B=(1-\lambda)A+\lambda B.
\]
Using the definition of $J_\infty$ (cf.~Theorem~\ref{thm:classical_QML_consistency}), direct regrouping yields
\[
J_\infty(\nu)-J_\infty(\theta)=\nu\,G(A,B;\lambda).
\]
Combining this with \eqref{eq:jensen_gap_lower} gives
\[
J_\infty(\nu)-J_\infty(\theta)
\;\ge\;
\nu\cdot \frac{\lambda(1-\lambda)}{2(M + \varepsilon)^2}\,\|A-B\|_F^2.
\]
Now $\lambda=(\nu-\theta)/\nu$ and $1-\lambda=\theta/\nu$, hence
\[
\nu\lambda(1-\lambda)=\nu\cdot\frac{\nu-\theta}{\nu}\cdot\frac{\theta}{\nu}
=\frac{\theta(\nu-\theta)}{\nu}\;\ge\;\theta(\nu-\theta),
\]
since $\nu\le 1$.
Therefore,
\[
J_\infty(\nu)-J_\infty(\theta)
\;\ge\;
\frac{\theta}{2(M +\varepsilon)^2}\,\|A-B\|_F^2\,(\nu-\theta),
\qquad \nu\in(\theta,1].
\]
For $\nu<\theta$, a similar argument yields the estimate 
\[
J_\infty(\nu)-J_\infty(\theta)
\;\ge\;
\frac{1-\theta}{2(M+\varepsilon)^2}\,\|A-B\|_F^2\,(\theta-\nu),
\qquad \nu\in[0,\theta).
\]
 Since the ridge cancels in the difference, $A-B=\Sigma_{\leq}-\Sigma_{>}$ and hence $\|A-B\|_F=\|\Sigma_{>}-\Sigma_{\leq}\|_F$. Combining the two cases yields the desired inequality~\eqref{equ:cusp_separation_J_infty}.\medskip

Argument for (ii): We first prove the uniform convergence on trimmed grids. Fix $0<\delta<1/2$ and $\eta\in(0,1)$. Let $\mathcal{G}_\eta\subset [\delta,1-\delta]$ be a finite grid with mesh at most $\eta$.  Assumption~\ref{ass:SLLN_in_each_regime} and the definitions of $J_n$ and $J_\infty$ imply $J_n(\nu^\sharp)\to J_\infty(\nu^\sharp)$ for every $\nu^\sharp\in\mathcal{G}_\eta$; since the grid is finite,
\[
\max_{\nu^\sharp\in\mathcal{G}_\eta}|J_n(\nu^\sharp)-J_\infty(\nu^\sharp)|\xrightarrow[n\to\infty]{}0.
\]
For any $\nu\in\mathcal D_n(\delta)$ choose $\nu^\sharp\in\mathcal{G}_\eta$ with $|\nu-\nu^\sharp|\leq\eta$. We abbreviate the uniform Lipschitz defect
\[
R_n:=\sup_{\nu_1,\nu_2\in[\delta,1-\delta]}
\Big(|J_n(\nu_1)-J_n(\nu_2)|-C|\nu_2-\nu_1|\Big),
\qquad
\limsup_{n\to\infty}R_n\leq0,
\]
where  $C$ is the Lipschitz constant of Lemma~\ref{lem:Lipschitz_J} and the estimate holds by~\eqref{equ:J_n_Lipschitz_uniform} of Lemma~\ref{lem:Lipschitz_J}; note that $R_n$ controls the pair $(\nu,\nu^\sharp)\in[\delta,1-\delta]^2$ even though $\nu$ varies with $n$. By the triangle inequality and Lemma~\ref{lem:Lipschitz_J},
\begin{align*}
|J_n(\nu)-J_\infty(\nu)|
&\leq |J_n(\nu)-J_n(\nu^\sharp)|
+ |J_n(\nu^\sharp)-J_\infty(\nu^\sharp)|\\
&\qquad+ |J_\infty(\nu^\sharp)-J_\infty(\nu)|\\
&\leq 2C\eta+R_n+\max_{\mu\in\mathcal{G}_\eta}|J_n(\mu)-J_\infty(\mu)|.
\end{align*}
Sending first $n\to\infty$ and then $\eta\downarrow0$ proves~\eqref{equ:uniform_convergence_J_n}. This proves the whole assertion in the case $\varepsilon=0$. Since $\delta$ was arbitrary, the same argument also gives, for $\varepsilon>0$ and every fixed $s\in(0,1/2)$,
\begin{equation}\label{eq:middle_region_uniform}
    \lim_{n\to\infty}
    \sup_{\nu\in\mathcal D_n\cap[s,1-s]}
    |J_n(\nu)-J_\infty(\nu)|=0.
\end{equation}

It remains to remove the middle-region cutoff when $\varepsilon>0$. We do this by a sequential endpoint estimate. For $s>0$, set
\[
\begin{aligned}
E_n^L(s):=
\sup_{\nu\in\mathcal D_n\cap(0,s]}
|J_n(\nu)-J_\infty(\nu)|,\\
E_n^R(s):=
\sup_{\nu\in\mathcal D_n\cap[1-s,1)}
|J_n(\nu)-J_\infty(\nu)|.
\end{aligned}
\]
We claim
\begin{align}\label{eq:thm17_endpoint_terms}
    \lim_{s\downarrow0}\limsup_{n\to\infty}E_n^L(s)=0,
    \qquad
    \lim_{s\downarrow0}\limsup_{n\to\infty}E_n^R(s)=0.
\end{align}
By symmetry, it suffices to consider the left endpoint. We use the sequential
criterion for the first limit in~\eqref{eq:thm17_endpoint_terms}: it is enough
to prove that
\[
|J_n(\nu_n)-J_\infty(\nu_n)|\longrightarrow0
\]
for every sequence $\nu_n\in\mathcal D_n$ with $\nu_n\to0$. Indeed, failure
of the first limit would yield, by a diagonal choice, a sequence of this form
along which the difference stays bounded away from zero (completed by
$\nu_n:=1/n$ at any unused indices).

Fix such a sequence. Lemma~\ref{lem:boundary_logdet} gives
\[
\nu_n\left|\log\det\big(\widehat\Sigma_{0:\nu_n}+\varepsilon I_q\big)\right|
\longrightarrow0.
\]
Moreover, eventually $\nu_n<\theta$, so $\Sigma_{0:\nu_n}=\Sigma_{\leq}$ and
\[
\nu_n\left|\log\det\big(\Sigma_{0:\nu_n}+\varepsilon I_q\big)\right|
=\nu_n\left|\log\det\big(\Sigma_{\leq}+\varepsilon I_q\big)\right|
\longrightarrow0.
\]

It remains to compare the long right blocks. Set
\[
\Sigma_*:=\theta\Sigma_{\leq}+(1-\theta)\Sigma_{>}.
\]
 By~\eqref{eq:block_additivity} and Assumption~\ref{ass:SLLN_in_each_regime},
\[
\widehat\Sigma_{0:1}
=\frac{\lfloor\theta n\rfloor}{n}\widehat\Sigma_{0:\theta}
+\frac{n-\lfloor\theta n\rfloor}{n}\widehat\Sigma_{\theta:1}
\longrightarrow\Sigma_*.
\]
Writing $k_n:=n\nu_n\in\{1,\ldots,n-1\}$,  identity~\eqref{eq:block_additivity_solved} also gives
\[
\widehat\Sigma_{\nu_n:1}-\widehat\Sigma_{0:1}
=\frac{k_n}{n-k_n}\widehat\Sigma_{0:1}
-\frac{1}{n-k_n}\sum_{t=1}^{k_n}y_ty_t^{\top}.
\]
Consequently,
\[
\left\|\widehat\Sigma_{\nu_n:1}-\widehat\Sigma_{0:1}\right\|
\leq
\frac{k_n}{n-k_n}\left\|\widehat\Sigma_{0:1}\right\|
+\frac{1}{n-k_n}\sum_{t=1}^{k_n}\|y_t\|^2
\longrightarrow0,
\]
where the first term vanishes because $k_n/n=\nu_n\to0$ and
$\widehat\Sigma_{0:1}$ converges, while the second vanishes by
\eqref{eq:boundary_energy_sequential} of Lemma~\ref{lem:boundary_energy}.
Thus $\widehat\Sigma_{\nu_n:1}\to\Sigma_*$. On the other hand, the definition
of the population target gives
\[
\Sigma_{\nu_n:1}
=\frac{\theta-\nu_n}{1-\nu_n}\Sigma_{\leq}
+\frac{1-\theta}{1-\nu_n}\Sigma_{>}
\longrightarrow\Sigma_*.
\]
Hence
\[
\left\|\widehat\Sigma_{\nu_n:1}-\Sigma_{\nu_n:1}\right\|\longrightarrow0.
\]
Since the two regularized matrices dominate $\varepsilon I_q$,
Lemma~\ref{lem:lipischitz_log_det} now yields
\[
\left|
\log\det\big(\widehat\Sigma_{\nu_n:1}+\varepsilon I_q\big)
-\log\det\big(\Sigma_{\nu_n:1}+\varepsilon I_q\big)
\right|\longrightarrow0.
\]
Consequently,
\[
\begin{aligned}
|J_n(\nu_n)-J_\infty(\nu_n)|
&\leq \nu_n\left|\log\det\big(\widehat\Sigma_{0:\nu_n}+\varepsilon I_q\big)\right|\\
&\quad +\nu_n\left|\log\det\big(\Sigma_{0:\nu_n}+\varepsilon I_q\big)\right|\\
&\quad +(1-\nu_n)\left|
\log\det\big(\widehat\Sigma_{\nu_n:1}+\varepsilon I_q\big)
-\log\det\big(\Sigma_{\nu_n:1}+\varepsilon I_q\big)
\right|\\
&\longrightarrow0.
\end{aligned}
\]
This proves the sequential claim and
therefore the left endpoint limit in~\eqref{eq:thm17_endpoint_terms}. The right
endpoint follows symmetrically.

For fixed $s\in(0,\frac12\min\{\theta,1-\theta\})$, write
\[
E_n^M(s):=
\sup_{\nu\in\mathcal D_n\cap[s,1-s]}
|J_n(\nu)-J_\infty(\nu)|.
\]
 By~\eqref{eq:middle_region_uniform}, $E_n^M(s)\to0$. Since
\[
\sup_{\nu\in\mathcal D_n}|J_n(\nu)-J_\infty(\nu)|
\leq\max\big\{E_n^L(s),E_n^M(s),E_n^R(s)\big\},
\]
taking first $\limsup_{n\to\infty}$ and then $s\downarrow0$ in
\eqref{eq:thm17_endpoint_terms} proves
\eqref{equ:uniform_convergence_J_n_untrimmed}.

\medskip
Argument for (iii): Let $\theta_n:=\lfloor \theta n\rfloor/n$. Then $\theta_n\in\mathcal D_n$ for all large enough $n$ and $\theta_n\to\theta$. We first consider $\varepsilon>0$ and let $\hat\theta_n\in\arg\min_{\nu\in\mathcal D_n}J_n(\nu)$. Every subsequence has a further subsequence, again denoted by $\hat\theta_n$, converging to some $\nu\in[0,1]$. We prove that $\nu=\theta$.

Assume $\nu\neq\theta$ and set $\gamma:=|\nu-\theta|>0$. Let
\[
c_\infty:=
\frac{\min\{\theta,\,1-\theta\}}{2(M+\varepsilon)^2}\,
\|\Sigma_{>}-\Sigma_{\leq}\|_F^2>0
\]
be the cusp constant in~\eqref{equ:cusp_separation_J_infty}. For all large enough $n$, $|\hat\theta_n-\theta|\geq\gamma/2$. By part~(ii), for all large enough $n$,
\[
\sup_{\mu\in\mathcal D_n}|J_n(\mu)-J_\infty(\mu)|
\leq \frac{c_\infty\gamma}{16}.
\]
Let $C_\infty$ be the Lipschitz constant in~\eqref{equ:J_infty_Lipschitz}. Since $\theta_n\to\theta$, for all large enough $n$,
\[
|J_\infty(\theta_n)-J_\infty(\theta)|
\leq C_\infty|\theta_n-\theta|
\leq \frac{c_\infty\gamma}{16}.
\]
 By~\eqref{equ:cusp_separation_J_infty} and $|\hat\theta_n-\theta|\geq\gamma/2$, we have $J_\infty(\hat\theta_n)-J_\infty(\theta)\geq c_\infty|\hat\theta_n-\theta|\geq c_\infty\gamma/2$. Together with the two estimates above and the comparison point $\theta_n$, we get
\begin{align*}
J_n(\hat\theta_n)-J_n(\theta_n)
&\geq
J_\infty(\hat\theta_n)-J_\infty(\theta)
-2\sup_{\mu\in\mathcal D_n}|J_n(\mu)-J_\infty(\mu)|
-|J_\infty(\theta_n)-J_\infty(\theta)|\\
&\geq
c_\infty|\hat\theta_n-\theta|
-\frac{c_\infty\gamma}{8}
-\frac{c_\infty\gamma}{16} \geq
\frac{5}{16}c_\infty\gamma>0,
\end{align*}
contradicting $J_n(\hat\theta_n)\leq J_n(\theta_n)$. Hence every subsequential limit equals $\theta$, and $\hat\theta_n\to\theta$.

If $\varepsilon=0$, fix $0<\delta<\min\{\theta,1-\theta\}$ and let $\hat\theta_{n,\delta}\in\arg\min_{\nu\in\mathcal D_n(\delta)}J_n(\nu)$. The same argument applies with $\mathcal D_n(\delta)$ in place of $\mathcal D_n$, because $\theta_n\in\mathcal D_n(\delta)$ for all large enough $n$ and the trimmed uniform convergence~\eqref{equ:uniform_convergence_J_n} holds on that grid. Therefore $\hat\theta_{n,\delta}\to\theta$. \qed

\subsubsection{Proof of Theorem~\ref{thm:QML-robust}}\label{sec:proof_consistency_finite_signal_to_noise}

The proof implements the two-gap mechanism described in
Section~\ref{subsec:heuristic-eps}: the objective difference is decomposed
into a positive Jensen gap measuring separation between the regimes and a
negative gap measuring variation within a regime.
The two-sided estimate of Lemma~\ref{lem:jensen_gap_logdet} reduces positivity
to~\textup{(A3)} locally and~\textup{(A4)} globally, after which the continuity
and endpoint lemmas transfer the fixed-candidate bounds to minimizers on the
grid.

We first note that the hypotheses imply
$\Delta_{\mathrm{within}}<\infty$. Indeed, choosing the fixed blocks
$I(0,\theta)$ and $I(\theta,1)$ in the definition of
$\Delta_{\mathrm{between}}$, Assumption~(A1), the triangle inequality, and
$\|\cdot\|_F\leq\sqrt q\,\|\cdot\|$ give $\Delta_{\mathrm{between}}\leq2\sqrt q\,M<\infty$.
Under either~(A3) or~(A4), the corresponding strict inequality then forces
$\Delta_{\mathrm{within}}<\infty$. We may therefore use the auxiliary
lemmas above. For brevity, write $\ell_n(a,b):=\log\det\bigl(\widehat\Sigma_{a:b}+\varepsilon I_q\bigr)$.

\medskip
\noindent\emph{Fixed-candidate bounds.}
Fix $\nu>\theta$ and recall $k_\theta=\lfloor\theta n\rfloor$, $k_\nu=\lfloor\nu n\rfloor$.
For all large enough~$n$,  \eqref{eq:block_additivity} gives
\begin{align}\label{equ:decomp_empirical_second_moments}
\widehat\Sigma_{0:\nu}
=\frac{k_\theta}{k_\nu}\widehat\Sigma_{0:\theta}
+\frac{k_\nu-k_\theta}{k_\nu}\widehat\Sigma_{\theta:\nu},
\qquad
\widehat\Sigma_{\theta:1}
=\frac{k_\nu-k_\theta}{n-k_\theta}\widehat\Sigma_{\theta:\nu}
+\frac{n-k_\nu}{n-k_\theta}\widehat\Sigma_{\nu:1}.
\end{align}
Define
\ifwidedisplays
\begin{align*}
&A:=\widehat\Sigma_{0:\theta}+\varepsilon I_q,\quad
B:=\widehat\Sigma_{\theta:\nu}+\varepsilon I_q,\quad
C:=\widehat\Sigma_{\nu:1}+\varepsilon I_q,\qquad
\lambda_1:=\frac{k_\nu-k_\theta}{k_\nu},\quad
\lambda_2:=\frac{n-k_\nu}{n-k_\theta},\\
&\text{so that}\quad
\widehat\Sigma_{0:\nu}+\varepsilon I_q=(1-\lambda_1)A+\lambda_1B
\quad\text{and}\quad
\widehat\Sigma_{\theta:1}+\varepsilon I_q=(1-\lambda_2)B+\lambda_2C.
\end{align*}
\else
\begin{align*}
&A:=\widehat\Sigma_{0:\theta}+\varepsilon I_q,\quad
B:=\widehat\Sigma_{\theta:\nu}+\varepsilon I_q,\quad
C:=\widehat\Sigma_{\nu:1}+\varepsilon I_q,\\
&\lambda_1:=\frac{k_\nu-k_\theta}{k_\nu},\qquad
\lambda_2:=\frac{n-k_\nu}{n-k_\theta},\\
&\text{so that}\quad
\widehat\Sigma_{0:\nu}+\varepsilon I_q=(1-\lambda_1)A+\lambda_1B
\quad\text{and}\quad
\widehat\Sigma_{\theta:1}+\varepsilon I_q=(1-\lambda_2)B+\lambda_2C.
\end{align*}
\fi
Using the Jensen gap $G$ of Lemma~\ref{lem:jensen_gap_logdet}, direct
expansion yields
\begin{equation}\label{eq:two-gap}
J_n(\nu)-J_n(\theta)
=\frac{k_\nu}{n}G(A,B;\lambda_1)
-\frac{n-k_\theta}{n}G(B,C;\lambda_2)
+r_n(\nu),
\end{equation}
where the floor-weight remainder is
\begin{align}\label{eq:two-gap-remainder}
r_n(\nu)
:=\Big(\nu-\frac{k_\nu}{n}\Big)
\bigl(\ell_n(0,\nu)-\ell_n(\nu,1)\bigr)
-\Big(\theta-\frac{k_\theta}{n}\Big)
\bigl(\ell_n(0,\theta)-\ell_n(\theta,1)\bigr).
\end{align}

For fixed~$\nu$, Lemma~\ref{lem:opnorm_Sigma_ab} supplies a common finite
spectral ceiling for the four blocks in~\eqref{eq:two-gap-remainder}.
The ridge supplies a positive floor when $\varepsilon>0$; when
$\varepsilon=0$, (PD), Definition~\ref{def_singal_and_noise}, and the
 decomposition~\eqref{eq:block_additivity} of the crossing block supply one. Thus there are
constants $c_\nu,C_\nu>0$ such that, for all large enough~$n$,
\begin{align}\label{eq:block_sandwich}
c_\nu I_q
\preceq\widehat\Sigma_{a:b}+\varepsilon I_q
\preceq C_\nu I_q,
\qquad
(a,b)\in\{(0,\nu),(\nu,1),(0,\theta),(\theta,1)\}.
\end{align}
The sandwich~\eqref{eq:block_sandwich} bounds the four log-determinants.
 Since $0\leq\nu-k_\nu/n<1/n$ and $0\leq\theta-k_\theta/n<1/n$ in~\eqref{eq:two-gap-remainder},
\begin{align}\label{eq:remainder_vanishes}
r_n(\nu)\longrightarrow0
\qquad\text{for every fixed }\nu\in(0,1).
\end{align}

Since the macroscopic intervals underlying $A$ and $B$ belong to $\mathcal I_\theta^{\mathrm{br}}$, while those underlying $B$ and $C$ belong to $\mathcal I_\theta^{\mathrm{cut}}$, equation~\eqref{equ:def_M_m} gives
\[
\limsup_{n\to\infty}\bigl(\lambda_{\max}(A)\vee\lambda_{\max}(B)\bigr)
\leq M+\varepsilon,
\qquad
\liminf_{n\to\infty}\bigl(\lambda_{\min}(B)\wedge\lambda_{\min}(C)\bigr)
\geq m+\varepsilon>0.
\]
Here $m+\varepsilon>0$ follows from the ridge when $\varepsilon>0$, or from~(PD) when $\varepsilon=0$. The complementary bounds required in~\eqref{equ:hypo_jensen_gap}, a floor for $A,B$ and a ceiling for $B,C$, follow from the ridge or~(PD) and from Lemma~\ref{lem:opnorm_Sigma_ab}, respectively. Since the blocks underlying $A$ and $B$ lie in opposite regimes, whereas those underlying $B$ and $C$ lie in the same regime, equations~\eqref{eq:Delta-within} and~\eqref{eq:Delta-between} give
\[
\liminf_{n\to\infty}\|A-B\|_F\geq\Delta_{\mathrm{between}},
\qquad
\limsup_{n\to\infty}\|B-C\|_F\leq\Delta_{\mathrm{within}},
\]
where the $\varepsilon I_q$ terms cancel in both differences. Lemma~\ref{lem:jensen_gap_logdet} now yields the two estimates below.

Therefore
\ifwidedisplays
\[
\liminf_{n\to\infty}G(A,B;\lambda_1)
\geq
\frac{\theta(\nu-\theta)}{2\nu^2(M+\varepsilon)^2}
\Delta_{\mathrm{between}}^2,
\qquad
\limsup_{n\to\infty}G(B,C;\lambda_2)
\leq
\frac{(1-\nu)(\nu-\theta)}{2(1-\theta)^2(m+\varepsilon)^2}
\Delta_{\mathrm{within}}^2.
\]
\else
\begin{align*}
\liminf_{n\to\infty}G(A,B;\lambda_1)
&\geq
\frac{\theta(\nu-\theta)}{2\nu^2(M+\varepsilon)^2}
\Delta_{\mathrm{between}}^2,\\
\limsup_{n\to\infty}G(B,C;\lambda_2)
&\leq
\frac{(1-\nu)(\nu-\theta)}{2(1-\theta)^2(m+\varepsilon)^2}
\Delta_{\mathrm{within}}^2.
\end{align*}
\fi
Introduce the normalized signal and noise levels
$S:=\Delta_{\mathrm{between}}^2/(M+\varepsilon)^2$ and
$W:=\Delta_{\mathrm{within}}^2/(m+\varepsilon)^2$.
Inserting the last two estimates into~\eqref{eq:two-gap} and using
\eqref{eq:remainder_vanishes} gives
\begin{align}\label{eq:jn-lower-pre}
\liminf_{n\to\infty}\bigl(J_n(\nu)-J_n(\theta)\bigr)
\geq\frac12 T_+(\nu)(\nu-\theta),
\qquad
T_+(\nu):=\frac{\theta}{\nu}S-\frac{1-\nu}{1-\theta}W,
\qquad \nu>\theta.
\end{align}
Applying the same calculation to the time-reversed sequence, equivalently
interchanging left and right and replacing $(\nu,\theta)$ by
$(1-\nu,1-\theta)$, yields
\begin{align}\label{eq:two-gap-left}
\liminf_{n\to\infty}\bigl(J_n(\nu)-J_n(\theta)\bigr)
\geq\frac12 T_-(\nu)(\theta-\nu),
\qquad
T_-(\nu):=\frac{1-\theta}{1-\nu}S-\frac{\nu}{\theta}W,
\qquad \nu<\theta.
\end{align}

\medskip
\noindent\emph{Argument for (i).}
Assumption~(A3) is exactly $S>W$. Moreover, $T_+(\theta)=T_-(\theta)=S-W>0$, so by continuity there is $\bar\delta>0$ such that $\min\{T_-(\nu),T_+(\nu)\}\geq(S-W)/2$ whenever $|\nu-\theta|\leq\bar\delta$.
The two inequalities~\eqref{eq:jn-lower-pre} and~\eqref{eq:two-gap-left} prove part~(i)
with
\begin{align}\label{eq:const_local}
c=\frac{S-W}{4}.
\end{align}

\medskip
\noindent\emph{Argument for (ii).}
Assumption~(A4) is equivalent to
\begin{align}\label{eq:normalized_A4}
4\theta(1-\theta)S>W.
\end{align}
If $W=0$, then $S>0$ and $T_+(\nu)\geq\theta S$, $T_-(\nu)\geq(1-\theta)S$.
Suppose now that $W>0$ and set  $d:=2\sqrt{\theta(1-\theta)SW}-W$, which is positive by~\eqref{eq:normalized_A4}.
The arithmetic-geometric mean inequality gives, for $\nu\geq\theta$ and for $\nu\leq\theta$ respectively,
\ifwidedisplays
\[
T_+(\nu)
=\frac{\theta S}{\nu}+\frac{\nu W}{1-\theta}-\frac{W}{1-\theta}
\geq\frac{d}{1-\theta},
\qquad
T_-(\nu)
=\frac{(1-\theta)S}{1-\nu}+\frac{(1-\nu)W}{\theta}-\frac{W}{\theta}
\geq\frac{d}{\theta}.
\]
\else
\begin{align*}
T_+(\nu)
&=\frac{\theta S}{\nu}+\frac{\nu W}{1-\theta}-\frac{W}{1-\theta}
\geq\frac{d}{1-\theta},\\
T_-(\nu)
&=\frac{(1-\theta)S}{1-\nu}+\frac{(1-\nu)W}{\theta}-\frac{W}{\theta}
\geq\frac{d}{\theta}.
\end{align*}
\fi
Thus inequalities~\eqref{eq:jn-lower-pre} and~\eqref{eq:two-gap-left}
prove~\eqref{equ:minimizer_a4}. One may take
\begin{align}\label{eq:const_global}
c=
\begin{cases}
\frac12\min\{\theta,1-\theta\}S,&W=0,\\[2mm]
\dfrac{d}{2\max\{\theta,1-\theta\}},&W>0.
\end{cases}
\end{align}

\medskip
\noindent\emph{Consistency of the minimizers.}
Let $\theta_n:=\lfloor \theta n\rfloor/n$. Then
$\theta_n\in\mathcal D_n$ for all large enough~$n$ and
$\theta_n\to\theta$. Lemma~\ref{lem:seq_continuity_J} gives
\begin{align}\label{eq:grid_true_split_theta}
J_n(\theta_n)-J_n(\theta)=o(1).
\end{align}

First consider the trimmed estimator, which is the only assertion made when
$\varepsilon=0$. Fix $0<\delta<\min\{\theta,1-\theta\}$ and let $\widehat\theta_{n,\delta}\in\arg\min_{\nu\in\mathcal D_n(\delta)}J_n(\nu)$.
For all large enough~$n$, also $\theta_n\in\mathcal D_n(\delta)$. Every
subsequence has a further subsequence, again denoted by
$\widehat\theta_{n,\delta}$, converging to some
$\nu\in[\delta,1-\delta]$. If $\nu\neq\theta$, then
Lemma~\ref{lem:seq_continuity_J},~\eqref{equ:minimizer_a4}, and
\eqref{eq:grid_true_split_theta} give
\[
\liminf_{n\to\infty}
\bigl(J_n(\widehat\theta_{n,\delta})-J_n(\theta_n)\bigr)
=\liminf_{n\to\infty}
\bigl(J_n(\nu)-J_n(\theta)\bigr)
\geq c|\nu-\theta|>0.
\]
This contradicts the minimality of $\widehat\theta_{n,\delta}$. Hence every
subsequential limit equals~$\theta$, and
$\widehat\theta_{n,\delta}\to\theta$.

Now assume $\varepsilon>0$ and let $\widehat\theta_n\in\arg\min_{\nu\in\mathcal D_n}J_n(\nu)$.
The preceding argument excludes every interior subsequential limit other
than~$\theta$. It remains to exclude the endpoints. We treat~$1$; the case
$0$ is symmetric.

Suppose along a subsequence that $\widehat\theta_n\to1$. For all large enough
$n$, $\widehat\theta_n>\theta$. Set
\ifwidedisplays
\[
A_n:=\widehat\Sigma_{0:\theta}+\varepsilon I_q,\quad
B_n:=\widehat\Sigma_{\theta:\widehat\theta_n}+\varepsilon I_q,\quad
C_n:=\widehat\Sigma_{\widehat\theta_n:1}+\varepsilon I_q,\qquad
\lambda_{1,n}:=\frac{k_{\widehat\theta_n}-k_\theta}{k_{\widehat\theta_n}},\quad
\lambda_{2,n}:=\frac{n-k_{\widehat\theta_n}}{n-k_\theta}.
\]
\else
\begin{align*}
&A_n:=\widehat\Sigma_{0:\theta}+\varepsilon I_q,\quad
B_n:=\widehat\Sigma_{\theta:\widehat\theta_n}+\varepsilon I_q,\quad
C_n:=\widehat\Sigma_{\widehat\theta_n:1}+\varepsilon I_q,\\
&\lambda_{1,n}:=\frac{k_{\widehat\theta_n}-k_\theta}{k_{\widehat\theta_n}},\qquad
\lambda_{2,n}:=\frac{n-k_{\widehat\theta_n}}{n-k_\theta}.
\end{align*}
\fi
 Applying equation~\eqref{eq:two-gap} with $\nu=\widehat\theta_n$ gives
\begin{align}\label{eq:two-gap-endpoint}
J_n(\widehat\theta_n)-J_n(\theta)
=\frac{k_{\widehat\theta_n}}{n}
G(A_n,B_n;\lambda_{1,n})
-\frac{n-k_\theta}{n}G(B_n,C_n;\lambda_{2,n})
+r_n(\widehat\theta_n).
\end{align}
 By~\eqref{eq:two-gap-remainder},
\begin{align*}
r_n(\widehat\theta_n)
={}&\Big(\widehat\theta_n-\frac{k_{\widehat\theta_n}}{n}\Big)
\bigl(\ell_n(0,\widehat\theta_n)-\ell_n(\widehat\theta_n,1)\bigr)
-\Big(\theta-\frac{k_\theta}{n}\Big)
\bigl(\ell_n(0,\theta)-\ell_n(\theta,1)\bigr).
\end{align*}
Since $\widehat\theta_n\in\mathcal D_n$, the first floor error vanishes, $\widehat\theta_n-k_{\widehat\theta_n}/n=0$, and hence
\[
r_n(\widehat\theta_n)
=-\left(\theta-\frac{k_\theta}{n}\right)
\bigl(\ell_n(0,\theta)-\ell_n(\theta,1)\bigr)
=o(1)
\]
by Assumption~(A1) and the ridge lower bound.

 By Lemma~\ref{lem:endpoint_large_block} and $\|\cdot\|_F\leq\sqrt q\,\|\cdot\|$,
\begin{align}\label{eq:endpoint_block_convergence}
\big\|\widehat\Sigma_{\theta:\widehat\theta_n}
-\widehat\Sigma_{\theta:1}\big\|_F\longrightarrow0.
\end{align}
For $G(A_n,B_n;\lambda_{1,n})$, equation~\eqref{eq:endpoint_block_convergence}, the triangle
inequality, and the definition of~$M$ first give
\[
\limsup_{n\to\infty}\|\widehat\Sigma_{\theta:\widehat\theta_n}\|
\leq
\limsup_{n\to\infty}\|\widehat\Sigma_{\theta:1}\|
\leq M.
\]
For the Frobenius gap, the ridge terms cancel and
$A_n-B_n=(\widehat\Sigma_{0:\theta}-\widehat\Sigma_{\theta:1})-(\widehat\Sigma_{\theta:\widehat\theta_n}-\widehat\Sigma_{\theta:1})$.
 Equation~\eqref{eq:endpoint_block_convergence}, the reverse triangle inequality and the definition of
$\Delta_{\mathrm{between}}$ therefore yield
\[
\liminf_{n\to\infty}\|A_n-B_n\|_F
\geq
\liminf_{n\to\infty}
\|\widehat\Sigma_{0:\theta}-\widehat\Sigma_{\theta:1}\|_F
\geq\Delta_{\mathrm{between}}.
\]
Moreover,
$k_{\widehat\theta_n}/n\to1$ and
$\lambda_{1,n}\to1-\theta$. Lemma~\ref{lem:jensen_gap_logdet} therefore gives
\begin{align}\label{eq:endpoint_positive_gap}
\liminf_{n\to\infty}
\frac{k_{\widehat\theta_n}}{n}
G(A_n,B_n;\lambda_{1,n})
\geq
\frac{\theta(1-\theta)}{2(M+\varepsilon)^2}
\Delta_{\mathrm{between}}^2>0.
\end{align}

 For $G(B_n,C_n;\lambda_{2,n})$, identity~\eqref{eq:block_additivity} gives
$(1-\lambda_{2,n})B_n+\lambda_{2,n}C_n=\widehat\Sigma_{\theta:1}+\varepsilon I_q$,
and hence
\[
G(B_n,C_n;\lambda_{2,n})
=\ell_n(\theta,1)-\log\det B_n
+\lambda_{2,n}\bigl(\log\det B_n-\log\det C_n\bigr).
\]
The first difference is $o(1)$ by Lemma~\ref{lem:endpoint_large_block}, and
$\lambda_{2,n}\log\det B_n=o(1)$ because $\lambda_{2,n}\to0$ and
$\log\det B_n$ is bounded. Finally, for a constant independent of~$n$,
\[
\lambda_{2,n}|\log\det C_n|
\leq C(1-\widehat\theta_n)
\left|\log\det\bigl(
\widehat\Sigma_{\widehat\theta_n:1}+\varepsilon I_q
\bigr)\right|
\longrightarrow0
\]
by Lemma~\ref{lem:boundary_logdet}. Thus
$G(B_n,C_n;\lambda_{2,n})=o(1)$.

Together with~\eqref{eq:two-gap-endpoint} and
\eqref{eq:endpoint_positive_gap}, this gives
$\liminf_{n\to\infty}\bigl(J_n(\widehat\theta_n)-J_n(\theta)\bigr)>0$.
Equation~\eqref{eq:grid_true_split_theta} then implies
$J_n(\widehat\theta_n)>J_n(\theta_n)$ for all large enough~$n$, contradicting
minimality over~$\mathcal D_n$. Hence neither endpoint is a subsequential
limit, so $\widehat\theta_n\to\theta$. \qed

\subsubsection{Proof of Proposition~\ref{prop:sharp-A4-fixed-theta}}\label{sec:proof_optimality}
Let us define~$a= 1+ K+r$,~$b=1+K$, and~$c = 1$ and~$\nu_\star = 0.5$.\smallskip

Argument for~\eqref{equ:optimality_signal_and_noise} and~\eqref{equ:optimality_signal_and_noise_at_05}:
To compute $\Delta_{\rm within}$ and $\Delta_{\rm between}$ we observe that for~$q=1$ the block second moment on a macroscopic interval is just the block average of $y_t^2$, and the Frobenius norm~$\|\cdot\|_F$ is just the absolute value. On the left regime $[0,\theta]$ the profile is constant with second moment equal to $a$, so all left-block averages equal $a$ and therefore the within-left fluctuation is $0$.
On the right regime $(\theta,1]$, the profile takes the two macroscopic values $b$ and $c$ on
 $(\theta,\nu_\star]$ and $(\nu_\star,1]$, respectively.
Choosing one block entirely in the $b$-part and one entirely in the $c$-part yields
\[
\Delta_{\rm within}=|b-c|=K.
\]
For $\Delta_{\rm between}$, any left-block average equals $a$, while any right-block average lies in $[c,b]$
(since it is an average of $b$ and $c$). Hence the minimal separation is attained by a block inside the $b$-part:
\[
    \Delta_{\rm between}=|a-b|=r.
\]
At the candidate~$\nu_\star=\tfrac12$, every left-block average lies in
$[b,a]$, with both endpoints attained by blocks contained in the $b$- and
$a$-parts, while every right-block average equals~$c$. Hence
\[
\Delta_{\rm within}(\nu_\star)=a-b=r,
\qquad
\Delta_{\rm between}(\nu_\star)=b-c=K.
\]
Together with the preceding calculation, this proves
\eqref{equ:optimality_signal_and_noise}
and~\eqref{equ:optimality_signal_and_noise_at_05}.\smallskip

In the next step, we explicitly calculate~$J_\infty$.  Here $J_\infty$ denotes the pointwise limit of~$J_n$ for the three-level series~\eqref{eq:triangular-array-sharpA4}. Since Assumption~(A2) fails for this series, this is not the two-regime function of Definition~\ref{def:population_targets}; it is, however, given by the same formula, with $\Sigma_{0:\nu}$ and $\Sigma_{\nu:1}$ replaced by the block limits computed next.
As our data is piecewise constant, direct averaging yields the following formulas:
\[
\Sigma_{0:\nu}=
\begin{cases}
a, & 0<\nu\le \theta,\\[0.4em]
\dfrac{\theta a+(\nu-\theta)b}{\nu}, & \theta<\nu\le \nu_\star,\\[0.8em]
\dfrac{\theta a+(\nu_\star-\theta)b+(\nu-\nu_\star)c}{\nu}, & \nu_\star<\nu<1,
\end{cases}
\]
and
\[
\Sigma_{\nu:1}=
\begin{cases}
\dfrac{(\theta-\nu)a+(\nu_\star-\theta)b+(1-\nu_\star)c}{1-\nu}, & 0<\nu<\theta,\\[0.8em]
\dfrac{(\nu_\star-\nu)b+(1-\nu_\star)c}{1-\nu}, & \theta\le \nu<\nu_\star,\\[0.8em]
c, & \nu_\star\le \nu<1.
\end{cases}
\]
Using the definition~$\varphi(x):=\log(x+\varepsilon)$, the limiting objective function is therefore explicitly
\begin{equation}\label{eq:Jinfty-explicit-prop18}
J_\infty(\nu)=\nu\,\varphi\!\big(\Sigma_{0:\nu}\big)+(1-\nu)\,\varphi\!\big(\Sigma_{\nu:1}\big),
\qquad \nu\in(0,1),
\end{equation}
where $\Sigma_{0:\nu}$ and $\Sigma_{\nu:1}$ are given by the piecewise expressions above.
In particular, at the two relevant points one has
\[
J_\infty(\theta)=\theta\,\varphi(a)+(1-\theta)\,\varphi\!\left(\frac{(\nu_\star-\theta)b+(1-\nu_\star)c}{1-\theta}\right),
\]
\[
J_\infty(\nu_\star)=\nu_\star\,\varphi\!\left(\frac{\theta a+(\nu_\star-\theta)b}{\nu_\star}\right)
+(1-\nu_\star)\,\varphi(c).
\]
We now show, for an arbitrary~$\varepsilon>0$, that the only candidates for a local minimizer of~$J_\infty$ are $\theta$ and $\nu_\star=0.5$ and that the global minimum is attained at $\nu_\star$. As~$J_\infty$ is piecewise smooth, we consider several cases. \smallskip

Case 1: If $\nu\in(\nu_\star,1)$, then by definition
\[
\Sigma_{0:\nu}
=\frac{\theta a+(\nu_\star-\theta)b+(\nu-\nu_\star)c}{\nu} > c = \Sigma_{\nu:1}.
\]
(The inequality is strict because the displayed mixture assigns positive
weight to $a,b>c$.)
Hence
\begin{equation}\label{eq:J-right}
J_\infty(\nu)=\nu\,\varphi(\Sigma_{0:\nu})+(1-\nu)\,\varphi(c),
\qquad \mbox{for } \nu\in(\nu_\star,1),
\end{equation}
and differentiating gives
\begin{equation}\label{eq:Jprime-right}
J_\infty'(\nu)
=\varphi(\Sigma_{0:\nu})-\varphi(c)+\nu\,\varphi'(\Sigma_{0:\nu})\,\Sigma_{0:\nu}',
\end{equation}
where 
\[
\Sigma_{0:\nu}'= -\frac{\theta a+ (\nu_\star-\theta)b  - \nu_\star c}{\nu^2} = - \frac{\Sigma_{0: \nu}}{ \nu}  + \frac{c}{\nu} .
\]
Using strict concavity of $\varphi$,
\[
\varphi(\Sigma_{0:\nu})-\varphi(c) > \varphi'(\Sigma_{0:\nu})\big(\Sigma_{0:\nu}-c\big),
\]
we obtain for all~$\nu\in(\nu_\star,1)$ that
\begin{align}
J_\infty'(\nu) & > \varphi'(\Sigma_{0:\nu})(\Sigma_{0:\nu}-c)
+\nu\,\varphi'(\Sigma_{0:\nu})\,\Sigma_{0:\nu}'\\
& = \varphi'(\Sigma_{0:\nu})\big[(\Sigma_{0:\nu}-c)+\nu\Sigma_{0:\nu}'\big]=0.
\end{align}
Thus $J_\infty$ is strictly increasing on $(\nu_\star,1)$, so there is no local minimizer
in $(\nu_\star,1)$ and $\nu_\star$ is a (one-sided) local minimizer from the right. \smallskip

Case 2: The argument for $ \nu \in (0,\theta)$ is similar to Case~1. For the convenience of the reader we carry out all details. For~$\nu \in (0, \theta)$,  
\begin{align}
\Sigma_{0:\nu} =  a \quad \mbox{and} \quad \Sigma_{\nu:1}=\frac{(\theta-\nu)a+(\nu_\star-\theta)b+(1-\nu_\star)c}{1-\nu}.    
\end{align}
Hence $J_\infty$ is smooth on $(0,\theta)$ and given by (recalling $\varphi(x)=\log(x+\varepsilon)$)
\[
J_\infty(\nu)=\nu\,\varphi(a)+(1-\nu)\,\varphi(\Sigma_{\nu:1}),
\qquad \nu\in(0,\theta).
\]
Its derivative is given by
\begin{equation}
\label{eq:case2-derivative}
J_\infty'(\nu)=\varphi(a)-\varphi(\Sigma_{\nu:1})+(1-\nu)\,\varphi'(\Sigma_{\nu:1}) \Sigma_{\nu:1}',
\end{equation}
where 
\begin{align}\label{eq:case2-key}
    \Sigma_{\nu:1}' = - \frac{a}{1-\nu} + \frac{\Sigma_{\nu :1}}{1-\nu} .    
\end{align}

Using the strict concavity of $\varphi$ in combination with~$a> \Sigma_{\nu:1}$
(a strict mixture with positive weight on $b,c<a$), we obtain
\begin{equation}
\label{eq:case2-concavity}
\varphi(a)-\varphi(\Sigma_{\nu:1})<\varphi'(\Sigma_{\nu:1})(a-\Sigma_{\nu:1}).
\end{equation}
Inserting \eqref{eq:case2-concavity} into \eqref{eq:case2-derivative} and using
\eqref{eq:case2-key}, we obtain
\begin{align}
J_\infty'(\nu)
< \varphi'(\Sigma_{\nu:1})(a-\Sigma_{\nu:1}) +\varphi'(\Sigma_{\nu:1})(1-\nu) \Sigma_{\nu:1}' =0.
\end{align}
Thus $J_\infty'(\nu)<0$ for all $\nu\in(0,\theta)$, and in particular there
cannot be a local minimizer in $(0,\theta)$. \smallskip

Case 3: For~$\nu \in (\theta, \nu_\star)$, we will show below that
\begin{equation}\label{eq:Jpp-middle}
J_\infty''(\nu)<0,
\end{equation}
so $J_\infty$ is strictly concave on $(\theta,\nu_\star)$ and therefore
has at most one critical point there; if it exists, it is necessarily a strict local maximum. In particular, there is no local minimizer inside $(\theta,\nu_\star)$. Therefore let us move to the verification of~\eqref{eq:Jpp-middle}. For~$\nu \in (\theta, \nu_\star)$
we have
\[
\Sigma_{0:\nu}
= b+\frac{\theta(a-b)}{\nu},
\qquad
\Sigma_{\nu:1}
= c+\frac{(\nu_\star-\nu)(b-c)}{1-\nu}.
\]
Recalling $\varphi(x)=\log(x+\varepsilon)$, we can write
\[
J_\infty(\nu)=\nu\,\varphi(\Sigma_{0:\nu})+(1-\nu)\,\varphi(\Sigma_{\nu:1}),
\qquad \nu\in(\theta,\nu_\star).
\]
A direct differentiation gives
\[
J_\infty'(\nu)
=\varphi(\Sigma_{0:\nu})+\nu \varphi'(\Sigma_{0:\nu})\,\Sigma_{0:\nu}'
-\varphi(\Sigma_{\nu:1})+(1-\nu)\varphi'(\Sigma_{\nu:1})\,\Sigma_{\nu:1}'.
\]
Differentiating once more,
\begin{align*}
J_\infty''(\nu)
&= 2\varphi'(\Sigma_{0:\nu})\,\Sigma_{0:\nu}'
   +\nu\Big(\varphi''(\Sigma_{0:\nu})(\Sigma_{0:\nu}')^2+\varphi'(\Sigma_{0:\nu})\,\Sigma_{0:\nu}''\Big) \\
&\quad -2\varphi'(\Sigma_{\nu:1})\,\Sigma_{\nu:1}'
   +(1-\nu)\Big(\varphi''(\Sigma_{\nu:1})(\Sigma_{\nu:1}')^2+\varphi'(\Sigma_{\nu:1})\,\Sigma_{\nu:1}''\Big).
\end{align*}
Direct computation shows
\begin{align}
& \Sigma_{0:\nu}'= -\frac{\theta(a-b)}{\nu^2},\quad     
\Sigma_{0:\nu}''= \frac{2\theta(a-b)}{\nu^3}, \mbox{ and} \\
& \Sigma_{\nu:1}'= -\frac{(1-\nu_\star)(b-c)}{(1-\nu)^2},\quad
\Sigma_{\nu:1}''= -\frac{2(1-\nu_\star)(b-c)}{(1-\nu)^3}.
\end{align}
Hence we obtain the cancellations
\[
2\Sigma_{0:\nu}'+\nu \Sigma_{0:\nu}'' =0, \quad \mbox{and} \quad -2\Sigma_{\nu:1}'+(1-\nu)\Sigma_{\nu:1}''
= 0.
\]
This allows us to simplify~$J_\infty''$ as the $\varphi'$-terms drop out and we are left with
\[
J_\infty''(\nu)
= \nu\,\varphi''(\Sigma_{0:\nu})(\Sigma_{0:\nu}')^2
  + (1-\nu)\,\varphi''(\Sigma_{\nu:1})(\Sigma_{\nu:1}')^2.
\]
Finally, since $\varphi(x)=\log(x+\varepsilon)$ satisfies
\[
\varphi''(x)=-(x+\varepsilon)^{-2}<0,
\]
and since $a-b=r>0$, the identity
$\Sigma_{0:\nu}'=-\theta(a-b)/\nu^2$ shows that
$\Sigma_{0:\nu}'\neq0$. Together with
$\nu\in(\theta,\nu_\star)\subset(0,1)$, this makes the first summand in the
preceding display strictly negative, and hence we get
\[
J_\infty''(\nu)<0\qquad\text{for all }\nu\in(\theta,\nu_\star),
\]
which is \eqref{eq:Jpp-middle}. \smallskip

Overall, the only candidates for a local minimizer of~$J_{\infty}$ on~$[0,1]$ are the kink points $\theta$ and $\nu_\star$: by Case~2 the objective is strictly decreasing on~$(0,\theta)$, by Case~3 strictly concave on~$(\theta,\nu_\star)$, and by Case~1 strictly increasing on~$(\nu_\star,1)$. Consequently the global minimum over~$[0,1]$ is attained at $\theta$ or at $\nu_\star$. To identify which, it remains to show that
\begin{align}\label{equ:optimal_energy_gap}
    J_\infty(\nu_\star)<J_\infty(\theta).    \smallskip
\end{align}

Verification of~\eqref{equ:optimal_energy_gap}: The argument shares similarities to the calculations in the proof of Theorem~\ref{thm:classical_QML_consistency} and Theorem~\ref{thm:QML-robust}. At the two points $\nu=\theta$ and $\nu=\nu_\star$ we have
\begin{align}
    & \Sigma_{0:\theta}=a,\quad 
\Sigma_{\theta:1}=\frac{(\nu_\star-\theta)b+(1-\nu_\star)c}{1-\theta}, \quad \mbox{and} \\
    & \Sigma_{0:\nu_\star}=\frac{\theta a+(\nu_\star-\theta)b}{\nu_\star},
\quad
\Sigma_{\nu_\star:1}=c.    
\end{align}
To simplify the computation, we introduce the weights
\[
\lambda:=\frac{\theta}{\nu_\star}=2\theta\in(0,1),
\qquad
\mu:=\frac{1-\nu_\star}{1-\theta}=\frac{1}{2(1-\theta)}\in(0,1).
\]
Then, direct calculation gives
\[
\Sigma_{0:\nu_\star}=(1-\lambda)b+\lambda a, 
\quad \mbox{and} \quad 
\Sigma_{\theta:1}=(1-\mu)b+\mu c
\]
Hence
\begin{align}
J_\infty(\nu_\star)
&=\nu_\star\,\varphi\!\big((1-\lambda)b+\lambda a\big)+(1-\nu_\star)\,\varphi(c),\\
J_\infty(\theta)
&=\theta\,\varphi(a)+(1-\theta)\,\varphi\!\big((1-\mu)b+\mu c\big).
\end{align}
Using $\theta=\lambda\nu_\star$ and $1-\nu_\star=\mu(1-\theta)$, we can regroup the difference as 
\begin{align}
J_\infty(\nu_\star)-J_\infty(\theta) &=\nu_\star\Big(\varphi\!\big((1-\lambda)b+\lambda a\big)-\lambda\varphi(a)\Big)\\
&\qquad -(1-\theta)\Big(\varphi\!\big((1-\mu)b+\mu c\big)-\mu\varphi(c)\Big) \\
&=\nu_\star\Big(\varphi\!\big((1-\lambda)b+\lambda a\big)-(1-\lambda)\varphi(b)-\lambda\varphi(a)\Big)\label{eq:diff-gap}\\
&\qquad -(1-\theta)\Big(\varphi\!\big((1-\mu)b+\mu c\big)-(1-\mu)\varphi(b)-\mu\varphi(c)\Big).
\end{align}
where we added and subtracted the~$\varphi(b)$ term in the last step, observing that by definition $\nu_\star (1- \lambda) = (1- \theta) (1- \mu)$. Identifying positive scalars with $1\times1$ positive-definite matrices and recalling the Jensen gap~$G$ from Lemma~\ref{lem:jensen_gap_logdet}, equation~\eqref{eq:diff-gap} becomes
\[
J_\infty(\nu_\star)-J_\infty(\theta)
=\nu_\star\,G(b+\varepsilon,a+\varepsilon;\lambda)
-(1-\theta)\,G(b+\varepsilon,c+\varepsilon;\mu).
\]

For the first Jensen gap, the interpolation lies between~$b+\varepsilon$ and~$a+\varepsilon$. Since~$a>b$, we may take $\mu_{\min}=b+\varepsilon$. The upper bound in Lemma~\ref{lem:jensen_gap_logdet} therefore gives
\begin{equation}\label{eq:gap-upper-simple}
G(b+\varepsilon,a+\varepsilon;\lambda)
\leq
\frac{\lambda(1-\lambda)}{2(b+\varepsilon)^2}(a-b)^2
=
\frac{\lambda(1-\lambda)}{2(b+\varepsilon)^2}r^2.
\end{equation}

For the second Jensen gap, the interpolation lies between~$c+\varepsilon$ and~$b+\varepsilon$. Since~$b>c$, we may take $\mu_{\max}=b+\varepsilon$. The lower bound in Lemma~\ref{lem:jensen_gap_logdet} gives
\begin{equation}\label{eq:gap-lower-simple}
G(b+\varepsilon,c+\varepsilon;\mu)
\geq
\frac{\mu(1-\mu)}{2(b+\varepsilon)^2}(b-c)^2
=
\frac{\mu(1-\mu)}{2(b+\varepsilon)^2}K^2.
\end{equation}

Combining \eqref{eq:diff-gap}, \eqref{eq:gap-upper-simple}, and \eqref{eq:gap-lower-simple} yields
\begin{align*}
J_\infty(\nu_\star)-J_\infty(\theta)
&\le \frac{1}{2(b+\varepsilon)^2}\Big(\nu_\star\lambda(1-\lambda)r^2-(1-\theta)\mu(1-\mu)K^2\Big).
\end{align*}

Now we plug in $\nu_\star=\tfrac12$, $\lambda=2\theta$, and $\mu=\frac{1}{2(1-\theta)}$ and obtain
\[
\nu_\star\lambda(1-\lambda)=\frac12 (2\theta)(1-2\theta)=\theta(1-2\theta),
\qquad
(1-\theta)\mu(1-\mu)
=\frac{1-2\theta}{4(1-\theta)}.
\]
Hence
\begin{align}
J_\infty(\nu_\star)-J_\infty(\theta)
&\le \frac{1-2\theta}{2(b+\varepsilon)^2}
\left(\theta r^2-\frac{K^2}{4(1-\theta)}\right).\label{eq:final-diff-bound}
\end{align}
Since $\theta\in(0,\tfrac12)$ we have $1-2\theta>0$. Therefore \eqref{eq:final-diff-bound} is strictly negative as soon as
\[
\theta r^2<\frac{K^2}{4(1-\theta)}
\qquad\Longleftrightarrow\qquad
\frac{r}{K}<\frac{1}{2\sqrt{\theta(1-\theta)}}.
\]
This holds due to~\eqref{equ:counterexample_conditions}, and we conclude
\[
J_\infty(\nu_\star)-J_\infty(\theta)<0,
\qquad\text{i.e.}\qquad
J_\infty(\nu_\star)<J_\infty(\theta),
\]
which proves \eqref{equ:counterexample_energy}.
Together with the monotonicity and concavity properties established above --- $J_\infty$ is strictly decreasing on~$(0,\theta)$, strictly increasing on~$(\nu_\star,1)$, and strictly concave on~$(\theta,\nu_\star)$, so that on the middle interval it exceeds its smaller endpoint value~$J_\infty(\nu_\star)$ --- this shows that~$J_\infty$ has its unique minimizer at~$\nu_\star=0.5$. \qed

\subsubsection{Proof of Proposition~\ref{prop:projected_slln_factor}}\label{sec:proof_qml_consistency_factor}

The proof has three steps: we first show that the error and cross terms are
negligible on the pervasive scale; we then use the order-$p$ eigengap and the
Davis--Kahan theorem to control the empirical projection; finally, we
transfer the regime separation through that projection to identify
$\Delta_{\mathrm{within}}$ and $\Delta_{\mathrm{between}}$.

Let us recall the notation of Section~\ref{sec:abstract_QML} and some basic facts. First, we want to mention that~$p=p_n$ is a sequence in~$n$, and we suppressed the dependency on~$n$ for convenience. Since
$y_t=\tfrac1{\sqrt p}\widehat U^{\top}x_t$, we have for the macroscopic interval
$I=I(a,b)$ by definition
\[
\widehat\Sigma_{a:b}
=\tfrac1p\,\widehat U^{\top}
\Bigl(\tfrac1{|I|}\sum_{t\in I}x_tx_t^{\top}\Bigr)
\widehat U.
\]
We recall the decomposition $x_t=s_t+e_t$ into signal~$s_t$ and noise~$e_t$. For a macroscopic interval $I=I(a,b)$ we have defined
\[
S_{a:b}=\tfrac1{|I|}\sum_{t\in I}s_ts_t^{\top},\qquad
E_{a:b}=\tfrac1{|I|}\sum_{t\in I}e_te_t^{\top},\qquad
C_{a:b}=\tfrac1{|I|}\sum_{t\in I}\bigl(s_te_t^{\top}+e_ts_t^{\top}\bigr).
\]
This yields the decomposition
\begin{align}
\widehat\Sigma_{a:b}
=\tfrac1p\widehat U^{\top}(S_{a:b}+C_{a:b}+E_{a:b})\widehat U.
\end{align}
We use throughout the contraction $\|\widehat U^{\top} A\widehat U\|\le\|A\|$, since~$\widehat U$ has orthonormal columns. We also use the operator Cauchy--Schwarz inequality
\[
\bigl\|\tfrac1{|I|}\sum_{t\in I}s_te_t^{\top}\bigr\|
\le
\|S_{a:b}\|^{1/2}\|E_{a:b}\|^{1/2},
\]
which gives the estimate
\begin{align}\label{equ:op_cauchy_schwarz}
\|C_{a:b}\|\le 2\|S_{a:b}\|^{1/2}\|E_{a:b}\|^{1/2}.
\end{align}

\smallskip

We continue with the review of notation and fix a macroscopic interval $I=I(a,b)$ in the regime $\bullet$. For~$t \in I(a,b)$, we have $s_t=\Lambda_\bullet f_t$. Therefore,
\[
S_{a:b}=\Lambda_\bullet G_{a:b}\Lambda_\bullet^{\top},
\qquad
G_{a:b}:=\tfrac1{|I|}\sum_{t\in I}f_tf_t^{\top}.
\]
We also have
\[
\Sigma_\bullet^{\rm sig}:=\Lambda_\bullet\Sigma_{F_\bullet}\Lambda_\bullet^{\top}
\]
by definition.

\smallskip

\noindent\emph{Operator-norm estimates.}
We start by providing auxiliary estimates on operator norms. By pervasiveness~(F1),
\begin{align}
\label{eq:estimate_lambda_bullet_pervasiveness}
\tfrac1p\|\Lambda_\bullet\|^2=O(1).
\end{align}
We also have, by operator-norm submultiplicativity and~\eqref{eq:estimate_lambda_bullet_pervasiveness},
\begin{align}\label{equ:signal_op_bounds_sigma_bullet}
\tfrac1p\|\Sigma_\bullet^{\rm sig}\|
\le
\tfrac1p\|\Lambda_\bullet\|^2\,\|\Sigma_{F_\bullet}\|
=O(1).
\end{align}

We now deduce that
\begin{align}\label{equ:signal_op_bounds_S_a_b}
\tfrac1p\|S_{a:b}\|=O(1).
\end{align}
Indeed, for intervals contained in the regime $\bullet$,
\begin{align}\label{equ:signal_block_estimate}
\tfrac1p\bigl\|S_{a:b}-\Sigma_\bullet^{\rm sig}\bigr\|
& =
\tfrac1p\bigl\|\Lambda_\bullet(G_{a:b}-\Sigma_{F_\bullet})\Lambda_\bullet^{\top}\bigr\| \\
& \le
\Bigl(\tfrac1p\|\Lambda_\bullet\|^2\Bigr)
\bigl\|G_{a:b}-\Sigma_{F_\bullet}\bigr\|
\xrightarrow[n\to\infty]{}0,
\end{align}
by a combination of~\eqref{eq:estimate_lambda_bullet_pervasiveness} and Assumption~(F2). If an interval crosses the breakpoint, the same $O(1)$ bound~\eqref{equ:signal_op_bounds_S_a_b} follows by decomposing it into its left- and right-regime parts.

Allowing from now on arbitrary fixed macroscopic intervals~$I(a,b)$, we observe that Assumption~(F4) controls the noise, i.e.
\begin{align}
\label{eq:HDHSS_error_control}
\tfrac1p\|E_{a:b}\|\to0.
\end{align}
 Additionally, the cross terms vanish by~\eqref{equ:op_cauchy_schwarz} and~\eqref{eq:HDHSS_error_control}:
\begin{align}\label{equ:cross_estimate}
\tfrac1p\|C_{a:b}\|
\le
2\Bigl(\tfrac1p\|S_{a:b}\|\Bigr)^{1/2}
\Bigl(\tfrac1p\|E_{a:b}\|\Bigr)^{1/2}
\xrightarrow[n\to\infty]{}0.
\end{align}

\smallskip

\noindent\emph{Eigengap and consistency of the projection.}
Writing
\[
\Lambda:=[\,\Lambda_\le\ \Lambda_>\,]
\]
and
\[
\bar D:=\operatorname{diag}\bigl(\theta\Sigma_{F_\le},(1-\theta)\Sigma_{F_>}\bigr)\succ0,
\]
the global signal moment is
\begin{align}
\Sigma^{\rm sig}
:=
\theta\Sigma_\le^{\rm sig}+(1-\theta)\Sigma_>^{\rm sig}
=
\Lambda\bar D\Lambda^{\top}.
\end{align}
The non-zero eigenvalues of~$\frac{1}{p}\Sigma^{\rm sig}$ coincide with those of
\[
\bar D^{1/2}\bigl(\tfrac1p\Lambda^{\top}\Lambda\bigr)\bar D^{1/2}.
\]
By Assumption~(F1),
\[
\bar D^{1/2}\bigl(\tfrac1p\Lambda^{\top}\Lambda\bigr)\bar D^{1/2}
\longrightarrow
\bar D^{1/2}\Sigma_\Lambda\bar D^{1/2}.
\]
Since $\bar D\succ0$ and $\operatorname{rank}(\Sigma_\Lambda)=q$, the limit
$\bar D^{1/2}\Sigma_\Lambda\bar D^{1/2}$ has rank~$q$. Hence the leading \(q\) eigenvalues of \(\frac{1}{p}\Sigma^{\rm sig}\) are bounded away from \(0\) in the limit whereas the rest are \(0\). Therefore, for $n$ large, there is an eigengap of order~$p$ at level~$q$, i.e.
\begin{align}\label{equ:eigengap_signal}
\lambda_q(\Sigma^{\rm sig})-\lambda_{q+1}(\Sigma^{\rm sig})
\ge
\gamma_0\,p
\qquad\text{for some }\gamma_0>0 .
\end{align}

 Let $P_{\rm sig}$ denote the orthogonal projection onto the leading $q$-dimensional eigenspace of $\Sigma^{\rm sig}$, and let $\widehat P:=\widehat U\widehat U^{\top}$ denote the orthogonal projection onto $\Span(\widehat U)$, where $\widehat U$ contains the leading $q$ eigenvectors of the full-data second moment
\[
\tfrac1n\sum_{t=1}^n x_tx_t^{\top}=S_{0:1}+C_{0:1}+E_{0:1}.
\]
Set $w_n:=\tau/n$.  Lemma~\ref{lem:block_additivity}, applied to the signal vectors $s_t\in\mathbb R^p$, and the definition of~$\Sigma^{\rm sig}$ give
the exact weight-mismatch decomposition
\begin{align*}
S_{0:1}-\Sigma^{\rm sig}
&=w_n\bigl(S_{0:\theta}-\Sigma_\le^{\rm sig}\bigr)
+(1-w_n)\bigl(S_{\theta:1}-\Sigma_>^{\rm sig}\bigr)\\
&\quad +(w_n-\theta)
\bigl(\Sigma_\le^{\rm sig}-\Sigma_>^{\rm sig}\bigr),\\
\tfrac1p\bigl\|S_{0:1}-\Sigma^{\rm sig}\bigr\|
&\leq w_n\tfrac1p\bigl\|S_{0:\theta}-\Sigma_\le^{\rm sig}\bigr\|
+(1-w_n)\tfrac1p\bigl\|S_{\theta:1}-\Sigma_>^{\rm sig}\bigr\|\\
&\quad +|w_n-\theta|\tfrac1p
\bigl\|\Sigma_\le^{\rm sig}-\Sigma_>^{\rm sig}\bigr\|
\xrightarrow[n\to\infty]{}0.
\end{align*}
Here we used~\eqref{equ:signal_block_estimate}, $w_n\to\theta$, and the
$O(1)$ signal bounds~\eqref{equ:signal_op_bounds_sigma_bullet}. Together with
\eqref{eq:HDHSS_error_control} and~\eqref{equ:cross_estimate}, this gives
\begin{align}
& \tfrac1p\bigl\|(S_{0:1}+C_{0:1}+E_{0:1})-\Sigma^{\rm sig}\bigr\|\\
& \qquad \le
\tfrac1p\bigl\|S_{0:1}-\Sigma^{\rm sig}\bigr\|
+\tfrac1p\|C_{0:1}\|
+\tfrac1p\|E_{0:1}\|
\xrightarrow[n\to\infty]{}0 .
\end{align}

Let $\widehat P$ and $P$ denote the orthogonal projections onto the leading eigenspaces, of dimension $q$, of symmetric matrices $\widehat\Sigma$ and $\Sigma$, and let $\Theta$ denote the diagonal matrix of the principal angles between their ranges. If $\gamma:=\lambda_q(\Sigma)-\lambda_{q+1}(\Sigma)>0$, then the Davis--Kahan theorem in the form of Yu, Wang, and Samworth~\cite[Theorem~2]{YuWangSamworth2015} (a variant of the classical Davis--Kahan theorem~\cite{davis1970rotation}) gives
\[
\|\widehat P-P\|
=\|\sin\Theta\|
\leq\|\sin\Theta\|_F
\leq\frac{2\sqrt q}{\gamma}\,\|\widehat\Sigma-\Sigma\|.
\]
With the gap~\eqref{equ:eigengap_signal} this yields
\begin{align}\label{equ:DK_projection}
\bigl\|\widehat P-P_{\rm sig}\bigr\|
\le
\frac{2\sqrt{q}}{\gamma_0}
\cdot
\tfrac1p
\bigl\|(S_{0:1}+C_{0:1}+E_{0:1})-\Sigma^{\rm sig}\bigr\|
\xrightarrow[n\to\infty]{}0.
\end{align}

We shall also use the following consequence of pervasiveness. For every fixed matrix
$A\in\mathbb R^{(q_\le+q_>)\times(q_\le+q_>)}$,
\begin{align}\label{equ:pervasive_projection_estimate}
\tfrac1p
\bigl\|
\Lambda A\Lambda^{\top}
-
P_{\rm sig}\Lambda A\Lambda^{\top}P_{\rm sig}
\bigr\|_F
\xrightarrow[n\to\infty]{}0 .
\end{align}
Indeed, since $P_{\rm sig}$ is the projection onto the leading left singular space of
$\Lambda\bar D^{1/2}$, we have
\begin{align}
\bigl\|(I-P_{\rm sig})\Lambda\bar D^{1/2}\bigr\|_F^2
=
\sum_{j>q}\lambda_j(\Sigma^{\rm sig})
=
o(p),
\end{align}
where we used that only finitely many eigenvalues are involved and that
$\lambda_{q+1}(\Sigma^{\rm sig})=o(p)$. Since $\bar D^{-1/2}$ is bounded, this implies
\begin{align}\label{equ:Lambda_tail_small}
\bigl\|(I-P_{\rm sig})\Lambda\bigr\|_F
\le
\bigl\|(I-P_{\rm sig})\Lambda\bar D^{1/2}\bigr\|_F
\bigl\|\bar D^{-1/2}\bigr\|
=
o(\sqrt p).
\end{align}
Moreover, by~(F1),
\[
\|\Lambda\|=O(\sqrt p).
\]
Therefore,
\begin{align}
\bigl\|(I-P_{\rm sig})\Lambda A\Lambda^{\top}\bigr\|_F
&\le
\bigl\|(I-P_{\rm sig})\Lambda\bigr\|_F
\|A\|
\|\Lambda\|
=o(p),\\
\bigl\|\Lambda A\Lambda^{\top}(I-P_{\rm sig})\bigr\|_F
&\le
\|\Lambda\|
\|A\|
\bigl\|\Lambda^{\top}(I-P_{\rm sig})\bigr\|_F
=o(p).
\end{align}
Since
\[
\Lambda A\Lambda^{\top}
-
P_{\rm sig}\Lambda A\Lambda^{\top}P_{\rm sig}
=
(I-P_{\rm sig})\Lambda A\Lambda^{\top}
+
P_{\rm sig}\Lambda A\Lambda^{\top}(I-P_{\rm sig}),
\]
we obtain~\eqref{equ:pervasive_projection_estimate}.

\medskip

\noindent\emph{Argument for (i).}
By the contraction and the estimates~\eqref{equ:signal_op_bounds_S_a_b},~\eqref{eq:HDHSS_error_control}, and~\eqref{equ:cross_estimate}, every macroscopic block satisfies
\begin{align}
\|\widehat\Sigma_{a:b}\|
&=
\tfrac1p
\bigl\|\widehat U^{\top}(S_{a:b}+C_{a:b}+E_{a:b})\widehat U\bigr\|\\
& \le
\tfrac1p\|S_{a:b}\|
+\tfrac1p\|C_{a:b}\|
+\tfrac1p\|E_{a:b}\|
=O(1).
\end{align}
Set
\[
C_0:=\max_{\bullet\in\{\leq,>\}}
\limsup_{n\to\infty}\tfrac1p\|\Sigma_\bullet^{\rm sig}\|<\infty,
\]
where finiteness follows from~\eqref{equ:signal_op_bounds_sigma_bullet}.
Every break-anchored block lies in one of the two regimes. For each fixed
$I(a,b)\in\mathcal I_\theta^{\rm br}$ in regime~$\bullet$,
\eqref{equ:signal_block_estimate},~\eqref{eq:HDHSS_error_control}, and
\eqref{equ:cross_estimate} give
\[
\limsup_{n\to\infty}\|\widehat\Sigma_{a:b}\|
\leq
\limsup_{n\to\infty}\tfrac1p\|\Sigma_\bullet^{\rm sig}\|
\leq C_0.
\]
Since $C_0$ is independent of the particular block,
\[
\sup_{I(a,b)\in\mathcal I_\theta^{\rm br}}
\limsup_{n\to\infty}\|\widehat\Sigma_{a:b}\|
\leq C_0,
\]
which is precisely $M<\infty$, and hence Assumption~(A1).

\medskip

\noindent\emph{Argument for (ii).}
We fix two macroscopic intervals $I_1,I_2$. Using the identity
\[
\widehat\Sigma_{I}
=
\tfrac1p\widehat U^{\top}(S_I+C_I+E_I)\widehat U
\]
for $I=I_1$ and $I=I_2$, together with~\eqref{eq:HDHSS_error_control} and~\eqref{equ:cross_estimate}, we get
\begin{align}\label{equ:diff_decomp}
\| \widehat\Sigma_{I_1}-\widehat\Sigma_{I_2} \|
& \le
\tfrac1p \| S_{I_1}-S_{I_2} \| \\
& \qquad
+
\underbrace{
\tfrac1p\bigl(\|C_{I_1}\|+\|C_{I_2}\|+\|E_{I_1}\|+\|E_{I_2}\|\bigr)
}_{\xrightarrow[n\to\infty]{}0}.
\end{align}
Since the projected matrices are $q\times q$ and $q$ is fixed, the same estimates imply the corresponding convergence in Frobenius norm.

If $I_1$ and $I_2$ lie in the same regime $\bullet$, then by~\eqref{equ:signal_block_estimate} both
\[
\tfrac1p\|S_{I_j}-\Sigma_\bullet^{\rm sig}\|\to0,
\qquad j=1,2.
\]
Hence
\[
\tfrac1p\|S_{I_1}-S_{I_2}\|\to0,
\]
and therefore
\[
\|\widehat\Sigma_{I_1}-\widehat\Sigma_{I_2}\|_F\to0
\]
for each such pair. As $\Delta_{\textup{within}}$ takes the limit in $n$ before the supremum over pairs,
\begin{align}
\Delta_{\textup{within}}
=
\sup_{\text{same regime}}\
\limsup_{n\to\infty}
\|\widehat\Sigma_{I_1}-\widehat\Sigma_{I_2}\|_F
=
0 .
\end{align}

Now assume that $I_1\subset[0,\theta]$ and $I_2\subset[\theta,1]$, and set
\[
\Xi:=\Sigma_\le^{\rm sig}-\Sigma_>^{\rm sig}.
\]
By~\eqref{equ:signal_block_estimate},
\[
\tfrac1p\|(S_{I_1}-S_{I_2})-\Xi\|\to0.
\]
Therefore, in Frobenius norm,
\begin{align}\label{equ:between_projected_Xi}
\widehat\Sigma_{I_1}-\widehat\Sigma_{I_2}
=
\tfrac1p\widehat U^{\top}\Xi\widehat U+o(1).
\end{align}
Since $\widehat U$ has orthonormal columns,
\begin{align}
  \label{equ:projection_coordinates_frobenius_norm_equivalence}
  \|\widehat U^{\top}\Xi\widehat U\|_F
  =
  \|\widehat P\Xi\widehat P\|_F.
\end{align}
We now compare $\widehat P\Xi\widehat P$ with $\Xi$. First observe that
\[
\Xi=\Lambda A_\Xi\Lambda^{\top},
\qquad
A_\Xi:=
\operatorname{diag}\bigl(\Sigma_{F_\le},-\Sigma_{F_>}\bigr).
\]
Thus, applying~\eqref{equ:pervasive_projection_estimate} with $A=A_\Xi$ gives
\begin{align}\label{equ:Xi_projection_asymptotic}
\tfrac1p\|\Xi-P_{\rm sig}\Xi P_{\rm sig}\|_F\to0.
\end{align}
Moreover,
\[
\tfrac1p\|\Xi\|_F=O(1),
\]
because $\Xi$ has rank at most $q_\le+q_>$ and
$\tfrac1p\|\Xi\|=O(1)$ by~\eqref{equ:signal_op_bounds_sigma_bullet}. Hence, using~\eqref{equ:DK_projection},
\begin{align}
\tfrac1p
\bigl\|\widehat P\Xi\widehat P-P_{\rm sig}\Xi P_{\rm sig}\bigr\|_F
&\le
\tfrac1p
\bigl\|(\widehat P-P_{\rm sig})\Xi\widehat P\bigr\|_F
+
\tfrac1p
\bigl\|P_{\rm sig}\Xi(\widehat P-P_{\rm sig})\bigr\|_F\\
&\le
2\Bigl(\tfrac1p\|\Xi\|_F\Bigr)
\|\widehat P-P_{\rm sig}\|
\xrightarrow[n\to\infty]{}0.
\end{align}
Combining this estimate with~\eqref{equ:Xi_projection_asymptotic}, we obtain
\begin{align}
\tfrac1p
\bigl|\,
\|\widehat P\Xi\widehat P\|_F-\|\Xi\|_F
\,\bigr|
&\le
\tfrac1p
\|\widehat P\Xi\widehat P-\Xi\|_F
\xrightarrow[n\to\infty]{}0.
\end{align}
A combination of the last statement,~\eqref{equ:projection_coordinates_frobenius_norm_equivalence}, and~\eqref{equ:between_projected_Xi} yields that for every between-regime pair
\begin{align}
\|\widehat\Sigma_{I_1}-\widehat\Sigma_{I_2}\|_F
=
\tfrac1p
\bigl\|
\Sigma_\le^{\rm sig}-\Sigma_>^{\rm sig}
\bigr\|_F
+o(1).
\end{align}
The leading term is independent of the chosen between-regime pair. Taking the infimum over between-regime pairs and then $\liminf_n$ gives
\begin{align}
\Delta_{\textup{between}}
=
\liminf_{n\to\infty}
\tfrac1p
\bigl\|
\Sigma_\le^{\rm sig}-\Sigma_>^{\rm sig}
\bigr\|_F,
\end{align}
which is positive by~(F3). This proves Proposition~\ref{prop:projected_slln_factor}. \qed

\section*{Acknowledgments}
We thank Alec Kercheval and Mihai Cucuringu for valuable discussions on the topics of this work. We acknowledge the use of the AI assistants Claude (Anthropic, via the Claude Code command-line interface) and ChatGPT (OpenAI) for assistance with editing the manuscript and with developing the code for the numerical examples; the authors reviewed all such output and take full responsibility for the content of this paper. We thank the UCLA Department of Mathematics for making these tools available, and the Olga Radko Endowed Math Circle at UCLA for financial support.

\appendix

\section{Code}\label{sec:code}

The code that generates the figures in this paper is included with the manuscript source as ancillary files (see the \texttt{anc/} directory of the arXiv submission). The equity data for Figure~\ref{fig:gfc_smoothness} are daily adjusted close prices from Yahoo Finance, retrieved with the \texttt{yfinance} package; the ticker universe is defined inside the generating script, and names with insufficient price coverage over 2005--2010 are dropped by the filter documented there, leaving $p=200$ names. Table~\ref{tab:code_figure_map} lists, for each figure, the source that generates it.

\begin{table}[h]
\centering
\begin{tabular}{l l}
\hline
Figure & Source \\
\hline
Figure~\ref{fig:intro_qml_demo} & \texttt{anc/intro\_qml\_demo\_fig.py} \\
Figure~\ref{fig:intro_inhomogeneous_lln} & \texttt{anc/intro\_inhomogeneous\_lln\_fig.py} \\
Figure~\ref{fig:boundary_ridge} & \texttt{anc/boundary\_ridge\_fig.py} \\
Figure~\ref{fig:gfc_smoothness} & \texttt{anc/gfc\_smoothness\_fig.py} \\
Figure~\ref{fig:illustration_optimality} & drawn in pgfplots directly in the \TeX{} source \\
\hline
\end{tabular}
\caption{Correspondence between the figures of this paper and the ancillary scripts that generate them.}
\label{tab:code_figure_map}
\end{table}

\bibliographystyle{abbrv}

\bibliography{references.bib}

\end{document}